\documentclass[10pt]{article}

\usepackage{tikz}
\usepackage{graphicx}  
\usepackage{float}
\usepackage{amsmath,amssymb,amsfonts, theorem,makeidx,latexsym,epsfig, subfigure}
\allowdisplaybreaks
\usepackage{color}

\usetikzlibrary{positioning, arrows.meta, shapes}
\usepackage{xcolor}
\definecolor{Darkbrown}{RGB}{101, 67, 33}
\definecolor{verydarkbrown}{RGB}{60, 30, 10}
\definecolor{darkpurple}{RGB}{75,0,130}
\definecolor{slategray}{RGB}{70,90,110}
\definecolor{darkmagenta}{RGB}{139,0,139}
\definecolor{tealish}{RGB}{0,128,128}

\newtheorem{defn}{Definition}[section]

{\theorembodyfont{\rmfamily}

	\newtheorem{ex}[defn]{Example}}

\newtheorem{thm}[defn]{Theorem}

\newtheorem{prop}[defn]{Proposition}

\newtheorem{cor}[defn]{Corollary}

\newtheorem{lem}[defn]{Lemma}

\newtheorem{rem}[defn]{Remark}

\numberwithin{equation}{section}

\def\bp{{\noindent\bf Proof. \ }}

\def\ep{\hfill$\square$\par\bigskip}

\def\bee{\begin{eqnarray}}

	\def\ene{\end{eqnarray}}

\def\bes{\begin{eqnarray*}}

	\def\ens{\end{eqnarray*}}

\def\bei{\begin{itemize}}

	\def\eni{\end{itemize}}

\def\bt{\begin{thm}}

	\def\et{\end{thm}}

\def\bc{\begin{cor}}

	\def\ec{\end{cor}}

\def\bpr{\begin{prop}}

	\def\epr{\end{prop}}

\def\bl{\begin{lem}}

	\def\el{\end{lem}}

\def\bd{\begin{defn}}

	\def\ed{\end{defn}}

\def\bex{\begin{ex}}

	\def\enx{\end{ex}}

\def\bfi{\begin{fig}}

	\def\efi{\end{fig}}

\def\br{\begin{rem}}

	\def\er{\end{rem}}

\usepackage[
    colorlinks=false,      
    linkbordercolor={1 0 0},  
    citebordercolor={0 1 0},  
    urlbordercolor={0 0 1}    
]{hyperref}

	\newcommand{\norm}[1]{\left\Vert#1\right\Vert}

	\newcommand{\R}{\mathbb R}

	\newcommand{\Z}{\mathbb Z}
	
	\newcommand{\ltg}{L^2(G)}

	\newcommand{\hg}{\widehat{G}}
    
    \newcommand{\J}{\mathcal{J}}

    \def\al{\alpha}

\def\la{\lambda}

\def\e{\eta}

\def\z{\zeta}

\newcommand{\Th}{\Theta}

\newcommand{\pj}{P_j}

\newcommand{\spj}{\sum_{P_j}}
\newcommand{\mupj}{\mu_{P_j}}

\newcommand{\newgtione}{\displaystyle\bigcup_{j \in \mathcal{J}}\{T_{\la}g_p^{(1)}\}_{\la \in \eta \Gamma_{j}, p \in P_j}}

\newcommand{\newgtitwo}{\displaystyle\bigcup_{j \in \mathcal{J}}\{T_{\la}g_p^{(2)}\}_{\la \in \zeta\Gamma_{j}, p \in P_j}}

\newcommand{\wj}{w_{f;g_p^{(1)}, g_p^{(2)},j}}
\newcommand{\newwj}{w_{f;\e g_p^{(1)}, \z g_p^{(2)},j}}
\newcommand{\w}{w_{f;g_p^{(1)}, g_p^{(2)}}}
\newcommand{\neww}{w_{f;\e g_p^{(1)},  \z g_p^{(2)}}}
\newcommand{\wphij}{w_{f;\Phi_{j}}}

\newcommand{\wphiJ}{w_{f;\Phi_{J+1}}}

\newcommand{\wphiJminusone}{w_{f;\Phi_{J}}}

\newcommand{\wphijnote}{w_{f;\Phi_{j_0}}}

\newcommand{\wphijnoteplusone}{w_{f;\Phi_{j_0+1}}}

\newcommand{\wphijnoteplustwo}{w_{f;\Phi_{j_0+2}}}

\newcommand{\wphiJoneplusone}{w_{f;\Phi_{J_1+1}}}

\newcommand{\wphiJtwo}{w_{f;\Phi_{J_2}}}

\newcommand{\gpi}{g_p^{(i)}}

\newcommand{\gpone}{g_p^{(1)}}

\newcommand{\gptwo}{g_p^{(2)}}

\newcommand{\wgponej}{w_{f;\gpone,j}}

\newcommand{\wgpij}{w_{f;\gpi,j}}

\newcommand{\wgponejnote}{w_{f;\gpi,j_0}}

\newcommand{\wgpi}{w_{f;\gpi}}

\newcommand{\wgpone}{w_{f;\gpone}}

\newcommand{\dwphij}{\displaystyle\int\limits_{\Gamma_{j}} \left| \langle T_x f, T_\la \mathcal{F}^{-1}\Phi_{j} \rangle \right|^2 \ d\mu_{\Gamma_{j}}(\la) }

\newcommand{\dwgponejnote}{\int\limits_{P_{j_0}}\int\limits_{\Gamma_{j_0}} \left| \langle T_x f, T_\la \gpi \rangle \right|^2 \ d\mu_{\Gamma_{j_0}}(\la) d\mu_{P_{j_0}}(p)}

\newcommand{\M}{\mathcal{M}}
\newcommand{\Mg}{\M_{\la}}

\newcommand{\F}{\mathcal{F}}

\newcommand{\Fdwgponejnote}{\int\limits_{P_{j_0}}\int\limits_{\Gamma_{j_0}} \left| \langle \F(T_x f), \F(T_\la \gpi) \rangle \right|^2 \ d\mu_{\Gamma_{j_0}}(\la) d\mu_{P_{j_0}}(p)}

\newcommand{\FEdwgponejnote}{\int\limits_{P_{j_0}}\int\limits_{\Gamma_{j_0}} \left| \langle \F(T_x f), \Mg \widehat{\gpi} \rangle \right|^2 \ d\mu_{\Gamma_{j_0}}(\la) d\mu_{P_{j_0}}(p)}

\newcommand{\ta}{t_\alpha}

\newcommand{\dtna}{\displaystyle\sum_{j \in \mathcal{J} : \alpha \in \Gamma^\perp_j} \frac{1}{s(\Gamma_{j})} \int\limits_{P_j}  \overline{\widehat{g_{p}^{(1)}}(\gamma)} \widehat{g_{p}^{(2)}}(\gamma+\alpha)  d\mu_{P_j}(p)}

\newcommand{\dtza}{\displaystyle\sum_{j \in \mathcal{J} } \frac{1}{s(\Gamma_{j})} \int\limits_{P_j}  \overline{\widehat{g_{p}^{(1)}}(\gamma)} \widehat{g_{p}^{(2)}}(\gamma)  d\mu_{P_j}(p)}

\newcommand{\Gj}{\Gamma_{j}}

\newcommand{\Gjp}{\Gamma_{j}^\perp}

\newcommand{\dnewwj}{\int\limits_{P_j}\int\limits_{\Gamma_{j}} \langle T_x f, T_{\e\la} g^{(1)}_p \rangle \langle  T_{\z\la} g^{(2)}_p, T_x f \rangle \ d\mu_{\Gamma_{j}}(\la) d\mu_{P_j}(p)}

\newcommand{\dnewtna}{\displaystyle\sum_{j \in \mathcal{J} : \alpha \in \Gamma^\perp_j} \frac{1}{s(\Gamma_{j})} \int\limits_{P_j}  \overline{\widehat{g_{p}^{(1)}}(\eta^\ast\gamma)} \widehat{g_{p}^{(2)}}(\zeta^\ast(\gamma+\alpha))  d\mu_{P_j}(p)}

\newcommand{\dnewtza}{\displaystyle\sum_{j \in \mathcal{J} } \frac{1}{s(\Gamma_{j})} \int\limits_{P_j}  \overline{\widehat{g_{p}^{(1)}}(\eta^\ast\gamma)} \widehat{g_{p}^{(2)}}(\zeta^\ast\gamma)  d\mu_{P_j}(p)}

\newcommand{\tpj}{\widetilde{\Phi_{j}}}

\newcommand{\tpjplusone}{\widetilde{\Phi_{j+1}}}

\newcommand{\tgmij}{\widetilde{G^{(i)(m)}_j}}

\newcommand{\tgmijplusone}{\widetilde{G^{(i)(m)}_{j+1}}}

\newcommand{\thj}{\widetilde{H_j}}	

\newcommand{\thjplusone}{\widetilde{H_{j+1}}}

\newcommand{\tsij}{\widetilde{\Psi^{(i)(m)}_j}}

\newcommand{\gtione}{\displaystyle\bigcup_{j \in \mathcal{J}}\{T_{\lambda}g_p^{(1)}\}_{\la \in \Gamma_{j}, p \in P_j}}

\newcommand{\gtitwo}{\displaystyle\bigcup_{j \in \mathcal{J}}\{T_{\lambda}g_p^{(2)}\}_{\la \in \Gamma_{j}, p \in P_j}}

\newcommand{\g}{\gamma}

\newcommand{\G}{\Gamma}

\newcommand{\union}{\bigcup\limits_{j \in \mathcal{J}} \Gamma_{j}^\perp}

\newcommand{\umzero}{\bigcup\limits_{j \in \mathcal{J}} \Gamma_j^\perp\setminus \{0\}}

\newcommand{\gtii}{\displaystyle\bigcup_{j \in \mathcal{J}}\{T_{\lambda}g_p^{(i)}\}_{\gamma \in \Gamma_{j}, p \in P_j}}

\newcommand{\gtiiz}{\displaystyle\bigcup_{j \in \Z}\{T_{\lambda}g_p^{(i)}\}_{\la \in \Gamma_{j}, p \in P_j}}

\def\dx{\{\eta_j \la :   \la \in \Gamma_j\}_{j \in \J}}

\def\dy{\{\zeta_j \la :   \la \in \Gamma_j\}_{j \in \J}}

\def\tx{\mathcal{T}_{\mathcal{X}}}
\def\x{\mathcal{X}}
\def\y{\mathcal{Y}}
\newcommand{\ve}{V_{E}}

\def\ty{\mathcal{T}_{\mathcal{Y}}}

\def\ops{\bigoplus\limits_{j\in \J} L^2(\Gamma_j)}

\def\G{\Gamma}
\def\uniontsi{\bigcup_{j \in \mathcal{J}} \{T_\la \psi_{E}\}_{\la \in \eta_j\G_j}}
\def\uniontsbi{\bigcup_{j \in \mathcal{J}} \{T_\la \psi_{F}\}_{\la \in \zeta_j\G_j}}

\title{Characterization and Construction of Pairwise Orthogonal Parseval Frames with Applications to Sampling}
	
	\date{\today}

	\author{}

\begin{document}
			\author{Navneet Redhu, Anupam Gumber, Hartmut Führ, Niraj K. Shukla}
		
		\maketitle

		\begin{abstract}
			In this paper, we provide a characterization of pairwise orthogonal frames with generalized translation-invariant (GTI) structures, based on the unconditional convergence property (UCP). These GTI systems are generated by translating functions over a countable family of closed, co-compact subgroups of a locally compact abelian (LCA) group $G$, where the families of subgroups associated with each system may differ. As an application of this characterization, we establish necessary and sufficient criteria for the orthogonality of various structured systems, including Gabor, wavelet, and shearlet systems on LCA groups. Furthermore, we derive a characterization of  GTI Parseval (tight) frames and present explicit constructions of pairs of GTI systems using filters. Each constructed system satisfies the $\infty$-UCP and admits a Calder\'on sum equal to one. As a consequence of these results, the constructed systems form Parseval frames and are pairwise orthogonal. The proposed construction improves upon the technique in \cite{RGS} by relaxing the stationary assumption on the families of subgroups. Finally, we illustrate the results with examples using $B$-splines as generating functions and discuss applications of pairwise orthogonal frames in sampling theory.
		\end{abstract}
		
			\begin{minipage}{145mm}	
				\end{minipage}
		
		\section{Introduction}
		Orthogonal frames, introduced independently by Han and Larson \cite{HL} and  Balan \cite{B2, B2000}, have become fundamental tools in various areas such as signal processing, computer graphics, and engineering. In multiple-access communication, pairwise orthogonal Parseval frames $\mathbb{F}=\{f_m\}_{m \in I}$ and $\mathbb{G}=\{g_m\}_{m \in I}$ enable the recovery of signals $f, g \in \mathcal{H}$ from combined coefficients of the form $\langle f, f_m \rangle + \langle g, g_m \rangle$. This multiplexing principle underlies a variety of communication systems, including radio and television, and has been widely studied in \cite{B2000, BDV2000, AW}. Beyond communication theory, orthogonality plays an important role in frame theory itself. In particular, given a frame $\mathbb{F}$ with canonical dual $S^{-1}\mathbb{F}$, and another frame $\mathbb{H}$ that is orthogonal to $\mathbb{F}$, the system $S^{-1}\mathbb{F} + \mathbb{H}$ provides a non-canonical dual of $\mathbb{F}$. Such constructions are not only theoretically significant but can also result in reduced reconstruction error as compared to canonical duals \cite{BLPY}. Moreover, when two frames in Hilbert spaces are orthogonal, their direct sum naturally forms a superframe for the combined space \cite{B2000}. In sampling theory as well, orthogonal frames guarantee perfect reconstruction even when some frame elements lie outside the target subspace \cite{A, W, W2004}.

The concept of pairwise orthogonal frames has attracted significant attention due to its theoretical relevance and practical importance, particularly within structured systems such as wavelet, Gabor, and shift-invariant frames (see, e.g., \cite{AW, BG, GS2018, GS2019, KLS, KLS1, LY, Ri, SN, SNSCM, Tang, W}). In most of the existing literature, orthogonality is studied for families of structured frames that share a common translation lattice or, more generally, the same co-compact subgroup. However, several applications in sampling theory naturally require orthogonality between structured systems associated with different translation lattices or sequences of co-compact subgroups \cite{W2004}. Such situations arise, for example, when analyzing the orthogonal ranges of sampling operators, which play a central role in stable reconstruction, signal denoising \cite{benedetto-treiber-2001}, and multiple-access communication systems \cite{ALTW2004}. Motivated by these considerations, Aldroubi et al.\ studied the orthogonality of pairs of Gabor and wavelet frames associated with distinct lattice translations in $L^2(\R^d)$ \cite{ALTW2004}. Related results for generalized shift-invariant (GSI) systems with different translation lattices were obtained by Weber in \cite{W}.

The generalized translation-invariant (GTI) systems provide a unified method for analyzing a broad category of function systems, including discrete and continuous systems of wavelets, Gabor, wave-packets, and shearlets \cite{GS2018, JL2016}. The GTI system in the setting of the locally compact abelian (LCA) groups was introduced by Jakobsen and Lemvig in \cite{JL2016}. This system arises by applying translations of the generating functions along a countable collection of closed, co-compact subgroups within the second-countable LCA group $G$; see Section~\ref{sec_preliminaries} for further details.   In view of this unifying nature of GTI systems, it is natural to study properties that are common and fundamental across different structured systems within a single framework. In particular, motivated by applications in sampling theory that require orthogonality between structured systems associated with different translation subgroups, this paper investigates the orthogonality properties of GTI systems for which the underlying families of co-compact subgroups may differ. This setting allows us to simultaneously treat a broad range of structured frames, such as wavelet and Gabor systems, generated by distinct translation structures. More precisely, we consider pairs of GTI systems of the form
\begin{align}
		\bigcup_{j \in \mathcal{J}} \{T_{\lambda} g^{(1)}_p\}_{\la \in \eta_{j}\Gamma_j, \, p \in P_j}
		\qquad \text{and} \qquad
		\bigcup_{j \in \mathcal{J}} \{T_{\lambda} g^{(2)}_p\}_{\la \in \zeta_{j}\Gamma_j, \, p \in P_j},\label{eq_p3_intro_1}
	\end{align}
where $\mathcal{J}$ is a countable index set, $g_j^{(1)}, g_j^{(2)} \in L^2(G)$, $\eta_j$ and $\zeta_j$ are bi-continuous automorphisms of $G$, $P_j$ is countable or uncountable index set, and $\Gamma_j$ is a co-compact subgroup of $G$.   We note that, when considering a single GTI system, the simultaneous use of a family
of automorphisms $(\eta_j)_j$ together with subgroups $(\Gamma_j)_j$ is not strictly
necessary, since the resulting system depends only on the family $(\eta_j(\Gamma_j))_j$. In such a case, one could equivalently absorb the automorphisms into the subgroups by defining $\Gamma'_j=\eta_j\Gamma_j$ and $\eta'_j=\mathrm{id}$. However, automorphisms become essential when studying relationships between several GTI systems, as in \eqref{eq_p3_intro_1}, where both systems are naturally formulated relative to the same family of subgroups $(\Gamma_j)_j$. This viewpoint is particularly important when analyzing orthogonality of coefficients, which is defined with respect to integration over the subgroups $\Gamma_j$ (see Definition~\ref{defn_pairwise_orthogonal} for exact details). The advantages of including automorphisms in the formulation~\eqref{eq_p3_intro_1}, especially from the perspective of sampling applications, are further illustrated in Section~5.

Motivated by the applications discussed above, we address the following problems:
\begin{itemize}
\item[\textbf{(P1).}] Find the necessary and sufficient conditions under which the GTI Bessel (frame) systems in \eqref{eq_p3_intro_1} form pairwise orthogonal Bessel (frame) systems in $L^2(G)$ (in the sense of Definition \ref{defn_pairwise_orthogonal}).
\item[\textbf{(P2).}] Provide explicit constructions of pairwise orthogonal Parseval GTI frames as defined in \eqref{eq_p3_intro_1}.
\end{itemize}

 Gumber and Shukla addressed Problem~(P1) in \cite{GS2018} under the restrictive assumption that $\e_j = \z_j = \mathrm{id}$ (the identity automorphism) for each $j \in \mathcal{J}$, assuming the local integrability condition (LIC). Building on this result, Redhu, Gumber, and Shukla \cite{RGS} proposed a construction method for pairwise orthogonal GTI systems of the form \eqref{eq_p3_intro_1}, also under the same restrictive assumption $\e_j = \z_j = \mathrm{id}$. To the best of our knowledge, extending these results to nontrivial automorphisms $\e_j$ and $\z_j$ under weaker assumptions than the LIC remains an open problem.

The construction method presented in this work is based on periodic filters, also known as matrix mask functions. Filters have long served as a fundamental tool in the construction of frames \cite{BJM, CBR, CHL, CHS, DH, Bin1, HanLu, San}. In a wide range of applications, including signal processing, image processing, and data analysis, filter-based frame constructions play a significant role. As an illustration, one may apply filters to isolate distinctive features from datasets, which in turn enhances tasks in pattern recognition, including fingerprint analysis and facial identification (see, e.g., \cite{C, RTCAW} and related references). Beyond this, in other contexts, such as the construction of a multiresolution analysis (MRA) system, where a signal is represented at different levels of detail, filters play a fundamental role. Such representations form the basis for important applications, including signal compression, noise reduction, and feature extraction. In the MRA setting associated with filters, several extension principles have been introduced to obtain Parseval frames, including the unitary extension principle (UEP) \cite{CG1, CG2021, RS1997} and the oblique extension principle (OEP) \cite{DHRS}.
At the end, we show that pairwise orthogonal GTI frames constructed under this framework possess significant applications in sampling theory.

The main results of this paper can be described as follows:
\begin{itemize}
    \item[(i)] Theorems~\ref{thm_new_characterization_orthogonal_frame}--\ref{thm_char_Parseval_frames} provide rather far-reaching characterizations of various properties of GTI Bessel systems and frames, including orthogonality and tightness, under suitable UCP assumptions.
    \item[(ii)] Direct applications of Theorems \ref{thm_new_characterization_orthogonal_frame}--\ref{thm_characterization_orthogonal_frame} in various structured systems, including wavelet and time-frequency contexts, are formulated in Subsection~\ref{sec_application_to_characterization_result}.
    \item[(iii)] A general construction of GTI systems using filters, together with sufficient conditions ensuring that they satisfy the $\infty$-UCP and form Parseval frames with Calderón sum equal to 1, is formulated in Theorem~\ref{thm_UCP}.
    \item[(iv)] As an application of Theorem~\ref{thm_UCP}, a novel \textit{oblique extension principle for GTI systems} is proved (Theorem~\ref{thm_QEP}). Moreover, under the conditions of Theorem~\ref{thm_UCP} and additional assumptions on the filters, an application of Theorem~\ref{thm_characterization_orthogonal_frame} yields a construction of pairwise orthogonal Parseval frames (Theorem~\ref{thm_pairwise_orthogonal_frames}), which ultimately leads to the construction of $N$ pairwise orthogonal Parseval frames from a given one (Theorems~\ref{thm_construction_of_N_pair_OF}).
    \item[(v)] We prove novel characterizations of orthogonal Shannon sampling (Theorem~\ref{thm_pairwise_orthogonal_samples}). The results in Section~\ref{sec 5} highlight the usefulness of including the automorphisms $(\eta_j)_{j \in J}$ and $(\zeta_j)_{j \in J}$ in the discussion.
\end{itemize}
While our explicit results mostly stay in the realm of Euclidean Fourier, wavelet, and time-frequency analysis, the general setup developed in Section 3 provides a unified view of a variety of settings studied in the literature. We mention the following implications:
  \bei
	\item[(i)] Our characterization result (Theorems \ref{thm_new_characterization_orthogonal_frame}) for pairwise orthogonal frames extends several known results, including the Euclidean case~\cite{W}, the discrete setting~\cite{LH},  the uniform lattice case~\cite{KLS}, and GTI systems in $\ltg$  \cite{GS2018}.  
		\item[(ii)] Because structured constructions, including Gabor, shearlet, wavelet, and wave-packet frames, can be regarded as particular instances of GTI systems~\cite{BCZ, GS2018, GS20192, HLW, JL2016, RS1, RS}, the present findings offer a unified approach for analyzing orthogonal frames within all of these frameworks.
		\item[(iii)] Many important groups, such as $\mathbb{Z}^n$, the torus $\mathbb{T}^n$, and the $p$-adic field $\mathbb{Q}_p$, are examples of LCA groups. Our results provide a unified approach to studying pairwise orthogonal frames on these groups as well.
		\item[(iv)] Our construction method for pairwise orthogonal Parseval frames generalizes several existing approaches, including well-known constructions of orthogonal Parseval wavelet frames in $L^2(\mathbb{R})$~\cite{BJW}, $L^2(\mathbb{R}^n)$~\cite{BG} and $L^2(G)$~\cite{RGS} (see, Section~\ref{sec4} for details).
    \eni

This paper is organized as follows:
	
	Section~\ref{sec_preliminaries} presents the necessary preliminaries and background material for the remainder of the paper.

	In Section~\ref{sec_charactrizations_results}, we first recall the dual $1$-UCP and discuss its properties (see Subsection~\ref{sub_def_UCP}). We then present our main characterization result for pairwise orthogonal GTI Bessel (frame) systems satisfying the dual $1$-UCP (see Theorem~\ref{thm_new_characterization_orthogonal_frame}). This result generalizes \cite[Theorem 3.5]{GS2018} in two ways. First, it applies to the GTI systems associated with different sequences of co-compact subgroups. Second, it holds under the dual $1$-UCP, which is weaker than the dual $\alpha$-LIC condition used in \cite[Theorem 3.5]{GS2018}.  Additionally, Theorem~\ref{thm_char_Parseval_frames} characterizes when a GTI system forms a tight frame. This theorem extends two earlier results: \cite[Theorem 3.12]{FL2019}, proved for GSI systems, and \cite[Theorem 3.5]{JL2016}, proved for GTI systems under the stronger assumption of $\alpha$-LIC. In the final subsection, Subsection~\ref{sec_application_to_characterization_result}, we apply Theorem~\ref{thm_characterization_orthogonal_frame} to derive characterization results for pairwise orthogonal Bessel systems with various structures, including Gabor and wavelet systems.

Section~\ref{sec_charactrizations_results} lays the groundwork for the construction recipes presented in Section~\ref{sec4}. As an initial step, we introduce two GTI systems, $\gtione$ and $\gtitwo$, generated by filters (see Subsection~\ref{subsec_4.1}). Next, in Subsection~\ref{subsec_sufficent_condtions_for_UCP}, Theorem~\ref{thm_UCP} demonstrates that the constructed GTI systems satisfy the $\infty$-UCP and form Parseval frames with Calder\'on sum equal to~$1$. In this setting, the index set $\mathcal{J} \subseteq \mathbb{Z}$ is taken to be one of the following: $\mathbb{Z}$, $\{j\}_{j=j_0}^{\infty}$, $\{j\}_{j=j_0}^{j_1}$, or $\{j\}_{j=-\infty}^{j_0}$, for some $j_0, j_1 \in \mathbb{Z}$. As a special case, when $\mathcal{J}=\{j\}_{j=j_0}^{\infty}$, the GTI system defined in Theorem~\ref{thm_UCP} takes the form
\[
\left\{T_{\lambda}\mathcal{F}^{-1}\Phi_{j_0}\right\}_{\lambda \in \Gamma_{j_0}}
\;\cup\;
\bigcup_{j = j_0}^{\infty} \left\{T_{\lambda}g_p^{(i)}\right\}_{\lambda \in \Gamma_{j},\, p \in P_j},
\]
where $\mathcal{F}^{-1}$ stands for the inverse Fourier transform (see Remark~\ref{rem_comparision with_CG_UEP}). This form of GTI system, which is a Parseval frame, was previously established by Christensen and Goh via the UEP for LCA groups \cite{CG2021}. In the present work, we allow greater flexibility in the choice of the index set and additionally observe that the $\infty$-UCP holds and that the Calder\'on sum equals~$1$.
In Example~\ref{ex_B_spline}, we explicitly construct a pair of GTI systems in $\ell^2(\mathbb{Z})$, each satisfying the $\infty$-UCP. Furthermore, we establish the OEP in Theorem~\ref{thm_QEP}, which provides a more flexible variant of the UEP. Finally, in Subsection~\ref{subsec_pairwise_orthogonal}, we provide constructions of pairwise orthogonal Parseval frames and present a recipe for constructing $N$-Parseval frames that are pairwise orthogonal, starting from a given one.

	Finally, in Section \ref{sec 5}, we discuss applications of orthogonal frames in sampling theory. In particular, Proposition \ref{prop_ness_suff_conditions_for_pos_UCP} provides necessary and sufficient conditions characterizing the orthogonality of two families of samples associated with the given sampling sets under the $1$-UCP assumption. Theorem \ref{thm_pairwise_orthogonal_samples} then establishes that the unions of sampling sets are orthogonal on specified bands if and only if each corresponding individual pair of samples is orthogonal under the dual $1$-UCP assumption. Moreover, Theorem \ref{thm_tight_sample} characterizes the tightness of the union of sampling sets, showing that it holds if and only if each individual sampling set is tight.


        \section{Preliminaries} \label{sec_preliminaries}
	In this section, we establish the foundational concepts and notations necessary for our study. We begin with key definitions and a brief overview of locally compact abelian (LCA) groups, which form the underlying setting for our analysis.
	
	Let $G$ denote a LCA group that is second countable,  equipped with the group operation $+$ and identity element $0$. A homomorphism from $ G \to \mathbb{T},$  is called  \textit{character}, where $\mathbb{T}=\{ z \in \mathbb{C}  : |z|=1  \}$. We denote by $\widehat{G}$ the set of all continuous characters, referred to as the \textit{dual group} of $G$. Equipped with the appropriate topology and the group operation defined by
	\begin{align*}
		(\gamma +\gamma') (x):= \gamma(x)\cdot \gamma'(x) \ \mbox{ for all } \gamma, \ \gamma' \in \widehat{G}, x \in G.
	\end{align*}
    the dual group $\widehat{G}$ is itself a LCA group.
Let $\mu_G$ denote the Haar measure on $G$. For $1 \leq p < \infty$, the space $L^p(G)$ is defined in the standard manner with respect to $\mu_G$. In particular, $L^2(G)$ forms a Hilbert space when equipped with the inner product
	$$\langle f,g \rangle = \int_G f(x) \overline{g(x)} d\mu_{G}(x).$$
    
	The \textit{Fourier transform} $\mathcal{F} : L^1(G) \to C_0(\widehat{G})$ of a function $f \in L^1(G)$ is given by
    $$ \mathcal{F}f(\gamma):=\int_G f(x)\gamma(-x) d\mu_{G}(x),$$
	where $-x$ denotes the inverse of $x$ in the group $G$.
	We use the notation $\widehat{f}$ to represent the Fourier transform of $f$, that is, $\widehat{f}=\mathcal{F}f$.
	A subgroup $\Gamma$ of $G$ is called a \textit{closed co-compact subgroup} of $G$ if $\Gamma$ is closed in $G$ and the quotient space $G/\Gamma$ is compact. Whenever $\Gamma$ is a discrete subgroup of $G$, it is called a (uniform) \textit{lattice} in $G$. The \textit{annihilator} of $\Gamma$ is the subset $\Gamma^\perp=\{\gamma\in \widehat{G} : \gamma(x)=1, \, \forall \, x \in \Gamma \}$. 
	Let $\Omega$ be a \textit{fundamental domain} corresponding to $\Gamma^\perp$. That is, $\Omega$ is a Borel measurable subset of $\widehat{G}$ satisfying
	\begin{equation}\label{eq_1_intro}
		\widehat{G}= \bigcup_{w \in \Gamma^\perp} (w+ \Omega), \ (w+ \Omega) \cap (w'+ \Omega)= \emptyset \mbox{ for }w \neq w', \ w, \ w' \in \Gamma^\perp.
	\end{equation}

       We denote by $\mu_{\Gamma}$ and $\mu_{G/\Gamma}$ the Haar measures on $\Gamma$ and $G/\Gamma$, respectively. Once two among the measures $\mu_G$, $\mu_{\Gamma}$, and $\mu_{G/\Gamma}$ are fixed, the third can be normalized in such a way that the following formula holds for all $f \in L^1(G)$:
	\begin{equation}\label{eq_Weil's_integration_formula}
		\displaystyle\int\limits_{G} f(x) d \mu_G(x)= \displaystyle\int\limits_{G/\Gamma} \displaystyle\int\limits_{\Gamma} f(x+\gamma) \ d \mu_{\Gamma}(\gamma) \ d\mu_{G/\Gamma}(\dot{x}),
	\end{equation}
	where $\dot{x}$ denotes the coset $x + \Gamma.$ We normalize the Haar measures $\mu_G$, $\mu_\Gamma$, and $\mu_{G/\Gamma}$ so that \textit{Weil's integral formula}~(\ref{eq_Weil's_integration_formula}) is true. With this choice, the induced dual measures on $\widehat{G}$, $\widehat{G/\Gamma}\cong \Gamma^\perp$, and $\widehat{\Gamma} \cong \widehat{G}/\Gamma^\perp$ satisfy the identity
	\begin{equation}\label{eq_weil_integral_formula_2}
		\int\limits_{\widehat{G}} \widehat{f}(\omega) d \mu_{\widehat{G}}(\omega)=\int\limits_{\widehat{G}/\Gamma^\perp} \int\limits_{\Gamma^\perp} \widehat{f}(\omega \xi) d\mu_{\Gamma^\perp}(\xi) d\mu_{\widehat{G}/\Gamma^\perp} (\dot{\omega}).
	\end{equation}
    We normalize Haar measures and their dual counterparts so that the Plancherel theorem holds. Consequently, the system of measures satisfies a uniqueness principle: once two non-dual measures are chosen from
$\mu_G,\ \mu_\Gamma,\ \mu_{G/\Gamma},\ \mu_{\widehat{G}},\ \mu_{\Gamma^\perp},\ \mu_{\widehat{G}/\Gamma^\perp},$
the remaining measures are uniquely fixed by enforcing Weil’s integral relations \eqref{eq_Weil's_integration_formula} and \eqref{eq_weil_integral_formula_2}. The \textit{covolume} of $\Gamma$ is defined by 
	$$s(\Gamma):=\int_{G/ \Gamma} \ d\mu_{G/ \Gamma}(\dot{x}).$$
	For further details regarding LCA groups, we refer to \cite{BHP1, BMR, COB, GJ, GS2018, HRJDS, HR, JL2016, KL, KK, R}.
		
	We now turn to generalized translation-invariant (GTI) systems. These systems over LCA groups were first studied by Jakobsen and Lemvig in \cite{JL2016}. The definition is given below.
	
	\bd\label{defn_GTI}
		Consider a countable index set $\mathcal{J}$. A \textit{generalized translation-invariant (GTI) system} is given by 
\bee\label{eq_defn_GTI}	\bigcup_{j \in \mathcal{J}} \{T_{\la} g_p : \la \in \Gamma_j, \, p \in P_j\},\ene
		where, for $\lambda \in G$, the \emph{translation operator} is defined by $T_\lambda : L^2(G) \to L^2(G), \ (T_\lambda f)(x) := f(x-\lambda), \ x \in G.$
		Here, each $P_j$ is a (countable or uncountable) index set, each $\Gamma_j$ is a closed co-compact subgroup of $G$, and $\{g_p\}_{p \in P_j} \subset L^2(G)$.
	\ed	
Translation-invariant (TI) systems correspond to the special case of GTI systems where the family of subgroups reduces to a single subgroup, that is, if $\Gj=\G$ for each $j\in\mathcal{J}$.  
    
	As in~\cite{JL2016}, the GTI system introduced in Definition~\ref{defn_GTI} is assumed to satisfy the following hypotheses.\\

\noindent\textbf{Standing assumptions:} 
We begin by fixing some notation. Take a countable subset $\J$ of $\mathbb{Z}$.  
For every $j$ in $\J$, associate a measure space $(\pj, \spj, \mupj)$.
If $X$ is a topological space, its Borel $\sigma$-algebra will be denoted by $\mathcal{B}_X$. 
For every $j \in \J$, the Cartesian product $\pj \times G$ is viewed as a measurable space equipped with the product $\sigma$-algebra $\spj \otimes \mathcal{B}_G$ and the corresponding product measure $\mupj \otimes \mu_G$.
Throughout the paper, we assume that for each $j \in \J$ the following properties are satisfied:
\begin{enumerate}
	\item[(I)]\label{eq_I} The measure space $(\pj, \spj, \mupj)$ is $\sigma$-finite.
	
	\item[(II)]\label{eq_II} The function $p \mapsto g_{j,p}$, defined from $(\pj, \spj)$ into $(L^2(G), \mathcal{B}_{L^2(G)})$, is measurable.
	
	\item[(III)]\label{eq_III} The map $(p,x) \mapsto g_{j,p}(x)$ from $(\pj \times G, \spj \otimes \mathcal{B}_G)$ into $(\mathbb{C}, \mathcal{B}_{\mathbb{C}})$ is measurable.
\end{enumerate}
In Sections~\ref{sec4} and \ref{sec 5}, we restrict ourselves to countable index sets $P_j$ endowed with the counting measure. Under this choice, the assumptions (I)–(III) are automatically satisfied (see \cite{JL2016}).	

		A GTI system (defined in \eqref{eq_defn_GTI})  is said to form a \textit{GTI frame} relative to the  $\{L^2(P_j \times \Gamma_j): j \in \J\}$ if there exist two constants $0<A, B < \infty$ such that  
\bee\label{eq_defn_frame}
A\norm{f}^2 \leq \sum_{j \in \J} \int\limits_{P_j} \int\limits_{\Gamma_j} 
|\langle f, T_{ \la} g_p \rangle|^2 \,
d\mu_{\Gamma_j}(\la)\, d\mu_{P_j}(p) \leq B \norm{f}^2 
\quad \text{for all } f \in \ltg.
\ene  
The constants $A$ and $B$ are referred to as the \textit{frame bounds}. 
If $A$ and $B$ coincide, the frame is called \textit{tight}. 
In the particular case where $A = B = 1$, it is termed a 
\textit{GTI Parseval frame}.  If only the right-hand side inequality in \eqref{eq_defn_frame} is satisfied, then the system given in \eqref{eq_defn_GTI} is called a 
\textit{GTI Bessel system}.

	We study the orthogonality of two GTI Bessel systems of the form  
	\begin{equation}\label{eq_more_GTI_version}
	\newgtione
		\qquad \text{and} \qquad 
	\newgtitwo,
	\end{equation}
    where \(\eta\) and \( \zeta \) are bi-continuous group automorphisms of $G$. 
	Suppose both the GTI systems defined in \eqref{eq_more_GTI_version} are Bessel, then we define the \emph{mixed dual Gramian operator} $\Theta_{\e, \z}$ associated with these GTI systems, is given by
	\begin{align}\label{mixed_dual_grammian_operator}
		\Theta_{\e, \z}: \ltg \to \ltg; f \mapsto  \sum_{j \in \mathcal{J}} \int\limits_{P_j} \int\limits_{\Gamma_j} 
		\langle f, T_{\eta \la} g^{(1)}_p \rangle \, T_{\zeta \la} g^{(2)}_p \,
		d\mu_{\Gamma_j}(\la)\, d\mu_{P_j}(p).
	\end{align}
  In the special case $\e=\z=\mathrm{Id}$, we simply write $\Theta_{\e,\z}=\Theta$.
	\bd\label{defn_pairwise_orthogonal}
	Let $\newgtione$ and  
	$\newgtitwo$  
	be two GTI Bessel (or frame) systems. We say that these systems are \emph{pairwise orthogonal Bessel (frame) systems} if the associated mixed dual Gramian operator satisfies $\Theta = 0$ for all $f \in L^2(G)$. Furthermore, if the GTI systems are Parseval and pairwise orthogonal, then they are called \emph{pairwise orthogonal Parseval frames}.
	\ed		
In the classical setting $G=\mathbb{R}^d$, the orthogonality of shift-invariant systems with different lattice translations was studied by Weber~\cite{W}. More generally, the orthogonality of systems of the form~\eqref{eq_more_GTI_version} is closely related to the analysis of orthogonal sampling transforms; see, for instance,~\cite{W2004}. From an operator-theoretic perspective, such orthogonality is naturally characterized by the vanishing of the associated mixed dual Gramian operator, which guarantees that the ranges of the corresponding sampling operators are mutually orthogonal. This property is fundamental in applications where independent or non-interfering signal components are required, such as signal denoising~\cite{benedetto-treiber-2001} and multiple-access communication systems~\cite{ALTW2004}. In these contexts, orthogonality of the sampling ranges ensures stable reconstruction and effective separation of information. A detailed discussion of these connections is presented in Section~\ref{sec 5}.

\textit{In the remainder of this paper, $\operatorname{Aut}(G)$ denotes set of all bi-continuous group automorphisms of $G$, \(\eta, \zeta \in \operatorname{Aut}(G)\), and \(\Gamma_j^\perp \subseteq \widehat{G}\) is the annihilator of a co-compact subgroup \(\Gamma_j\). Moreover, we have \((\eta\Gamma_j )^\perp = \eta^\ast \Gamma_j^\perp\), where \(\eta^\ast := (\eta')^{-1}\) denotes the inverse of the adjoint of \(\eta\).}
\section{Characterization of Orthogonal GTI Systems}\label{sec_charactrizations_results}		
The main objective of this section is to establish the necessary and sufficient conditions under which GTI Bessel (or frame) systems, defined in \eqref{eq_more_GTI_version}, are pairwise orthogonal in the sense of Definition~\ref{defn_pairwise_orthogonal}. We also provide a characterization of when a GTI system forms a tight frame.

To accomplish these objectives, Subsection~\ref{sub_def_UCP} introduces two central concepts: the dual $1$-unconditional convergence property (dual $1$-UCP) and dual $\infty$-UCP for a pair of GTI systems. They play a key role in our main characterization of pairwise orthogonality.
	Section~\ref{sec_SCR} then presents the main characterization theorem (Theorem~\ref{thm_new_characterization_orthogonal_frame}), which identifies when two GTI systems, defined in \eqref{eq_more_GTI_version}, are pairwise orthogonal Bessel (or frame) systems. We also include Theorem~\ref{thm_char_Parseval_frames}, which characterizes a GTI tight frame. The technical details and proofs of these results appear in Section~\ref{sec_proof_of_characterization_result}. Finally, Section~\ref{sec_application_to_characterization_result} applies the main theorem to concrete cases, yielding characterization results for orthogonal frames built from Gabor, wavelet, composite wavelet, and shearlet systems.	

    \subsection{Unconditional convergence property}\label{sub_def_UCP}
In this subsection, we introduce UCP, which plays a central role in analyzing the frame properties of GTI systems.  Führ et al.~\cite{FL2019} introduced the UCP and its variants (such as the dual $1$-UCP and dual $\infty$-UCP), which provide a weaker and more flexible alternative to the classical LIC. These concepts were first developed for GSI systems, a subclass of GTI systems. The $1$-UCP condition for a GTI system was introduced in \cite{LVV}.
We begin by defining the set
$$ \mathcal{D}_{B}:= \{f \in L^2(G) :\widehat{f} \in L^{\infty}(\widehat{G})  \mbox{ and } \operatorname{supp}\widehat{f} \ \mbox{ is compact in } \widehat{G}\setminus B \},$$
where \( B \subset \widehat{G} \) is a Borel set with \( \mu_{\widehat{G}}(B) = 0 \), referred to as a \emph{blind spot}. The exclusion of the null set $B$ ensures that subsequent Fourier-domain identities hold almost everywhere without ambiguity. The space $\mathcal{D}_B$ is dense in $L^2(G)$, since compactly supported functions are dense in $L^2(G)$, and removing a null set does not affect density. Moreover, \( \mathcal{D}_{B} \) is translation invariant, because $\operatorname{supp} \widehat{T_x f}= \operatorname{supp} \widehat{f}$.

Given a pair of GTI systems $\newgtione$ and $\newgtitwo$, and for any \( f \in \mathcal{D}_B \), we define a family of functions $\newwj: G \to \mathbb{C}$ by 
\begin{equation}\label{eq_wfg1g2j}
	\newwj(x)=\dnewwj \mbox{ for } x \in G.
\end{equation}

      Assuming conditions (I)–(III) hold, the expressions in \eqref{eq_wfg1g2j} are well-posed and finite.  If the GTI systems $\newgtione$ and $\newgtitwo$ are Bessel, then using a similar argument as in \cite{JL2016, KL}, each $\newwj$ is a trigonometric polynomial, $\Gjp$-periodic, continuous, and bounded.
        
  The function $\neww: G \to \mathbb{C}$ is defined by
\begin{align}
	\neww(x)&=\sum_{j \in \J}\newwj(x) \nonumber\\
	&=\sum_{j \in \J} \dnewwj \label{eq_wfgi}
\end{align}
provided the series on the right-hand side converges. In the special case $\e=\z=\mathrm{Id}$, we simply write $\neww=\w$ and $\newwj=\wj$. Also if  $\gpone=\gptwo$ for all $p \in \pj, j \in \J$, we write $\w=\wgpone$ and $\wj=\wgponej$. 
\begin{rem} In view of \cite{LVV} and by following a similar argument to that in \cite[Lemma~3.3]{FL2019}, it is clear that
	\leavevmode 
	\begin{itemize}
		\item[(i)] Under standing assumptions (I)–(III),  $\sum\limits_{j \in \J}\wgponej$ converges in [0,$\infty$], ensuring that $\wgpone$ is well-defined,  although $\w$ may not be.
		\item[(ii)]  If each GTI system forms a Bessel family, then $\wj$ converges absolutely and uniformly on compact subsets of \( G \). 
	\end{itemize}
\end{rem}

We are now ready to introduce the notion of unconditional convergence, motivated by \cite{FL2019}.
\begin{defn}\label{defn_UCP}
	Let	$\newgtione$ and $\newgtitwo$  be two GTI systems in $\ltg$.
	\begin{itemize}
		\item[(i)] They satisfy the \textit{dual $1$-unconditional convergence property (dual 1-UCP)}  w.r.t. $B \in \mathcal{B}$, if for each $f \in \mathcal{D}_B(G)$,  $\neww$ is almost periodic and 
		\begin{equation}\label{eq_1UCP}
			\neww=\displaystyle\sum_{j \in \J} \newwj
		\end{equation}
		converges unconditionally w.r.t. the mean $M(|\cdot|)$. That is, for any $\epsilon > 0$, one can find a finite subset $\J' \subset \J$ such that for all finite sets $\J''$ containing $\J'$, 
		$$M\left(\left|\neww-\sum\limits_{j \in \J''}\newwj\right|\right)< \epsilon,$$
		where $M$ is the mean, and its exact definition can be found in \cite[Theorem 3.6]{FL2019}.
		\item[(ii)] They satisfy the  \textit{dual $\infty$-UCP} if the series in \eqref{eq_1UCP} converges uniformly on \( G \).
	\end{itemize}
	In case $\gpone=\gptwo$ for all $p \in \pj$, we refer to the above conditions collectively as the \textit{$\alpha$-UCP} for the system  $\newgtione$, where $\alpha \in \{1, \infty\}.$ 
\end{defn}
Here, $M(\cdot)$ denotes the invariant mean on the space of almost periodic functions. Convergence \lq\lq with respect to the mean\rq\rq\ is weaker than uniform convergence and does not require pointwise convergence. The dual $1$-UCP therefore provides a weaker alternative to uniform convergence assumptions such as the LIC.

Note that in the definition of the dual $1$-UCP, the function $\neww$ is assumed to be almost periodic, whereas in the dual $\infty$-UCP this follows automatically from uniform convergence.  We now describe the relationships among the dual $\infty$-UCP (respectively, $\infty$-UCP), the dual $1$-UCP (respectively, $1$-UCP), and the dual $\alpha$-LIC (respectively, $\alpha$-LIC) for GTI systems. For the definitions of the dual $\alpha$-LIC (respectively, $\alpha$-LIC), we refer to \cite{JL2016, GS2018}.

\begin{rem}\label{rem_relation_between_UCP_LIC}
 The following relationships hold for GTI systems:
	\begin{itemize} 
		 \item[(i)] If the dual $\alpha$-LIC holds for the given two GTI systems, then the dual $\infty$-UCP holds, and consequently, the dual $1$-UCP holds as well. The proof follows by adapting the argument in \cite[Remark 6]{FL2019}, replacing the GSI structure with the GTI setting.
		\item[(ii)] If the $\alpha$-LIC  holds for a given GTI system, then the $\infty$-UCP holds, and hence the $1$-UCP also holds for that GTI system \cite{LVV}.
	\end{itemize}
	The logical implications among these properties can be visualized as follows:\\~\\
    	
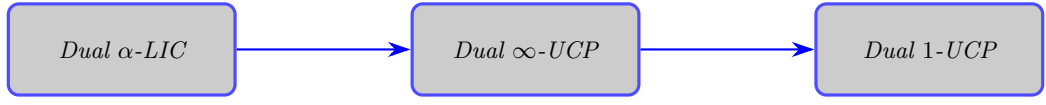
\centering	\begin{tikzpicture}[
		box/.style={
			rectangle, draw=blue!70, fill=black!20, very thick,
			minimum height=1.2cm, minimum width=3cm,
			text centered, rounded corners, font=\small
		},
		arr/.style={
			-{Stealth[length=3mm, width=2mm]}, thick, blue
		},
		node distance=2cm and 2.3cm
		]
		
		\node[box] (dualAlphaLIC) {Dual $\alpha$-LIC };
		\node[box, right=of dualAlphaLIC] (dualInfinityUCP) {Dual $\infty$-UCP };
		\node[box, right=of dualInfinityUCP] (dualOneUCP) {Dual $1$-UCP};
		
		
		\draw[arr] (dualAlphaLIC) -- (dualInfinityUCP);
		\draw[arr] (dualInfinityUCP) -- (dualOneUCP);
		
		%
	\end{tikzpicture}
\end{rem}		
        \begin{rem}\label{rem_dual_UCP}
	The Bessel property and the $\infty$-UCP are generally independent of each other; neither implies the other. However, by a straightforward adaptation of the argument in \cite[Lemma~3.9]{FL2019}, if one GTI system is Bessel and the other satisfies the $\infty$-UCP, then together they satisfy the dual $\infty$-UCP. The same holds for the $1$-UCP.
\end{rem}

For further details on the connection between Bessel systems and the 1-UCP (or $\infty$-UCP), we refer the reader to \cite{FL2019,LVV}. With the unconditional convergence property in place, we are now in a position to state our main result: a characterization of when two GTI Bessel systems are pairwise orthogonal under the dual \(1\)-UCP condition. The dual $1$-UCP ensures that the auxiliary function $\neww$ admits a well-defined Fourier expansion whose coefficients can be computed explicitly. This is the key ingredient in characterizing orthogonality.

	
		\subsection{Statement of Characterization result}\label{sec_SCR}
		In this section, we provide a characterization for two GTI Bessel (frame) systems, defined in \eqref{eq_more_GTI_version}, to be pairwise orthogonal under the  $1$-UCP condition. To the best of our knowledge, a necessary and sufficient Fourier-domain characterization for pairwise orthogonality of GTI systems under $1$-UCP associated with possibly different translation subgroups has not been previously established, even in the Euclidean setting $L^2(\mathbb{R}^n)$. In this context, Theorem~\ref{thm_new_characterization_orthogonal_frame} gives necessary and sufficient conditions for GTI Bessel (frame) systems in $\ltg$ to be pairwise orthogonal. We return to this result in Subsection~\ref{sec_application_to_characterization_result} and Section~\ref{sec 5}, where we discuss its further consequences and applications.

\begin{thm}\label{thm_new_characterization_orthogonal_frame}
	Let \( \newgtione \) and \( \newgtitwo \) be two GTI Bessel (frame) systems  in \( L^2(G) \) such that one of the system  is satisfying the  \(1\)-UCP. Then the following assertions are equivalent:
	\begin{itemize}
		\item[(i)] 	The systems \( \newgtione \) and \( \newgtitwo \) are pairwise orthogonal Bessel (frame).
		\item[(ii)] For every \( \alpha \in \umzero \), we have
		\begin{equation*}\label{eq_t_alpha}
			\ta(\gamma):=	\dnewtna =0  \mbox{ for a.e. } \gamma \in \hg,
		\end{equation*}
		and 
		\begin{equation*}
			t_0(\gamma):=\dnewtza	=0  \mbox{ for a.e. } \gamma \in \hg.
		\end{equation*} 
	\end{itemize}
\end{thm}
Recall that for each $\alpha \in \widehat{G}$, the function $t_\alpha$ denotes the Fourier coefficient of the auxiliary function $\neww$ corresponding to frequency $\alpha$.
In view of Remark \ref{rem_relation_between_UCP_LIC},  a similar result holds under the  $\infty$-UCP condition. To prove Theorem~\ref{thm_new_characterization_orthogonal_frame}, we first establish the following auxiliary result.



\begin{thm}\label{thm_characterization_orthogonal_frame}
	Let \( \gtione \) and \( \gtitwo \) be two GTI Bessel (frame) systems in \( L^2(G) \) satisfying the dual \(1\)-UCP. Then the following assertions are equivalent:
	\begin{itemize}
		\item[(i)] For every \( \alpha \in \umzero \), we have
		\begin{equation*}\label{eq_t_alpha}
			\ta(\gamma):=	\dtna =0  \mbox{ for a.e. } \gamma \in \hg,
		\end{equation*}
		and 
		\begin{equation*}
			t_0(\gamma):=\dtza	=0  \mbox{ for a.e. } \gamma \in \hg.
		\end{equation*}
		\item[(ii)] The systems \( \gtione \) and \( \gtitwo \) are pairwise orthogonal.
	\end{itemize}
\end{thm}
\noindent
The following result provides a necessary and sufficient condition under which a GTI system satisfying the $1$-UCP becomes a tight frame.
We further exploit this characterization in  Section~\ref{sec 5}, where concrete applications and sampling-theoretic consequences are derived.
\begin{thm}\label{thm_char_Parseval_frames}  
	Let the  GTI system $\gtione$  satisfy the  1-UCP. Then the following assertions are equivalent:
	\begin{itemize}
		\item[(i)] The above GTI system is a tight frame with bound $K$. 
		\item [(ii)] For each $\alpha \in  \union$,
		\begin{align*}
			\displaystyle\sum_{j \in \mathcal{J} : \alpha \in \Gamma^\perp_j} \frac{1}{s(\Gamma_{j})} \int\limits_{P_j}  \overline{\widehat{g_{p}^{(1)}}(\gamma)} \widehat{g_{p}^{(1)}}(\gamma+\alpha)  d\mu_{P_j}(p) =\delta_{\al,0}K \mbox{ for a.e. } \gamma \in \widehat{G}.
		\end{align*}
	\end{itemize}
\end{thm}
It is worth mentioning that if $K=1$, the above result provides a necessary and sufficient condition under which a GTI system is a Parseval frame.

Further, it can be observed that Theorem~\ref{thm_char_Parseval_frames} extends \cite[Theorem 3.12]{FL2019}, which provides characterization for Parseval frames in the special case $P_j = \{j\}$ and $\Gamma_j$ taken as lattices.  Moreover, our theorem generalizes \cite[Theorem 3.5]{JL2016}, where a similar Parseval frame characterization was obtained under the  $\alpha$-LIC.

		\subsection{Proof of Theorems~\ref{thm_new_characterization_orthogonal_frame}, \ref{thm_characterization_orthogonal_frame}, and \ref{thm_char_Parseval_frames}}\label{sec_proof_of_characterization_result}
For proving Theorems~\ref{thm_new_characterization_orthogonal_frame}, \ref{thm_characterization_orthogonal_frame}, and \ref{thm_char_Parseval_frames}, we first provide a result,  which identifies a key structural property of the mixed dual Gramian operator. Specifically, it shows that the vanishing of the components \( t_\alpha \), for \( \alpha  \in \umzero \), is equivalent to the mixed dual Gramian operator commuting with translations.  Moreover, it characterizes such operators as Fourier multipliers. 
	
	\begin{prop}\label{prop_Fourier_multiplier}
		Let \( \gtione \) and \( \gtitwo \) be two GTI Bessel systems satisfying the dual \(1\)-UCP, and let \( \Theta \), defined in \eqref{mixed_dual_grammian_operator}, be the mixed dual Gramian operator associated with these systems. Then the following statements are equivalent:
		\begin{itemize}
			\item[(i)] The operator \( \Theta \) commutes with the family of translation operators \( \{T_x\}_{x \in G} \), i.e., \( \Theta T_x = T_x \Theta \) for all \( x \in G \).
			
			\item[(ii)] For every $\alpha \in  \umzero$, the function
			\begin{align*}
				\ta(\gamma)=\dtna =0  \mbox{ for a.e. } \gamma \in \widehat{G}.
			\end{align*}
		\end{itemize}
		Furthermore, if either (i) or (ii) holds, then \( \Theta \) is a Fourier multiplier operator with symbol
		\[
		s(\gamma) = \sum_{j \in \mathcal{J}} \frac{1}{s(\Gamma_j)} \int\limits_{P_j} \overline{\widehat{g_p^{(1)}}(\gamma)}\, \widehat{g_p^{(2)}}(\gamma)\, d\mu_{P_j}(p),
		\]
		that is, \( \widehat{\Theta f}(\gamma) = s(\gamma)\, \widehat{f}(\gamma) \) for all \( f \in L^2(G) \), where the latter series converges absolutely for a.e.\ $\gamma$ and defines an element $s \in L^\infty(\widehat{G})$.
	\end{prop}
	
	 The above proposition may be viewed as a generalization of \cite[Lemma 1]{W2004}, which addresses the classical Euclidean case \( L^2(\mathbb{R}^d) \), and \cite[Proposition 3.7]{GS2018}, which is formulated for LCA groups. The following lemma will be used in the proof of Proposition~\ref{prop_Fourier_multiplier}.

	\begin{lem}\label{lem_Fourier_cofficents}
		Suppose the  GTI systems $\gtione$ and $\gtitwo$ satisfy the assumptions of Proposition~\ref{prop_Fourier_multiplier}. Then the operator $M_{\bar{t}_\alpha}: L^2(\widehat{G}) \to L^2(\widehat{G})$, which maps any $f$ to the product $f \cdot \bar{t}_\alpha$, is well-defined and bounded.  Moreover, $\forall$ $f \in \mathcal{D}_B$, the following identity holds:
		\begin{equation}\label{eq_using_t_alpha_0}
			\sum\limits_{j \in \mathcal{J}: \alpha \in \Gamma_j^\perp} d_{P_j,\alpha}=\left\langle \widehat{f}, M_{\bar{t}_\alpha} T_{-\alpha}\widehat{f} \right\rangle,
		\end{equation}
		where 
		\begin{equation}\label{eq_def_dpjalpha}
			d_{P_j, \alpha}=\frac{1}{s(\Gamma_{j})} \displaystyle\int\limits_{P_j} \displaystyle\int\limits_{\widehat{G}}\widehat{ f}(w) \overline{\widehat{f}(w+\alpha)} \overline{\widehat{g_p^{(1)}}(w)} \widehat{g_p^{(2)}}(w+\alpha)\ d\mu_{\widehat{G}}(w) d\mu_{P_j}(p).
		\end{equation}
	\end{lem}
	\bp
		We have
		$$\ta(\gamma)=\displaystyle\sum_{j \in \mathcal{J} : \alpha \in \Gamma^\perp_j} \frac{1}{s(\Gamma_{j})} \int\limits_{P_j}  \overline{\widehat{g_{p}^{(1)}}(\gamma)} \widehat{g_{p}^{(2)}}(\gamma+\alpha)  d\mu_{P_j}(p).$$
		Note that the right-hand side converges absolutely, as seen from the following chain of inequalities:
		\begingroup
		\allowdisplaybreaks
		\begin{align}
			\displaystyle\sum_{j \in \mathcal{J} : \alpha \in \Gamma^\perp_j}& \frac{1}{s(\Gamma_{j})} \int\limits_{P_j}  \left|\overline{\widehat{g_{p}^{(1)}}(\gamma)} \widehat{g_{p}^{(2)}}(\gamma+\alpha) \right| d\mu_{P_j}(p) \leq \displaystyle\sum_{j \in \mathcal{J} } \frac{1}{s(\Gamma_{j})} \int\limits_{P_j}  \left|\overline{\widehat{g_{p}^{(1)}}(\gamma)} \widehat{g_{p}^{(2)}}(\gamma+\alpha) \right| d\mu_{P_j}(p) \nonumber\\
			&\leq\left(\displaystyle\sum_{j \in \mathcal{J}} \frac{1}{s(\Gamma_{j})} \int\limits_{P_j}  \left|\overline{\widehat{g_{p}^{(1)}}(\gamma)}  \right|^2 d\mu_{P_j}(p)\right)^{1/2} \left(\displaystyle\sum_{j \in \mathcal{J}} \frac{1}{s(\Gamma_{j})} \int\limits_{P_j}  \left|\overline{\widehat{g_{p}^{(2)}}(\gamma+\alpha)}  \right|^2 d\mu_{P_j}(p)\right)^{1/2} \nonumber\\ &\leq B_{g^{(1)}}^{1/2} B_{g^{(2)}}^{1/2},  \nonumber
		\end{align}
		\endgroup
		where the last inequality follows from \cite[Proposition 3.3]{JL2016}, and \( B_{g^{(1)}} \) and \( B_{g^{(2)}} \) are the Bessel constants associated with \( \gtione \) and \( \gtitwo \), respectively. Therefore, $\|t_\alpha\|_\infty \leq B_{g^{(1)}}^{1/2} B_{g^{(2)}}^{1/2}$, so $t_\alpha \in L^\infty(\widehat{G})$. Hence, the multiplication operator 
$M_{\bar{t_\alpha}} : L^2(\widehat{G}) \to L^2(\widehat{G})$, defined by 
$f \mapsto f.\bar{t}_\alpha$, 
is well-defined and bounded on $L^2(\widehat{G})$.  Also note that
		 for $f \in \mathcal{D}_B$,   
		\begin{align}
			\sum_{j \in \J: \alpha \in \Gjp} \int\limits_{\hg}\left|\widehat{f}(w) \overline{\widehat{f}(w+\alpha)} \frac{1}{s(\Gj)} \int\limits_{P_j} \overline{\widehat{\gpone}(w)} \widehat{\gptwo}(w+\alpha) d\mu_{P_j}(p)\right| d\mu_{\hg}(w) < \infty. \label{eq_finite}
		\end{align}
	 Now, for $f \in \mathcal{D}_B$, we compute:
		\begin{align*}
			\langle \widehat{f}, M_{\bar{t}_\alpha} T_{-\alpha}\widehat{f} \rangle&=\int\limits_{\widehat{G}} \widehat{f}(w) \overline{M_{\bar{t}_\alpha} T_{-\alpha}\widehat{f}(w)} \ d \mu_{\widehat{G}}(w)=\int\limits_{\widehat{G}} \widehat{f}(w) t_\alpha(w) \overline{\widehat{f}(w+\alpha)} \ d \mu_{\widehat{G}}(w)\\
			&=\int\limits_{\widehat{G}} \widehat{f}(w)  \overline{\widehat{f}(w+\alpha)} \displaystyle\sum_{j \in \mathcal{J} : \alpha \in \Gamma^\perp_j} \frac{1}{s(\Gamma_{j})} \int\limits_{P_j}  \overline{\widehat{g_{p}^{(1)}}(w)} \widehat{g_{p}^{(2)}}(w+\alpha)  d\mu_{P_j}(p) \ d \mu_{\widehat{G}}(w).
		\end{align*} 
		Applying Fubini’s theorem, we obtain:
		\begin{align*}
			\langle \widehat{f}, M_{\bar{t}_\alpha T_{-\alpha}}\widehat{f} \rangle&=\displaystyle\sum_{j \in \J : \alpha \in \Gamma^\perp_j} \frac{1}{s(\Gamma_{j})}\int\limits_{P_j} \int\limits_{\widehat{G}} \widehat{f}(w)  \overline{\widehat{f}(w+\alpha)}  \, \overline{\widehat{g_{p}^{(1)}}(w)} \widehat{g_{p}^{(2)}}(w+\alpha)\ d \mu_{\widehat{G}}(w)  d\mu_{P_j}(p) \\
			&= \sum_{j \in \J:\alpha \in \Gamma_{j}^\perp} d_{P_j, \alpha},
		\end{align*}
		which completes the proof.
	\ep
	With the help of the previous lemma, we now proceed to prove Proposition~\ref{prop_Fourier_multiplier}.

       \noindent{\bf Proof of Proposition~\ref{prop_Fourier_multiplier}}
		It is easy to observe that \( \Theta \) commutes with the family  \( \{T_x\}_{x \in G} \) if and only if \( \w \) is constant for all \( f \in \mathcal{D}_B \). Clearly, if $\Theta$ commutes with translations \( \{T_x\}_{x \in G} \), then \( \w \) is constant for all \( f \in \mathcal{D}_B \). Conversely, under the dual $1$-UCP, constancy of \( \w \) implies the vanishing of all non-zero Fourier coefficients; hence $\Theta$ commutes with translations. Next, we show that \( \w \) is constant if and only if, for each $\alpha \in  \umzero$, we have
		\begin{align*}
			\ta(\gamma)=	\dtna =0 \quad \mbox{for a.e. } \gamma \in \widehat{G}.
		\end{align*}
		Since the GTI systems under consideration satisfy the dual \(1\)-UCP, we follow an approach similar to that used in \cite[Proposition 3.10]{FL2019}. For \(x \in G\), the function \(\w(x)\) can be expressed as follows:
		\begin{align*}
			\w(x)=\sum_{\alpha \in \union} \alpha(x) \widehat{\w}(\alpha),
		\end{align*}
		where for $\alpha \in \union$, 
		\begin{align}
			\widehat{\w}(\alpha)&=\displaystyle\sum\limits_{j \in \mathcal{J}:\al \in \Gjp} d_{P_j, \alpha}, \label{eq_prop_1}
		\end{align}  
		and $d_{P_j, \alpha}$ is given in \eqref{eq_def_dpjalpha}.	Now for $x \in G$, we consider the function
		\[
		z_{f,g_p^{(1)},g_p^{(2)}}(x)=\w(x)-\w(0),
		\]
		which is continuous due to the continuity of \( \w \). The generalized Fourier series of \( z_{f, g^{(1)}_p, g^{(2)}_p}(x) \) is given by
		\[
		z_{f, g^{(1)}_p, g^{(2)}_p}(x) = \sum_{\alpha \in \union} \alpha(x)\, \widehat{z_{f, g^{(1)}_p, g^{(2)}_p}}(\alpha),
		\]
		with generalized Fourier coefficients
		\begin{equation} \label{eq_prop_2}
			\widehat{z_{f;g_p^{(1)},g_p^{(2)}}}(\alpha)	=	\begin{cases}
				\widehat{\w}(\alpha)-\w(0) &\mbox{if } \alpha =0,\\
				\widehat{\w}(\alpha) &\mbox{if } \alpha \neq 0. 
			\end{cases} 
		\end{equation}
		The function \( \w \) is constant if and only if \( z_{f, g^{(1)}_p, g^{(2)}_p}(x) \equiv 0 \). In view of the uniqueness theorem of generalized Fourier series \cite[Theorem~7.12]{Corduneanu1968}, this holds if and only if for all $\alpha \in \union$, we have
		\[
		\widehat{z_{f, g^{(1)}_p, g^{(2)}_p}}(\alpha) = 0.
		\]
	This is equivalent to stating that, for all $\alpha \in \union$,
		\begin{equation} \label{eq_prop_*}
			\widehat{\w}(\alpha) = \delta_{\alpha, 0} \w(0).
		\end{equation}
		To prove the forward implication, assume that $\w$ is constant. Then, by \eqref{eq_prop_*}, for all $\alpha \in \umzero$, we have
		\begin{equation} \label{eq_prop_3}
			\widehat{\w}(\alpha) = 0.  
		\end{equation}
		Next, by Lemma~\ref{lem_Fourier_cofficents}, for \( f \in \mathcal{D}_B \), we have
		\[
		\sum_{j \in \mathcal{J} : \alpha \in \Gamma_j^\perp} d_{P_j, \alpha} = \left\langle \widehat{f}, \, M_{\bar{t}_\alpha} T_{-\alpha} \widehat{f} \right\rangle,
		\]
		and by \eqref{eq_prop_1}, it follows that
		\begin{equation} \label{eq_prop**}
			\left\langle \widehat{f}, \, M_{\bar{t}_\alpha} T_{-\alpha} \widehat{f} \right\rangle = \widehat{\w}(\alpha).
		\end{equation}
		Therefore, using \eqref{eq_prop_3}, we obtain
		\[
		\left\langle \widehat{f}, \, M_{\bar{t}_\alpha} T_{-\alpha} \widehat{f} \right\rangle = 0 \quad \text{for all } \alpha \in \umzero.
		\]
        By the polarization identity, this implies for all \( \alpha \in \umzero \),
\[
\langle \widehat{f}, M_{\bar{t}_\alpha} T_{-\alpha} \widehat{g} \rangle = 0
\quad \text{for all } \widehat{f}, \widehat{g} \in \mathcal{D}_B,
\]
and hence for all $\widehat{f}, \widehat{g} \in L^2(\widehat{G})$ as \( \mathcal{D}_B \) is dense in \( L^2(\widehat{G}) \). Therefore,
		\[
		M_{\bar{t}_\alpha} T_{-\alpha} = 0, \quad \forall \; \alpha \in \umzero.
		\]
		Hence for all \( \alpha \in \umzero \), we have
		\[
		M_{\bar{t}_\alpha} T_{-\alpha} \widehat{g}(\g) = \bar{t}_\alpha(\g) \widehat{g}(\g+\alpha) = 0 \quad \text{a.e. } \g \in \widehat{G}, \, \forall \, \widehat{g} \in L^2(\widehat{G}),
		\]
		which implies \( t_\alpha = 0 \) for all \( \alpha \in \umzero \).
		Conversely, we assume \( t_\alpha = 0 \) for all \( \alpha \in \umzero \). Then,
		\[
		\left\langle \widehat{f}, \, M_{\bar{t}_\alpha} T_{-\alpha} \widehat{f} \right\rangle = 0 \quad \forall f \in \mathcal{D}_B,
		\]
		and by \eqref{eq_prop**}, it follows that \( \widehat{\w}(\alpha) = 0 \) for \( \alpha \in \umzero \). Moreover, from \eqref{eq_prop_2}, we have  $\widehat{\w}(0)=\w(0)$, so
		\[
		\widehat{\w}(\alpha)=\delta_{\alpha,0} \w(0) \quad \text{for all } \alpha \in \union.
		\]
		Hence, by \eqref{eq_prop_*}, we conclude that \( \w \) is constant. 
		Without loss of generality, assume that (i) holds. It is a standard fact that whenever the operator $\Theta$ commutes with the translation operators $\{T_x\}_{x \in G}$, the operator $\Theta$ must act as a Fourier multiplier on LCA groups; see \cite[Theorem~4.1.1]{LarsenR}. That is  \( \widehat{\Theta f}(w) = s(w)\, \widehat{f}(w) \),  where  $s \in L^\infty(\widehat{G})$  is unique. 
		Now, using the definition of \( \w \), we have
		\begin{align} 
			\w(0) &= \left\langle \Theta f, f \right\rangle = \left\langle \widehat{\Theta f}, \widehat{f} \right\rangle \nonumber\\ &= \int\limits_{\widehat{G}} \widehat{\Theta f}(w)\, \overline{\widehat{f}(w)} \, d\mu_{\widehat{G}}(w)
			\nonumber\\ 		
			&= \int\limits_{\widehat{G}} s(w)\, \widehat{f}(w)\, \overline{\widehat{f}(w)} \, d\mu_{\widehat{G}}(w).\label{eq_prop_4}
		\end{align}
		Here, the last step follows from the Fourier multiplier representation of \( \Theta \). Next, using equation \eqref{eq_prop_1} and simplifying, we obtain
		\begingroup
		\allowdisplaybreaks
		\begin{align}
			\widehat{\w}(\alpha)
			&= \sum_{j \in \mathcal{J}: \alpha \in \Gamma_j^\perp} d_{P_j, \alpha} \nonumber  \\
			&= \sum_{j \in \mathcal{J}: \alpha \in \Gamma_j^\perp} \frac{1}{s(\Gamma_j)} \int\limits_{P_j} \int\limits_{\widehat{G}} 
			\widehat{f}(w)\, \overline{\widehat{f}(w+\alpha)}\, \overline{\widehat{g_p^{(1)}}(w)}\, \widehat{g_p^{(2)}}(w+\alpha)\,
			d\mu_{\widehat{G}}(w) \, d\mu_{P_j}(p) \nonumber \\
			&= \sum_{j \in \mathcal{J}: \alpha \in \Gamma_j^\perp} \int\limits_{\widehat{G}} \widehat{f}(w)\, \overline{\widehat{f}(w+\alpha)} 
			\left( \frac{1}{s(\Gamma_j)} \int\limits_{P_j} \overline{\widehat{g_p^{(1)}}(w)}\, \widehat{g_p^{(2)}}(w+\alpha) \, d\mu_{P_j}(p) \right)
			d\mu_{\widehat{G}}(w) \nonumber \\
			&=  \int\limits_{\widehat{G}} \widehat{f}(w)\, \overline{\widehat{f}(w+\alpha)} 
			\left(\sum_{j \in \mathcal{J}: \alpha \in \Gamma_j^\perp} \frac{1}{s(\Gamma_j)} \int\limits_{P_j} \overline{\widehat{g_p^{(1)}}(w)}\, \widehat{g_p^{(2)}}(w+\alpha) \, d\mu_{P_j}(p) \right)
			d\mu_{\widehat{G}}(w). \label{eq_prop_5}
		\end{align}
		\endgroup
		Also, using equation \eqref{eq_prop_*}, and applying \eqref{eq_prop_1} for \( \alpha = 0 \), we get
		\[
		\w(0)=\widehat{\w(0)}=\displaystyle\sum\limits_{j \in \J} d_{\pj,0}.
		\]
		By setting \( \alpha = 0 \) in equation \eqref{eq_prop_5} and substituting into the expression above, we obtain:
		\begin{align}
			\w(0) 
			&= \int\limits_{\widehat{G}} \left( \sum_{j \in \mathcal{J}} \frac{1}{s(\Gamma_j)} \int\limits_{P_j} 
			\overline{\widehat{g_p^{(1)}}(w)}\, \widehat{g_p^{(2)}}(w) \, d\mu_{P_j}(p) \right)
			\widehat{f}(w)\, \overline{\widehat{f}(w)} \, d\mu_{\widehat{G}}(w). \label{eq_prop_6}
		\end{align}
		Comparing \eqref{eq_prop_4} and \eqref{eq_prop_6}, we conclude that
		\[
		s(w) = \sum_{j \in \mathcal{J}} \frac{1}{s(\Gamma_j)} 
		\int\limits_{P_j} \overline{\widehat{g_p^{(1)}}(w)}\, \widehat{g_p^{(2)}}(w)\, d\mu_{P_j}(p).
		\]
		This completes the proof.		
	\ep
    \noindent{\bf Proof of Theorem \ref{thm_characterization_orthogonal_frame}.}
		\textbf{(i) $\implies$ (ii):}
Let $\Theta$ denote the mixed dual Gramian operator associated with given systems. 
Since $\ta(\gamma)=0$ for every $\alpha \in \umzero$, Proposition~\ref{prop_Fourier_multiplier} implies that $\Theta$ acts as a Fourier multiplier with symbol $s$, that is, \( \widehat{\Theta f}(w) = s(w)\, \widehat{f}(w) \). Moreover, 
		\[
		s(w) = \sum_{j \in \mathcal{J}} \frac{1}{s(\Gamma_j)} 
		\int\limits_{P_j} \overline{\widehat{g_p^{(1)}}(w)}\, \widehat{g_p^{(2)}}(w)\, d\mu_{P_j}(p)= t_0(w)=0
		\]
		for almost every \( w \in \widehat{G} \). Thus,	 $\widehat{\Theta f}=0$, which implies $\Th f=0$ for all $f \in \mathcal{D}_B$.  Therefore, $\Th=0$  and hence the GTI systems $\gtione$ and $\gtitwo$ are pairwise orthogonal Bessel families (and frames). This proves (ii).
		
		\medskip
		\noindent
		\textbf{(ii) $\implies$ (i):} Assume that the GTI systems $\gtione$ and $\gtitwo$ are pairwise orthogonal frames, i.e., $\Theta=0$. Then $\Th T_x f=0$ for all $x \in G$ and $f \in \mathcal{D}_B$. By definition of $\w$, we have 
		\begin{align*}
			\w(x)=\left\langle \Theta T_x f, T_x f \right \rangle \quad \mbox{for } f \in \mathcal{D}_B, \, x \in G.
		\end{align*}
		Therefore, $\w\equiv 0$.
		Applying the uniqueness theorem for generalized Fourier series, we conclude that $\widehat{\w}(\al)=0$, for all $\al \in \union$.
		Using \eqref{eq_prop_1} and Lemma \ref{lem_Fourier_cofficents}, for all \( \al \in \union \), we obtain
		\begin{equation}
			\widehat{\w}(\al)=\displaystyle\sum_{j \in \J: \al \in \Gjp} d_{j,\al}=\left \langle \widehat{f}, M_{\bar{t}_\al }T_{-\al}\widehat{f} \right \rangle=0, 
			\	\end{equation}
		$\forall$ \( f \in \mathcal{D}_B \). We have 
\(
M_{\bar{t}_\alpha} T_{-\alpha} \widehat{f} = 0 \quad \text{for all } \widehat{f} \in L^2(\widehat{G}),
\)
because $\mathcal{D}_B$ is dense in $L^2(\widehat{G})$. 	Hence,
		\[
		M_{\bar{t}_\alpha} T_{-\alpha} \widehat{g}(\g) = \bar{t}_\alpha(\g) \widehat{g}(\g+\alpha) = 0 \quad \text{for a.e. } \g \in \widehat{G}, \, \forall \, \widehat{g} \in L^2(\widehat{G})
		\]
		which implies \( t_\alpha = 0 \) for all \( \alpha \in \union \). This proves (i).
	\ep

	Following Remark \ref{rem_relation_between_UCP_LIC}, if the so-called dual $\alpha$-LIC holds for the two GTI systems, then these systems satisfy the dual $\infty$-UCP. Consequently, \cite[Theorem 3.5]{GS2018} becomes a corollary of Theorem~\ref{thm_characterization_orthogonal_frame}, as stated below:
	
	\begin{cor}
		Let $\gtione$ and $\gtitwo$ be two GTI Bessel (respectively, frame) systems in $L^2(G)$ satisfying the dual $\alpha$-LIC. Then the following statements are equivalent:
		\begin{itemize}
			\item[(i)] The above GTI Bessel (respectively, frame)  systems in $L^2(G)$ are pairwise orthogonal. 
			\item[(ii)] For each $\alpha \in \umzero,$ we have
			\begin{equation*}
				t_\alpha(\gamma):=	\displaystyle\sum_{j \in \mathcal{J} : \alpha \in \Gamma^\perp_j} \frac{1}{s(\Gamma_{j})} \int\limits_{P_j}  \overline{\widehat{g_{p}^{(1)}}(\gamma)} \widehat{g_{p}^{(2)}}(\gamma+\alpha)  d\mu_{P_j}(p) =0  \mbox{ for a.e. } \gamma \in \widehat{G},
			\end{equation*}
			and 
			\begin{equation*}
				t_0(\gamma):=	\displaystyle\sum_{j \in \mathcal{J} } \frac{1}{s(\Gamma_{j})} \int\limits_{P_j}  \overline{\widehat{g_{p}^{(1)}}(\gamma)} \widehat{g_{p}^{(2)}}(\gamma)  d\mu_{P_j}(p) =0  \mbox{ for a.e. } \gamma \in \widehat{G}.
			\end{equation*}
		\end{itemize}
	\end{cor}

  Next,  we define the dilation operator, for $\eta \in \operatorname{Aut}(G)$, by
\[
(D_\eta f)(x) = \Delta(\eta)^{-1/2} f(\eta x),
\]
where $\Delta(\eta)$ is determined by
\[
\mu_G(\eta(E)) = \Delta(\eta)\,\mu_G(E).
\]
Then $D_\eta$ is unitary on $L^2(G)$ and
\[
\widehat{D_\eta f}(\gamma) = \Delta(\eta)^{1/2}\widehat{f}(\eta^\ast(\gamma)).
\]    
With these preparations in place, we turn to the proof of Theorem~\ref{thm_new_characterization_orthogonal_frame}.
	
	\noindent{\bf Proof of Theorem \ref{thm_new_characterization_orthogonal_frame}}
	Let $D_{\e}$ and $D_{\z}$ be the (unitary) dilation operators associated with ${\e}$ and $\z$, respectively. By polarisation identity and density argument, $\newgtione$ and $\newgtitwo$ are pairwise orthogonal if and only if for every  $f \in \mathcal{D}_B$, we have
	\begin{align*}
		\langle D_\z \Theta_{\e, \z} D_\e^{-1}f, f \rangle =0,
	\end{align*}
	as $\mathcal{D}_B$ is dense in $\ltg$.
We may reformulate the left-hand side of the above expression as
	\begin{align}
		\langle D_\z  \Theta_{\e, \z} D_\e^{-1}f, f \rangle 
		&=	\langle \Theta_{\e, \z} D_\e^{-1}f, D_\z^{-1} f \rangle  \nonumber\\
		&=	\sum_{j \in \mathcal{J}} \int\limits_{P_j} \int\limits_{\Gamma_j} 
		\langle D_\e^{-1}f, T_{\eta \la} g^{(1)}_p \rangle \,  \langle T_{\z \la} g^{(2)}_p, D_\z^{-1} f \rangle  \,	d\mu_{\Gamma_j}(\la)\, d\mu_{P_j}(p) \nonumber\\
		&=\sum_{j \in \mathcal{J}} \int\limits_{P_j} \int\limits_{\Gamma_j} 
		\langle f, D_\e T_{\eta \la} g^{(1)}_p \rangle \,  \langle D_\z T_{\z \la} g^{(2)}_p, f \rangle  \,	d\mu_{\Gamma_j}(\la)\, d\mu_{P_j}(p). \nonumber
	\end{align}
	Using the intertwining relation
	\(	D_{\e} T_\la = T_{\e^{-1}(\la)} D_{\e}  \) for \( \la \in \Gamma_j,\) the above expression becomes
	\begin{align*}
		\langle D_\z  \Theta_{\e, \z} D_\e^{-1}f, f \rangle &=\sum_{j \in \mathcal{J}} \int\limits_{P_j} \int\limits_{\Gamma_j} 
		\langle f,  T_{\la} D_\e g^{(1)}_p \rangle \,  \langle  T_{\la} D_\z g^{(2)}_p, f \rangle  \,	d\mu_{\Gamma_j}(\la)\, d\mu_{P_j}(p). \nonumber
	\end{align*}
	Thus $\newgtione$ and $\newgtitwo$ are pairwise orthogonal if and only if the systems $\displaystyle\bigcup\limits_{j \in \mathcal{J}} \{T_{\la} D_\e g^{(1)}_p \}_{\g \in \Gamma_j}$ and $\displaystyle\bigcup\limits_{j \in \mathcal{J}} \{T_{\la} D_\z g^{(2)}_p \}_{\la \in \Gamma_j}$  are pairwise orthogonal.  Next, we prove that the GTI systems $\displaystyle\bigcup\limits_{j \in \mathcal{J}} \{T_{\la} D_\e g^{(1)}_p \}_{\g \in \Gamma_j}$ and $\displaystyle\bigcup\limits_{j \in \mathcal{J}} \{T_{\la} D_\z g^{(2)}_p \}_{\la \in \Gamma_j}$ satisfy dual $1$-UCP. By hypothesis, without loss of generality, we assume that $\newgtione$ satisfies $1$-UCP. Thus,  for each $f \in \mathcal{D}_{B}$, $\wgpone$ is almost periodic and 
	$$\wgpone=\sum_{j \in \mathcal{J}}\wgponej$$
	converges with respect to $M$. 
	
	Now for $f \in \mathcal{D}_B$ and $x \in G$, we have
	\begin{align*}
		w_{D_\e^{-1}f,\gpone,j}(x)&=\int\limits_{P_j}\int\limits_{\Gamma_{j}} \left| \langle T_x D_\e^{-1} f, T_{\e\la} \gpone \rangle \right|^2 \ d\mu_{\Gamma_{j}}(\la) d\mu_{P_j}(p)\\
		&=\int\limits_{P_j}\int\limits_{\Gamma_{j}} \left| \langle D_\e^{-1}T_{\e^{-1} x} f,  T_{\e\la} \gpone \rangle \right|^2 \ d\mu_{\Gamma_{j}}(\la) d\mu_{P_j}(p)\\
		&=\int\limits_{P_j}\int\limits_{\Gamma_{j}} \left| \langle T_{\e^{-1} x} f, D_\e T_{\e\la} \gpone \rangle \right|^2 \ d\mu_{\Gamma_{j}}(\la) d\mu_{P_j}(p)\\
		&=\int\limits_{P_j}\int\limits_{\Gamma_{j}} \left| \langle T_{\e^{-1} x} f,  T_{\la} D_\e\gpone \rangle \right|^2 \ d\mu_{\Gamma_{j}}(\la) d\mu_{P_j}(p).
	\end{align*}
	Thus for $g:=D_\e^{-1}f$ and $y:=\eta^{-1}x$, we have 
	\begin{align}\label{eq_wggpone_new_expression}
		w_{g,\gpone,j}(x)&=\int\limits_{P_j}\int\limits_{\Gamma_{j}} \left| \langle T_{y} g,  T_{\la} D_\e\gpone \rangle \right|^2 \ d\mu_{\Gamma_{j}}(\la) d\mu_{P_j}(p),
	\end{align}	
    which observes directly that both almost periodicity and convergence with respect to the invariant mean are preserved under composition with group automorphisms.
	Consequently, if $\newgtione$ satisfies the $1$-UCP, then the system  $\displaystyle\bigcup\limits_{j \in \mathcal{J}} \{T_{\la} D_\e g^{(1)}_p \}_{\la \in \Gamma_j}$ satisfy the $1$-UCP. Similarly, since $\newgtitwo$ is a Bessel, it follows that   $\displaystyle\bigcup\limits_{j \in \mathcal{J}} \{T_{\la} D_\z g^{(2)}_p \}_{\la \in \Gamma_j}$ is Bessel system. Now, in view of Remark \ref{rem_dual_UCP}, the systems   $\displaystyle\bigcup\limits_{j \in \mathcal{J}} \{T_{\la} D_\e g^{(1)}_p \}_{\la \in \Gamma_j}$ and $\displaystyle\bigcup\limits_{j \in \mathcal{J}} \{T_{\la} D_\z g^{(2)}_p \}_{\la \in \Gamma_j}$ satisfy the dual $1$-UCP.	By Theorem \ref{thm_characterization_orthogonal_frame},   these systems are pairwise orthogonal  if and only if for each  $\al \in \umzero,$ we obtain
	\begin{align*}
		\displaystyle\sum_{j \in \mathcal{J} : \alpha \in \Gamma^\perp_j} \frac{1}{s(\Gamma_{j})} \int\limits_{P_j}  \overline{\widehat{D_\e g_{p}^{(1)}}(\gamma)} \widehat{D_\z g_{p}^{(2)}}(\gamma+\al)  d\mu_{P_j}(p)=0 \quad \mbox{for a.e.  } \g \in \hg,
	\end{align*}
	and
	\begin{align*}
		\displaystyle\sum_{j \in \mathcal{J}} \frac{1}{s(\Gamma_{j})} \int\limits_{P_j}  \overline{\widehat{D_\e g_{p}^{(1)}}(\gamma)} \widehat{D_\z g_{p}^{(2)}}(\gamma)  d\mu_{P_j}(p)=0 \quad \mbox{for a.e.  } \g \in \hg.
	\end{align*}
	Since $\widehat{D_\e f}(\gamma)=(\Delta (\e))^{1/2}\widehat{f}(\e^{\ast}(\gamma))$, this orthogonality condition is equivalent to, for each  $\al \in \umzero,$ we have
	\begin{align*}
		\displaystyle\sum_{j \in \mathcal{J} : \alpha \in \Gamma^\perp_j} \frac{1}{s(\Gamma_{j})} \int\limits_{P_j}  \overline{\widehat{ g_{p}^{(1)}}(\e^\ast \gamma)} \widehat{g_{p}^{(2)}}(\z^\ast(\gamma+\al))  d\mu_{P_j}(p)=0, \quad \mbox{for a.e.  } \g \in \hg,
	\end{align*}
	and
	\begin{align*}
		\displaystyle\sum_{j \in \mathcal{J}} \frac{1}{s(\Gamma_{j})} \int\limits_{P_j}  \overline{\widehat{ g_{p}^{(1)}}(\e^\ast\gamma)} \widehat{g_{p}^{(2)}}(\z^\ast\gamma)  d\mu_{P_j}(p)=0 \quad \mbox{for a.e.  } \g \in \hg.
	\end{align*}
	This completes the proof.	
	\ep

    \noindent{\bf Proof of Theorem \ref{thm_char_Parseval_frames}.}
	Assume that (ii) is satisfied. Then $\ta(\gamma)=0$ for every $\alpha \in \umzero$. 
By Proposition~\ref{prop_Fourier_multiplier}, it follows that $\Theta$ is a Fourier multiplier, that is, \( \widehat{\Theta f}(w) = s(w)\, \widehat{f}(w) \), $w \in \hg$. Moreover,
	\[
	s(w) = \sum_{j \in \mathcal{J}} \frac{1}{s(\Gamma_j)} 
	\int\limits_{P_j} \overline{\widehat{g_p^{(1)}}(w)}\, \widehat{g_p^{(1)}}(w)\, d\mu_{P_j}(p)= t_0(w)=K, \mbox{ for } w \in \hg.
	\]
	by hypothesis. Thus $\norm{\widehat{\Theta f}(w)}=K\norm{\widehat{f}}$ and hence, the system $\gtione$ is a tight  frame with frame bound $K$. The converse part follows directly by observing the following formula:
	\begin{align*}
		\widehat{\w}(\al)=\displaystyle\sum_{j \in \J: \al \in \Gjp} d_{j,\al}=\left \langle \widehat{f}, M_{\bar{t}_\al }T_{-\al}\widehat{f} \right \rangle=K \delta_{\alpha, 0}.
	\end{align*}
This concludes the proof.
	\ep
	As an application of the GTI framework, we provide explicit criteria for the pairwise orthogonality of structured systems such as wavelets, Gabor, and shearlets.

    \subsection{Applications to Gabor, wavelet, and shearlet systems}\label{sec_application_to_characterization_result}

   This subsection illustrates applications of our main characterization results, Theorems~\ref{thm_new_characterization_orthogonal_frame} and \ref{thm_characterization_orthogonal_frame}, to several structured systems arising in harmonic analysis.
In particular, we apply Theorem~\ref{thm_new_characterization_orthogonal_frame} to Bessel families with TI, Gabor, wavelet, and cone-adapted shearlet structures, all of which can be viewed as special cases of GTI systems.
	
	This subsection is organized as follows.
In Corollary~\ref{cor_char_ortogonal_gabor}, we characterize pairwise orthogonality of Gabor systems. Subsequently, Proposition~\ref{prop_orthogonal_wavelet_characterization} provides a corresponding characterization for pairwise orthogonal wavelet systems in $\ltg$. Furthermore, pairwise orthogonality of composite wavelet and classical shearlet systems is established in  Proposition~\ref{prop_orthogonal_CWS} and Proposition~\ref{prop_char_orthogonal_CSS}, respectively. Finally, we present a characterization for cone-adapted shearlet systems (see, Proposition~\ref{prop_char_orthogonal_CASS}), which are often more effective in applications due to their ability to represent directional information more uniformly. Moreover, cone-adapted shearlet systems may be viewed as a finite union of shift-invariant systems.  In particular, they can also be interpreted as wavelet systems associated with composite dilations.

\subsubsection{Translation invariant systems}

Theorem \ref{thm_new_characterization_orthogonal_frame} immediately yields the following application:

\begin{prop}\label{prop_char_pairwise_orthogonal_TI}
	Let the TI systems \(\bigcup_{j \in \mathcal{J}}\{T_{\la} g^{(1)}_p\}_{\la \in \e\Gamma, p \in P}\) and \(\bigcup_{j \in \mathcal{J}}\{T_{ \la} g^{(2)}_p\}_{\la \in \z \Gamma, p \in P}\) be  Bessel (frame) systems. Then, these systems are pairwise orthogonal if and only if  for each $\al \in \Gamma^\perp$, we have 
	\begin{equation*}\label{eq_t_alpha}
		\ta(\gamma):=\sum_{j \in \J} \frac{1}{s(\Gamma)} \int\limits_{P_j} \overline{\widehat{g_p^{(1)}} (\e^\ast \gamma)} \widehat{g_p^{(2)}} (\z^\ast (\gamma +\al)) \, d\mu_{P_j}(p)=0  \mbox{ for a.e. } \gamma \in \hg.
	\end{equation*}
\end{prop}

\subsubsection{Gabor systems}
Given $\chi \in \widehat{G}$, we define the \textit{modulation operator} $M_\chi$ acting on $L^2(G)$ as $M_{\chi}(f)(x)=\chi(x)f(x) \mbox{ for all } x \in G.$
This operator satisfies the following Fourier domain identity:
\begin{equation}\label{eq_modulation_fourier}
	(\widehat{M_\chi f})(\g)=\int\limits_{G} \chi(x) f(x) \overline{\g(x)}d\mu_{G}(x)=\int\limits_{G}  f(x) \overline{(\g-\chi)(x)}d\mu_{G}(x)=T_\chi \widehat{f}(\g)
\end{equation}
for all $f \in L^2(G)$ and  a.e. $\gamma \in \widehat{G}$. Consider co-compact subgroups $\Gamma$ of $G$ and $\Lambda$ of $\widehat{G}$. Consider a measure space $(\Lambda, \Sigma_{\Lambda}, \mu_{\Lambda})$ satisfying the standing hypothesis, where $\Sigma_\Lambda$ is a $\sigma$-algebra and $\mu_\Lambda$ is a regular Borel measure. Let $\mathcal{J} \subset \mathbb{Z}$ be an index set, and let  $\Phi := \{\phi_j : j \in \mathcal{J}\} \subset L^2(G)$ be a set of functions. Then the collection 
\begin{equation}\label{eq_Gabor_system}
	\{T_\la M_\chi \phi_{j}\}_{\la \in \e\Gamma, \chi \in \Lambda, j \in \J}
\end{equation} 
is referred to as the \textit{Gabor system} generated by the set $\Phi$.  We can express the Gabor system as
\begin{align*}
	\{T_\la M_\chi \phi_{j}\}_{\la \in \eta \G, \chi \in \Lambda, j \in \J}=\bigcup\limits_{j \in \mathcal{J}} \{T_\la g_{p}^{(1)}\}_{\la \in \e\Gamma, p \in P_j},
\end{align*}
where $P_j = \{(j, \chi) : \chi \in \Lambda\}$ and $g_p^{(1)} = g_{j, \chi}^{(1)} := M_\chi \phi_j$. In this representation, the modulation parameter $\chi$ is incorporated into the index set $P_j$, while the translation subgroup remains fixed. Thus, the Gabor system is a special case of a TI system. Consequently, the following result is a direct corollary of Proposition~\ref{prop_char_pairwise_orthogonal_TI}.
\begin{cor}\label{cor_char_ortogonal_gabor}
	Let $	\{ T_\la M_\chi \psi_{j}\}_{  \la \in \e\Gamma, \chi \in \Lambda, j \in J}$ and $	\{ T_\la M_\chi \phi_{j}\}_{ \la \in \z \G, \chi \in \Lambda, j \in J}$ be two Gabor Bessel (frame) systems  in $L^2(G)$. Then these two systems are pairwise orthogonal Gabor Bessel (frame) systems in $L^2(G)$ if and only if, for each $\al \in \Gamma^\perp$, the following condition is satisfied:   
	\begin{align}
		\sum\limits_{j \in J}  \int\limits_{\Lambda} \overline{\widehat{\phi}_j( \e^\ast(\g-\chi))} \widehat{\psi}_j(\z^\ast(\g+\al-\chi))d \mu_{\Lambda}(\chi)=0 \quad \mbox{for a.e. } \g \in \widehat{G}. \nonumber
	\end{align}
\end{cor}
Note that if we assume $\e=\z=I$ (identity automorphism), then the above result coincides with \cite[Proposition~4.3]{GS2018}. Furthermore, if we assume $G=\mathbb{Z}^n$, and \(\e= \z = I\) (the identity matrix) and that \(P\) is a singleton set, then the result reduces to the characterization given by Lopez and Han in \cite[Theorem~1.4(ii)]{LH}. 

\subsubsection{Wavelet systems}
Let $\mathcal{J} \subset \mathbb{Z}$ be an index set and consider a family of automorphisms $\mathcal{A} = \{\eta_j : j \in \mathcal{J}\} \subset \operatorname{Aut}(G)$. Choose a co-compact subgroup $\Gamma$  of $G$. For a countable index set $\mathcal{I}$, let $\Psi := \{\psi_i : i \in \mathcal{I}\} \subset L^2(G)$ be a collection of functions. Then, the system
\[
\{D_{\eta_j \eta^{-1}} T_\la \psi_i\}_{j \in \mathcal{J},\, \la \in  \G,\, i \in \mathcal{I}}
\]
is called the \textit{wavelet system} generated by $\Psi$.
Using the intertwining relation
\(
D_{\eta_j \eta^{-1}} T_\la = T_{\eta \eta_j^{-1}(\la)} D_{\eta_j \eta^{-1}}  \) for \(\eta_j \in \mathcal{A} \text{ and } \la \in \G,\)
the wavelet system can be expressed in the form of a GTI system as \[
\{D_{\eta_j \eta^{-1}} T_\la \psi_i\}_{j \in \mathcal{J},\, \la \in \Gamma,\, i \in \mathcal{I}}=  \bigcup_{j \in \mathcal{J}} \{T_\la g_p^{(1)}\}_{\la \in \eta\Gamma_j,\, p \in P_j},
\] where $\Gamma_j = \eta_j^{-1} \Gamma$ for each $\eta_j \in \mathcal{A}$, the functions $g_p^{(1)} = g_{i,j}^{(1)} := D_{\eta_j \eta^{-1}} \psi_i$, and  $P_j = \{(i, j) : i \in \mathcal{I}\}$. The automorphism $\e' : \widehat{G} \to \widehat{G}$, defined on $\widehat{G}$, is the adjoint of $\e \in \operatorname{Aut}(G)$. The annihilator $\Gamma_j^\perp$ of $\Gamma_j$ is
$\Gamma_j^\perp = ((\eta_j)^{-1}\Gamma)^\perp = \e'_j(\Gamma^\perp),$
 \cite[Proposition 6.5]{BMR}. As an application of Theorem~\ref{thm_new_characterization_orthogonal_frame}, we obtain the following result.
\begin{prop}\label{prop_orthogonal_wavelet_characterization}
	Let   $\{D_{\eta_j \e^{-1}} T_\la \psi_i\}_{j \in \mathcal{J},\, \la \in  \G,\, i \in \mathcal{I}}$ and $\{D_{\e_j \z^{-1}} T_\la \phi_i\}_{j \in \mathcal{J},\, \la \in \Gamma,\, i \in \mathcal{I}}$ be   two wavelet Bessel (frame) systems in $L^2(G)$ such that one of the system  is satisfying the  \(1\)-UCP. Then the following assertions are equivalent:
	\begin{itemize}
		\item[(i)]  $	\{D_{\eta_j \e^{-1}} T_\la \psi_i\}_{j \in \mathcal{J},\, \la \in  \G,\, i \in \mathcal{I}}$ and $\{D_{\e_j \z^{-1}} T_\la \phi_i\}_{j \in \mathcal{J},\, \la \in \Gamma,\, i \in \mathcal{I}}$  are pairwise orthogonal wavelet Bessel (frame) systems  in $L^2(G)$.
		\item[(ii)] For each $\al \in \bigcup\limits_{j \in \J}\e'_j\Gamma^\perp,$ we have
		\begin{equation*}
			\displaystyle\sum_{j \in \mathcal{J} : \alpha \in \e'_j(\Gamma^\perp)} \frac{1}{s(\Gamma_{j})} \sum_{i \in \mathcal{I}}  \overline{\widehat{\psi_i}\left(\e'\e_j^\ast\e^\ast\gamma\right)} \widehat{\phi_i}\left(\z'\z_j^\ast\z^\ast(\gamma+\al)\right) =0 \quad \mbox{for a.e. } \gamma \in \widehat{G}.
		\end{equation*}
	\end{itemize}
\end{prop}
\bc\label{cor_orthogonal_wavelet_characterization}
Let   $\{D_{\eta_j} T_\la \psi_i\}_{j \in \mathcal{J},\, \la \in  \G,\, i \in \mathcal{I}}$ and $\{D_{\e_j} T_\la \phi_i\}_{j \in \mathcal{J},\, \la \in \Gamma,\, i \in \mathcal{I}}$ be   two wavelet Bessel (frame) systems in $L^2(G)$ satisfying the corresponding dual $1$-UCP. Then these wavelet systems are pairwise orthogonal if and only if 
for each $\al \in \bigcup\limits_{j \in \J}\e'_j\Gamma^\perp,$ we have
		\begin{equation*}
			\displaystyle\sum_{j \in \mathcal{J} : \alpha \in \e'_j(\Gamma^\perp)} \frac{1}{s(\Gamma_{j})} \sum_{i \in \mathcal{I}}  \overline{\widehat{\psi_i}\left(\e_j^\ast\gamma\right)} \widehat{\phi_i}\left(\z_j^\ast(\gamma+\al)\right) =0 \quad \mbox{for a.e. } \gamma \in \widehat{G}.
		\end{equation*}
\ec

\subsubsection{Composite wavelets and shearlet systems}
Let $\mathcal{K}$ and $\mathcal{J}$ be two countable index sets, and consider their Cartesian product $\mathcal{K} \times \mathcal{J}$.  
For each $k \in \mathcal{K}$ and $j \in \mathcal{J}$, let $C_k, D_j \in \mathrm{GL}_d(\mathbb{R})$.  
The system generated by the pair 
\[
\Big( \{C_k D_j\}_{(k,j) \in \mathcal{K} \times \mathcal{J}}, \Gamma \Big),
\] 
known as a \textit{wavelet system with composite dilation} in $L^2(\mathbb{R}^d)$, is given by
\[
\{ D_{C_k D_j} T_\lambda \psi_\ell \}_{ \ell \in \mathcal{I}, \, \lambda \in \Gamma,\,  j \in \mathcal{J},\, k \in \mathcal{K}},
\]
where $\{\psi_\ell : \ell \in \mathcal{I}\}$ is the generating set of functions and $\Gamma = E \mathbb{Z}^d \subset \R^d$ is a full-rank lattice (see~\cite{GuoLabateLimWeissWilson}).
It is common to assume that one of the two matrix families, for instance $\{C_k\}_{k \in \mathcal{K}}$, preserves volume.  
In the present framework, we require that the transposes $C_k^{\mathrm{T}}$, $k \in \mathcal{K}$, leave the annihilator $\Gamma^\perp$ invariant, i.e.,
\(
C_k^{\mathrm{T}} \Gamma^\perp = \Gamma^\perp.
\)
For example, if $\Gamma = \mathbb{Z}^d$, this assumption corresponds to $C_k \in \mathrm{SL}_d(\mathbb{Z})$.
Under this assumption, for each $(k,j) \in \mathcal{K} \times \mathcal{J}$, the annihilator subgroup can be written as
\(
 (C_k D_j \Gamma)^\perp = D_j^{\mathrm{T}} C_k^{\mathrm{T}} \Gamma^\perp = D_j^{\mathrm{T}} \Gamma^\perp
\). 
Thus, the following result is an application of Theorem~\ref{thm_characterization_orthogonal_frame}. 
\begin{prop}\label{prop_orthogonal_CWS}
	Let   $	\{ D_{C_k D_j} T_\la \psi_\ell \}_{k \in \mathcal{K},\, j \in \mathcal{J},\, \la \in \Gamma,\, \ell \in \mathcal{I}}$ and $\{ D_{C_k D_j} T_\la \phi_\ell \}_{k \in \mathcal{K},\, j \in \mathcal{J},\, \la \in \Gamma,\, \ell \in \mathcal{I}}$ be  two wavelet with composite dilation Bessel (frame) systems  in $L^2(\R^d)$ satisfying the corresponding dual $1$-UCP are pairwise orthogonal  if and only if 
	for each $\alpha \in \bigcup\limits_{j \in \J}D^T_j( E^{\ast}\Z^d),$ we have
	\begin{equation*}
		\displaystyle\sum_{j \in \mathcal{J} : \alpha \in D^T_j( E^{\ast}\Z^d)} \frac{1}{s(\Gamma_{j})} \sum_{\ell \in \mathcal{I}} \sum_{k \in \mathcal{K}}  \overline{\widehat{\psi_\ell}\left((C_{k}^{\ast}D_{j}^{\ast}\gamma\right)} \widehat{\phi_\ell}\left(C_{k}^{\ast}D_{j}^{\ast}(\gamma+\alpha)\right) =0 \quad \mbox{ for a.e. } \gamma \in \mathbb{R}^d,
	\end{equation*}
	where for a matrix $C\in \mathrm{GL}_d(\mathbb{R})$, we denote its inverse transpose by $C^{\ast} := (C^{{T}})^{-1}$.
\end{prop}
As a specific instance, wavelet systems with composite dilations correspond to the classical shearlet system. 
Here, we focus on $L^2(\mathbb{R}^2)$ for simplicity; however, we refer the reader to \cite[Section 3.4]{GuoLabateLimWeissWilson} for a detailed treatment of shearlet systems in the more general setting of $L^2(\mathbb{R}^d)$. We define
\[
P = \begin{pmatrix}
	1 & 1 \\
	0 & 1
\end{pmatrix}
\quad\text{and}\quad
Q = \begin{pmatrix}
	4 & 0 \\
	0 & 2
\end{pmatrix}
\]
Let $\Gamma=R\Z^2$  for $R\in GL_2(\R)$. Consider the pair $\big( \{P^k Q^j\}_{k,j \in \mathbb{Z}}, \Gamma \big)$.  
The corresponding wavelet system consists of functions of the form
\[
\{ D_{P^k Q^j} \, T_\lambda \, \psi_\ell \}_{ k,j \in \mathbb{Z},\, \lambda \in \Gamma,\, \ell \in \mathcal{I} }
\]
 is referred to as a \textit{classical shearlet system} in $L^2(\mathbb{R}^2)$. A shearlet system automatically satisfies the $\alpha$-LIC if it has the CC-condition.  Hence, the next result follows directly from Proposition~\ref{prop_orthogonal_CWS}.
\begin{prop}\label{prop_char_orthogonal_CSS}
	Let   $	\{ D_{P^k Q^j} T_\la \psi_\ell \}_{k\, j \in \Z,\, \la \in \Gamma,\, i \in \mathcal{I}}$ and $\{ D_{P^k Q^j} T_\la \phi_\ell \}_{k\, j \in \Z,\, \la \in \Gamma,\, \ell \in \mathcal{I}}$ be  two  classical shearlet Bessel systems  (frames) in $L^2(\R^2)$. Then, these systems   are pairwise orthogonal  if and only if 
	for each $m \in \Z$, $q \in ( R^{\ast}\Z^2 \setminus QR^{\ast}\Z^2),$ we have
	\begin{equation*}
		\frac{1}{s(\Gamma_{j})} \sum_{n=0}^{\infty}\sum_{k \in \Z}\sum_{\ell \in \mathcal{I}}  \overline{\widehat{\psi_\ell}\left((P^{k})^{\ast}Q^{n+m}\gamma\right)} \widehat{\phi_\ell}\left((P^{k})^{\ast}Q^{n}(Q^m\gamma+q)\right)  =0  \mbox{ for a.e. } \gamma \in \R^2
	\end{equation*}
	and 
	\begin{align*}
		\frac{1}{s(\Gamma_{j})} \sum_{j\in \Z}\sum_{k \in \Z}\sum_{\ell \in \mathcal{I}}  \overline{\widehat{\psi_\ell}\left((P^{k})^{\ast}Q^{-j}\gamma\right)} \widehat{\phi_\ell}\left((P^{k})^{\ast}Q^{-j}\gamma\right)  =0  \mbox{ for a.e. } \gamma \in \R^2.
	\end{align*}
\end{prop}
We now turn our attention to cone-adapted shearlet systems, which will be the focus of the remainder of this subsection. To introduce these systems, we define $ P_1=P \, Q_1=Q,$,
\[
P_2 = \begin{pmatrix}
	1 & 0 \\
	1 & 1
\end{pmatrix}
\quad \text{and} \quad
Q_2 = \begin{pmatrix}
	2 & 0 \\
	0 & 4
\end{pmatrix}.
\]
Consider generating functions $\phi$ and $\psi_\ell$ in $L^2(\mathbb{R}^2)$ for $\ell = 1,2$, together with $\Gamma_\ell = R_\ell \mathbb{Z}^2$, $\ell = 0,1,2$, is full-rank lattices.  
The corresponding cone-adapted shearlet system is defined by
\[ \left\{ T_\la \phi \right\}_{\la \in \Gamma_0} \cup \left\{ D_{P_\ell^k Q_\ell^j} T_\la \psi_\ell \right\}_{\ell \in \{1,2\}, \; \la \in \Gamma_\ell, \; k \in \{-K_j, \ldots, K_j\},\; j \in \mathbb{N}_0}, \]
where $K_j \in \mathbb{N}_0$ for each $j$, typically chosen as  $K_j = 2^j \pm 1$ or $K_j = 2^j$.   
For simplicity, we take all lattices to coincide, $\Gamma_\ell = \Gamma = C \mathbb{Z}^2$ for $\ell = 0,1,2$, where $R \in \mathrm{GL}_2(\mathbb{R})$ is chosen so that $R^{\mathrm{T}} Q_\ell R^{\mathrm{T}}$ has integer entries for $\ell=1,2$.  
Since shearlet systems typically satisfy the $\alpha$-LIC under standard assumptions (see \cite{LVV}), Proposition~\ref{prop_char_orthogonal_CSS} can be applied directly.
\begin{prop}\label{prop_char_orthogonal_CASS}
	Let  $$	\left\{ T_\la \phi^{(1)} \right\}_{\la \in \Gamma_0} 
	\cup 
	\left\{ D_{P_\ell^k Q_\ell^j} T_\la \psi_\ell^{(1)} \right\}_{\ell \in \{1,2\}, \; \la \in \Gamma_\ell, \; k \in \{-K_j, \ldots, K_j\},\; j \in \mathbb{N}_0}$$ 
	and 
	$$\left\{ T_\la \phi^{(2)} \right\}_{\la \in \Gamma_0} 
	\cup 
	\left\{ D_{P_\ell^k Q_\ell^j} T_\la \psi_\ell^{(2)} \right\}_{\ell \in \{1,2\}, \; \la \in \Gamma_\ell, \; k \in \{-K_j, \ldots, K_j\},\; j \in \mathbb{N}_0}$$ 
	 be two cone-adapted shearlet Bessel systems  (frames) in $L^2(\R^2)$. Then, these  are pairwise orthogonal  if and only if we have
	\begin{equation*}
		\overline{\widehat{\phi^{(1)}}(\g)} \widehat{\phi^{(2)}}(\g)+ 
		\sum_{\ell \in \{1,2\}} \sum_{j=0}^{\infty} \sum_{k=-K_j}^{K_j}
		\overline{\widehat{\psi_i^{(1)}}((P_\ell^\ast)^k Q_\ell^{-j} \g)} \widehat{\psi_\ell^{(2)}}((P_\ell^\ast)^k Q_\ell^{-j} \g) =0\mbox{ for a.e. } \gamma \in \R^2
	\end{equation*}
	and
	\begin{equation*}
		\overline{\widehat{\phi^{(1)}}(\g)} \widehat{\phi^{(2)}}(\g+\al)+ 
		\sum_{\ell \in \{1,2\}} \sum_{j=0}^{m_\ell} \sum_{k=-K_j}^{K_j}
		\overline{\widehat{\psi_\ell^{(1)}}((P_\ell^\ast)^k Q_\ell^{-j} \g)} \widehat{\psi_\ell^{(2)}}((P_\ell^\ast)^k Q_\ell^{-j} \g+\al)=0 \mbox{ for a.e. } \gamma \in \R^2.
	\end{equation*}
	Here for each $\ell \in \{1,2\}$, any nonzero element $\alpha \in \Gamma^\perp$ can be uniquely expressed as
$\alpha = Q_\ell^{m_\ell} q_\ell,$ where $m_\ell$ is nonzero and $q_\ell \in \Gamma^\perp \setminus Q_\ell \Gamma^\perp$.
\end{prop}

This section provided a characterization of pairwise orthogonal frames with GTI structures, generated by translations over families of closed, co-compact subgroups of $G$. The subgroup families associated with each system may differ. As an application, we derive necessary and sufficient conditions for the orthogonality of structured systems such as Gabor, wavelet, and shearlet frames, and we also characterize GTI tight frames. Building upon these characterization results, the next section focuses on the explicit construction of pairwise orthogonal frames.

\section{Explicit Construction of Pairwise Orthogonal Parseval Frames}\label{sec4}	
	We begin in Subsection~\ref{subsec_4.1} by constructing a pair of GTI systems using filters. Next, Subsection~\ref{subsec_sufficent_condtions_for_UCP} presents sufficient conditions ensuring that the constructed GTI systems are  Parseval frames and satisfy the $\infty$-UCP with Clader\'on sum $1$. (see Theorem~\ref{thm_UCP}).  We present an explicit example based on $B$-splines as generators in $\ell^2(\Z)$. As an application of this result, Theorem \ref{thm_QEP} provides the Oblique extension principle (OEP). The subsequent analysis, found in Subsection~\ref{subsec_pairwise_orthogonal}, demonstrates that the resulting Parseval frames are pairwise orthogonal (see Theorem~\ref{thm_pairwise_orthogonal_frames}). Finally, we introduce a general method to construct $N$ pairwise orthogonal Parseval frames (see Theorem~\ref{thm_construction_of_N_pair_OF}).
	\subsection{Construction of  GTI systems via filters}\label{subsec_4.1}
To construct the GTI systems, we follow the general setup of \cite{CG2021}, incorporating some additional assumptions that will be needed in the remainder of this paper. We fix the objects introduced below throughout this section.

  The index set $\mathcal J$ may be any integer interval, i.e., \(\mathbb Z, \quad \{j_0,j_0+1,\dots\}, \quad \{j_0,\dots,j_1\}, \quad \{-\infty,\dots,j_1\},\)
	where $j_0< j_1$. Consider closed co-compact subgroups $\{\Gamma_j\}_{j \in \mathcal{J}}$ of a  LCA group $G$ such that  
	\begin{equation}\label{GTI_nested}
		\cdots \subset \Gamma_j \subset \Gamma_{j+1} \subset \Gamma_{j+2} \subset \cdots,
	\end{equation}
	where the quotient groups $\Gamma_{j+1}/\Gamma_j$ are finite with cardinality $d_j \in \mathbb{N}$. 
	The annihilators form the decreasing chain  
	\[
	\cdots \supset \Gamma_j^\perp \supset \Gamma_{j+1}^\perp \supset \Gamma_{j+2}^\perp \supset \cdots,
	\]
	satisfying $|\Gamma_{j+1}/\Gamma_j| = |\Gamma_j^\perp/\Gamma_{j+1}^\perp| = d_j$ (see \cite{HRJDS}). 
	We write $\Omega_j$ for a fundamental domain of $\Gamma_j^\perp$. Moreover, for each $j$ one can select a sequence $\{v_{j,\ell}\}_{\ell=1}^{d_j} \subset \Gamma_j^\perp$, with $v_{j,1}=0$, such that  
	\begin{equation}\label{cosetpartition}
		\Gamma^{\perp}_{j}=\bigcup \limits_{\ell=1}^{d_j}(v_{j,\ell}+ \Gamma^{\perp}_{j+1}),\,\, (v_{j,\ell}+ \Gamma^{\perp}_{j+1})\cap (v_{j,\ell'}+ \Gamma^{\perp}_{j+1})=\emptyset \, \mbox{ for}~ \ell\neq \ell'.
	\end{equation}  
	
	\noindent{\bf Generating functions and filters:} Consider a sequence $\{\Phi_j\}_{j \in \mathcal J} \subset L^2(\widehat G)$ such that, 
	for each $j$, there exists a $\Gamma_{j+1}^\perp$-periodic function $H_{j+1}\in L^\infty(\Omega_{j+1})$ with  
	\begin{equation}\label{phi_relation}
		\Phi_j(\gamma) = H_{j+1}(\gamma)\Phi_{j+1}(\gamma), \qquad \text{a.e. }\gamma \in \widehat G.
	\end{equation}
	For $i\in\{1,2\}$ and $m=1,\dots,s_j$, we define  
	\begin{equation}\label{psi_def}
		\Psi_j^{(i)(m)}(\gamma) := G_{j+1}^{(i)(m)}(\gamma)\, \Phi_{j+1}(\gamma),
	\end{equation}
	where the filters $G_{j+1}^{(i)(m)}\in L^\infty(\Omega_j)$ are assumed to be $\Gamma_j^\perp$-periodic.  The corresponding time-domain generators are  
	\begin{equation}\label{gtime}
		g_{(m,j)}^{(i)} := \mathcal F^{-1}\Psi_j^{(i)(m)}.
	\end{equation}  
	
	\noindent{\bf GTI systems and objective:} With $P_j := \{(m,j): m=1,\dots, s_j\}$, the GTI systems  
	\begin{equation}\label{GTI eq6}
		\gtione \qquad \text{and} \qquad \gtitwo
	\end{equation}
	are the principal objects of study. We intend to establish requirements on the generators $\Phi_j$ and filters $G_{j}^{(i)(m)}$ that guarantee these systems constitute \emph{pairwise orthogonal Parseval frames} in $L^2(G)$. 
	In particular, we verify that they satisfy the dual $1$-UCP, cf. Theorem \ref{thm_characterization_orthogonal_frame}.  
	
	\noindent{\bf Matrix formulation:} For each $i\in\{1,2\}$ and $j \in \mathcal J$, we introduce the matrices  
	\begin{equation}\label{matrices}
		\mathfrak B_j^{(i)}(\gamma) := 
		\big(G_{j+1}^{(i)(m)}(\gamma+v_{j,n})\big)_{\substack{0\le m \le s_j \\ 1\le n\le d_j}},
		\qquad
		\widetilde{\mathfrak B}_j^{(i)}(\gamma) := 
		\big(G_{j+1}^{(i)(m)}(\gamma+v_{j,n})\big)_{\substack{1\le m \le s_j \\ 1\le n\le d_j}},
	\end{equation}
	for almost every $\gamma\in\Omega_j$, with the convention $G_{j+1}^{(i)(0)}:=H_{j+1}$.  
	We present the argument for $\J=\mathbb Z$; the remaining cases require only minor modifications. 
	For bounded intervals, one restricts $j$ accordingly so that \eqref{phi_relation}--\eqref{matrices} are well defined.

	\subsection{Sufficient conditions for the GTI systems satisfying the dual $\infty$-UCP}
	\label{subsec_sufficent_condtions_for_UCP}

    We now state the following standing assumptions:
	\begin{itemize}
		\item[($\mathcal{N}_1$)]	For every compact set $S$ in  $\widehat{G}\setminus B$  and $\epsilon > 0,$ there exists $J_1 \in \mathcal{J}$ such that for all $j \geq J_1, \ j \in \mathcal{J},$ $$ \left|\frac{1}{s(\Gamma_{j})}|\Phi_j(\gamma)|^2-1\right| \leq \epsilon , \,\, \forall \, \gamma \in S.$$
		\item[($\mathcal{N}_2$)]
		For every compact set $S$ in $\widehat{G}\setminus B$ and $\epsilon > 0,$ there exists $J_2 \in \mathcal{J}$ such that for all $j \leq J_2, \ j \in \mathcal{J},$ $$ \frac{1}{\sqrt{s(\Gamma_{j})}}|\Phi_j(\gamma)| \leq \epsilon , \,\, \forall \, \gamma \in S.$$
	\end{itemize}
The following result provides sufficient conditions under which the GTI system
  $\displaystyle\bigcup_{j \in \mathcal{J}}\{T_{\lambda}g_p^{(i)}\}_{\lambda \in \Gamma_{j}, p \in P_j}$, defined in \eqref{GTI eq6}, satisfies the $\infty$-UCP, forms a Parseval frame, and possesses a Calderón sum equal to one.
  \begin{thm}\label{thm_UCP}
	Assume that conditions $(\mathcal{N}_1)$--$(\mathcal{N}_2)$ hold. 
In addition, suppose that  for every compact set $S \in \hg$, there exist a $J \in \J$, such that $\mu_{\widehat{G}}\big( (\omega + S) \cap (\omega' + S) \big) = 0 \mbox{ for } \omega \neq \omega'  \mbox{ and }  \omega, \omega' \in \Gamma_J^\perp.$ Further, for each $j \in \mathcal{J}$, let the matrix-valued function $\mathfrak{B}^{(i)}_j(\gamma)$ defined in \eqref{matrices} satisfy
\begin{equation}\label{PFC}
	({\mathfrak{B}^{(i)}_j}(\gamma))^{\ast} \mathfrak{B}^{(i)}_j(\gamma)=\frac{s(\Gamma_{j})}{s(\Gamma_{j+1})}I_{d_j}\,  \mbox{  for a.e. } \gamma \in \Omega_j, 
\end{equation}
where $({\mathfrak{B}^{(i)}_j}(\gamma))^{\ast}$ denotes adjoint of $\mathfrak{B}^{(i)}_j(\gamma)$. Then, for each $i \in \{1,2\}$, the GTI system $\displaystyle\bigcup_{j \in \mathcal{J}}\{T_{\lambda}g_p^{(i)}\}_{\lambda \in \Gamma_{j}, p \in P_j}$ {(defined in \eqref{GTI eq6})}
	\begin{itemize}
		\item[(i)] satisfies the $\infty$-UCP,
		\item[(ii)] is a Parseval frame for $L^2(G)$,
		\item[(iii)] has the Calder\'on sum $1$, i.e., $\sum_{j \in \mathcal{J}} \frac{1}{s(\Gamma_j)} 
		\sum_{p \in P_j} \big|\widehat{g_p^{(i)}}(w)\big|^2 = 1
		\ \text{for a.e. } w \in \widehat{G}.$
	\end{itemize}
\end{thm}

    Before we proceed to prove the above theorem, we first require the following two Lemmas~\ref{Lemma_first_for_UCP} and \ref{lem_2_UCP}, which are essential for establishing our results. These lemmas are inspired by the proof of the unitary extension principle for $L^2(\mathbb{R})$. Specifically, a similar versions of  Lemmas~\ref{Lemma_first_for_UCP} and  \ref{lem_2_UCP}(ii)  also appear in \cite{CG2021}, while part (i) and (iii) of Lemma~\ref{lem_2_UCP} are presented here for the first time in the setting of LCA groups, as a generalization from the $L^2(\mathbb{R})$ setting.

Before stating the lemmas, we define the function $\wphij$ by
$$\wphij(x):= \dwphij$$	
 for each $j \in \mathcal{J}$ and $x \in G$.
\begin{lem}\label{Lemma_first_for_UCP}
	Let $i \in \{1, 2\}$, and consider the GTI system $\gtii$ defined in~\eqref{GTI eq6}. Suppose that, for some integers $j_0, J$ with $j_0 \leq J$, the matrix valued functions $\mathfrak{B}^{(i)}_j(\gamma)$, defined in~\eqref{matrices}, satisfy
	\begin{equation*}
		({\mathfrak{B}^{(i)}_j}(\gamma))^{\ast} \mathfrak{B}^{(i)}_j(\gamma)=\frac{s(\Gamma_{j})}{s(\Gamma_{j+1})}I_{d_j}\quad  \mbox{for a.e. } \gamma \in \Omega_j,  
	\end{equation*}
	for all $j = j_0, \ldots, J$. Then the following identity holds:
	\begin{align*}
		\displaystyle\sum_{j=j_0}^{J}\wgpij(x)=\wphiJ(x)-\wphijnote(x) \quad \mbox{for } x \in G,
	\end{align*}
	where $\wgpij$ is defined in \eqref{eq_wfgi}.
\end{lem}

\bp
	Since the Fourier transform preserves norms, we have
	\begin{align}
		\wgponejnote (x)&=\dwgponejnote \nonumber \\
		&=\Fdwgponejnote \nonumber\\
		&=\FEdwgponejnote. \label{eq_0_lem}
	\end{align}
	Substituting  $P_{j_0}=\{(m,j_0): m=1,2,\ldots, s_{j_0}\}$  and using the expression for $\widehat{\gpi}$ in the right-hand side of (\ref{eq_0_lem}),  we obtain
	\begin{align}
		\wgponejnote(x) &=\displaystyle\sum\limits_{m=1}^{s_{j_0}} \displaystyle\int\limits_{\Gamma_{j_0}} \left| \langle \F(T_x f), \Mg \Psi^{(i)(m)}_{j_0} \rangle \right|^2 \ d\mu_{\Gamma_{j_0}}(\la).
	\end{align}
	Similarly, the functions $\wphijnote(x)$ and $\wphijnoteplusone(x)$ can be expressed as 
	\begin{align}
		\wphijnote(x)=& \displaystyle\int\limits_{\Gamma_{j_0}} \left| \langle \F(T_x f), \Mg \Phi_{j_0} \rangle \right|^2 \ d\mu_{\Gamma_{j_0}}(\la)
	\end{align}
	and
	\begin{align}
		\wphijnoteplusone(x)=& \displaystyle\int\limits_{\Gamma_{j_0+1}} \left| \langle \F(T_x f), \Mg \Phi_{j_0+1} \rangle \right|^2 \ d\mu_{\Gamma_{j_0+1}}(\la).
	\end{align}
	Now using the matrix condition $({\mathfrak{B}^{(i)}_{j_0}}(\gamma))^{\ast} \mathfrak{B}^{(i)}_{j_0}(\gamma)=\frac{s(\Gamma_{j_0})}{s(\Gamma_{j_0+1})}I_{d_{j_0}}\,\,  \mbox{  for a.e. } \gamma \in \Omega_{j_0}$ and following a similar argument as in \cite[Lemma $3.1$]{CG2021}, we obtain 
	\begin{align*}
		\displaystyle\sum\limits_{m=1}^{s_{j_0}} \int\limits_{\Gamma_{j_0}} \left| \langle \F(T_x f), \Mg \Psi^{(i)(m)}_{j_0} \rangle \right|^2 \ d\mu_{\Gamma_{j_0}}(\la)
		=&\int\limits_{\Gamma_{j_0+1}} \left| \langle \F(T_x f), \Mg \Phi_{j_0+1} \rangle \right|^2 \ d\mu_{\Gamma_{j_0+1}}(\la)\\
		&- \int\limits_{\Gamma_{j_0}} \left| \langle \F(T_x f), \Mg \Phi_{j_0} \rangle \right|^2 \ d\mu_{\Gamma_{j_0}}(\la).
	\end{align*}
	Therefore, we conclude that $\wgponejnote(x)=\wphijnoteplusone(x)-\wphijnote(x)$. Applying the same reasoning recursively for $j \in \{ j_0 + 1, \ldots, J \}$, we obtain
	\begin{align*}
		\sum_{j = j_0}^{J} \wgpij(x) 
		=&\left( \wphijnoteplusone(x) - \wphijnote(x) \right) + \left( \wphijnoteplustwo(x) - \wphijnoteplusone(x) \right) 
	+ \cdots + \left( \wphiJ(x) - \wphiJminusone(x) \right) \\
		= &\wphiJ(x) - \wphijnote(x).
	\end{align*}
	This completes the proof.
\ep

\begin{lem}\label{lem_2_UCP}
	Under the assumptions of Theorem \ref{thm_UCP}, for any $f \in \mathcal{D}_B$, $x \in G$ and given $\epsilon > 0$, the following statements hold:
	\begin{itemize}
		\item[(i)] There exists an integer $J_2 \in \Z$ such that for all $j \in \J$ with $j \leq J_2$, we have
		\begin{align*}
			\wphij(x) \leq \epsilon \norm{f}^2.
		\end{align*}
		
		\item[(ii)] There exists an integer $J_1 \in \Z$ such that for all $j \in \J$ with $j \geq J_1$, we have
		\begin{align*}
			(1 - \epsilon)\|f\|^2 \leq \wphij(x) \leq (1 + \epsilon)\|f\|^2.
		\end{align*}
		\item[(iii)]  	Moreover, the following identity holds:
		\[
		\sum_{j \in \Z} \wgpij (x)= \|f\|^2.
		\]
	\end{itemize}
\end{lem}

\bp
	By following a similar approach as in Lemma \ref{Lemma_first_for_UCP}, we can express $\wphij(x)$ as  
	\begin{align}\label{eq_1_lemma_2_UCP}
		\wphij(x)=& \int\limits_{\Gamma_{j}} \left| \langle \F(T_x f), \Mg \Phi_{j} \rangle \right|^2 \ d\mu_{\Gamma_{j}}(\la). 
	\end{align}
	Let us define $S:= \mbox{supp} \ \F(T_x f).$ Since $\widehat{T_x f}(\gamma)=\gamma(-x)\widehat f(\gamma)$, the support $S$ is independent of $x$ and equals $\operatorname{supp}\widehat f$. For $j\in \J$ and $\omega \in \Gjp$, we define 
	\[
	S_{j,\omega} := \{ \gamma \in \Omega_j : \omega + \gamma \in S \}.
	\]
	Note that the set $S_{j,\omega}$ is measurable and satisfies $S_{j,\omega} = \Omega_j \cap (S - \omega)$. Moreover, using the tiling property in~\eqref{eq_1_intro}, we obtain
\[
S = \bigcup_{\omega \in \Gamma_j^\perp} (\omega + S_{j,\omega}),
\]
which gives a disjoint decomposition up to a set of measure zero.
	Now, by \cite[Proposition 2.2]{CG2021}, the right-hand side of~\eqref{eq_1_lemma_2_UCP} can be rewritten as: 
	\begin{align*}
		\wphij(x) = \frac{1}{s(\Gj)} \int\limits_{\Omega_j}  
		\left| \sum\limits_{\omega \in \Gjp} \F(T_x f)(\omega + \gamma) \overline{\Phi_j(\omega + \gamma)}  \right|^2 \, d\mu_G(\gamma).
	\end{align*}
	Observe that in the integral above, nonzero terms appear only for $\gamma \in \Omega_j$ whenever there exists some $\omega' \in \Gjp$ with $\omega' + \gamma \in S$. In other words, the contributions are restricted to 
\(\gamma \in \bigcup_{\omega' \in \Gjp} S_{j,\omega'}.\)
	Hence, we can write	  
	\begin{align}
		\wphij(x)
		&= \frac{1}{s(\Gj)}  \int\limits_{\left[ \bigcup_{\omega' \in \Gjp} S_{j,\omega'} \right]}
		\left| \sum\limits_{\omega \in \Gjp} \F(T_x f)(\omega + \gamma) \overline{\Phi_j(\omega + \gamma)} \right|^2 \, d\mu_G(\gamma) \nonumber\\
		&\leq \frac{1}{s(\Gj)} \sum_{\omega' \in \Gjp} \int\limits_{S_{k,\omega'} }
		\left| \sum\limits_{\omega \in \Gjp} \F(T_x f)(\omega + \gamma) \overline{\Phi_j(\omega + \gamma)} \right|^2 \, d\mu_G(\gamma).  \label{eq_2_lemma_2_UCP}
	\end{align}
We	observe that in \eqref{eq_2_lemma_2_UCP}, the summation over $\omega$, for a fixed $\omega' \in \Gjp$, only the term with $\omega = \omega'$ contributes within $S_{j,\omega'}$. Thus, the expression simplifies to
	\begin{align}
		\wphij(x) &\leq \frac{1}{s(\Gj)} \sum_{\omega' \in \Gjp} \int\limits_{S_{k,\omega'} }
		\left|  \F(T_x f)(\omega' + \gamma) \overline{\Phi_j(\omega' + \gamma)} \right|^2 \, d\mu_G(\gamma) \nonumber\\
		&=\frac{1}{s(\Gj)} \sum_{\omega' \in \Gjp} \int\limits_{\omega' +S_{k,\omega'} }
		\left|  \F(T_x f)( \gamma) \overline{\Phi_j( \gamma)} \right|^2 \, d\mu_G(\gamma) \nonumber\\
		&=\frac{1}{s(\Gamma_{j})} \int\limits_{S} \left | \F(T_x f)(\gamma) \Phi_{j}(\gamma) \right |^2 \ d \mu_{\hg}(\gamma).
	\end{align}
	Now, using the assumption $(\mathcal{N}_2)$, for a given $\epsilon > 0$ there exists $J_2 \in \mathcal{J}$ such that for all $j \in \mathcal{J}$ with $j \leq J_2$, we have
	\begin{align*}
		\wphij(x) \leq \epsilon \int\limits_{S} \left | \F(T_x f)(\gamma)  \right |^2 \ d \mu_{\hg}(\gamma)=\epsilon \norm{\F(T_xf)}^2= \epsilon \norm{f}^2.
	\end{align*}
	This completes the proof of part (i).
	
	By assumption, for a compact set $S$ in $\widehat{G}$, there exists $J \in \mathbb{Z}$ such that for all $\omega, \omega' \in \Gamma_J^\perp$ with $\omega \neq \omega'$, 
	\[
	\mu_{\hg}\big( (\omega + S) \cap (\omega' + S) \big) = 0.
	\]  
	Choose $J_1 \in \mathbb{Z}$ such that $J_1 > J$ and the assumption $(\mathcal{N}_1)$ is satisfied for all $j \geq J_1$. Then, following similar steps as in the proof of \cite[Lemma~3.2]{CG2021}, we obtain
	\begin{align*}
		\wphij(x)=\frac{1}{s(\Gamma_{j})} \int\limits_{S} \left | \F(T_x f)(\gamma) \Phi_{j}(\gamma) \right | \ d \mu_{\hg}(\gamma) \mbox{ for all } j \geq J_1.
	\end{align*}
	Now, using the choice of $J_1$ and the assumption $(\mathcal{N}_1)$, it follows that
	\[	(1 - \epsilon)  \norm{\F(T_x f)}^2 	\leq \wphij(x)	\leq (1 + \epsilon)  \norm{\F(T_x f)}^2 \mbox{ for all } j \geq J_1.	\]
	Since $\norm{\F(T_x f)}=\norm{T_xf}=\norm{f}$, the proof of part (ii) is complete.
	
	Now, by Lemma \ref{Lemma_first_for_UCP}, we have 
	\begin{align*}
		\displaystyle\sum_{j=j_0}^{J}\wgpij(x)=\wphiJ(x)-\wphijnote(x).
	\end{align*}
	Taking the limit as $j_0 \to -\infty$ on both sides and using part (i), which gives $\wphijnote(x) \to 0$ as $j_0 \to -\infty$, we obtain 
	\begin{align*}
		\displaystyle\sum_{j=-\infty}^{J}\wgpij(x)=\wphiJ(x).
	\end{align*}
	Next, applying part (ii), we know that for all $J \geq J_1$, we have
	\[
	(1 - \epsilon)  \|f\|^2 \, d\mu_G(\gamma)
	\leq \displaystyle\sum_{j=-\infty}^{J}\wgpij(x)
	\leq (1 + \epsilon)  \|f\|^2 \, d\mu_G(\gamma).
	\]
	Since this estimate holds for all $J \geq J_1$ and $\epsilon > 0$ is arbitrary, we conclude that
	\[
	\sum_{j \in \Z} \wgpij (x)= \|f\|^2
	\] which completes the proof of part (iii).
\ep

\noindent{\bf Proof of Theorem \ref{thm_UCP}}
	Let $\epsilon > 0$ be arbitrary. By Lemma \ref{lem_2_UCP}((i)--(ii)), there exist integers $J_1, J_2 \in\Z$ such that  
	\begin{align}
		\wphij (x)\leq \epsilon \norm{f}^2\quad \text{for all } j \leq J_2 \mbox{ and } (1 - \epsilon)  \|f\|^2 
		\leq \wphij(x)
		\leq (1 + \epsilon)  \|f\|^2  \quad \text{for all } j \geq J_1. \label{eq_1_thm_infity_UCP}
	\end{align}
	Define the interval $\J'=[J_2, J_1]\cap \Z$. We now estimate the error between the full sum and the partial sum
	\begin{align*}
		\left |  \sum_{j \in \Z} \wgpij (x)-\sum_{j \in \J'} \wgpij (x) \right |&= \sum_{j \in \Z} \wgpij (x)- \displaystyle\sum_{j=J_2}^{J_1}\wgpij(x)\\
		&=\|f\|^2 -\left[\wphiJoneplusone(x)-\wphiJtwo(x)\right],
	\end{align*}
	where the last equality follows from Lemma \ref{Lemma_first_for_UCP}.  Applying \eqref{eq_1_thm_infity_UCP} on the above expression gives
	\begin{align*}
		\left |  \sum_{j \in \Z} \wgpij (x)-\sum_{j \in \J''} \wgpij (x) \right |&=\|f\|^2 - \wphiJoneplusone(x)+\wphiJtwo(x)\\
		&\leq\norm{f}^2-(1 - \epsilon)  \|f\|^2+ \epsilon \|f\|^2\\
		&=2 \epsilon \norm{f}^2.
	\end{align*}
    Since each $\wgpij$ is non-negative, controlling the two tails outside $[J_1,J_2]$ also controls the error for any finite set containing $[J_1,J_2]$. 
	As $\epsilon>0$ is arbitrary, thus  $\sum_{j \in \Z}\wgpij$ converges uniformly to $\wgpi$. This proves (i). 
	
	By Lemma \ref{lem_2_UCP} (iii), we have 
	\[
	\sum_{j \in \Z} \wgpij (0)=\sum_{j \in \Z}\int\limits_{\Gamma_{j}} \left| \langle \F( f), \Mg \Phi_{j} \rangle \right|^2  d\mu_{\Gamma_{j}}(\gamma)= \|f\|^2.
	\]
Thus, we obtain
\[
\sum_{j\in\mathbb Z}\int_{P_j}\int_{\Gamma_j}
|\langle f,T_\lambda g_p^{(i)}\rangle|^2
\,d\mu_{\Gamma_j}(\lambda)d\mu_{P_j}(p)
=\|f\|^2,
\] and hence the GTI system $\gtiiz$  is a Parseval frame for $\ltg$. Hence, (ii) is proved. 

Now, since (i) and (ii) hold, then applying Theorem \ref{thm_char_Parseval_frames} yields that for each $\alpha \in  \umzero$,
	\begin{align*}
		\displaystyle\sum_{j \in \mathcal{J} : \alpha \in \Gamma^\perp_j} \frac{1}{s(\Gamma_{j})} \int\limits_{P_j}  \overline{\widehat{g_{p}^{(i)}}(\gamma)} \widehat{g_{p}^{(i)}}(\gamma+\alpha)  d\mu_{P_j}(p) =0 \mbox{ for a.e. } \gamma \in \widehat{G}
	\end{align*}
	and 
	\begin{equation}\label{eq_Cladronsum_one}
		\displaystyle\sum_{j \in \mathcal{J} } \frac{1}{s(\Gamma_{j})} \int\limits_{P_j}  \overline{\widehat{g_{p}^{(i)}}(\gamma)} \widehat{g_{p}^{(i)}}(\gamma)  d\mu_{P_j}(p) =1  \mbox{ for a.e. } \gamma \in \widehat{G}.
	\end{equation}
	Therefore, by \eqref{eq_Cladronsum_one}, the Calder\'on sum of the system $\gtiiz$ equals $1$ for each $i \in \{1,2\}$, proving (iii).
\ep

\begin{rem}\label{rem_comparision with_CG_UEP}
	As a special case $\mathcal{J}=\{j\}_{j=j_0}^{\infty}$, the GTI system defined in Theorem~\ref{thm_UCP} take the form  
	\bee\label{eq_CG_System}  
	\left\{T_{\lambda}\mathcal{F}^{-1}\Phi_{j_0}\right\}_{\lambda \in \Gamma_{j_0}} \cup \bigcup_{j = j_0}^{\infty} \left\{T_{\lambda}g_p^{(i)}\right\}_{\lambda \in \Gamma_{j},\; p \in P_j}.  
	\ene  
	Here, $\mathcal{F}^{-1}$ stands for the inverse Fourier transform. Christensen and Goh established in  \cite[Theorem~3.3]{CG2021} that systems of the form \eqref{eq_CG_System} constitute Parseval frames. The result in Theorem~\ref{thm_UCP} also establishes that the Calderón sum equals one, together with the $\infty$-UCP property. Moreover, we extend the framework by allowing the index set $\mathcal{J}$ to take the forms $\{j\}_{j=-\infty}^{j_1}$ or $\mathbb{Z}$, which were not considered in \cite{CG2021}.
\end{rem}

Using Theorem~\ref{thm_UCP}, we construct a pair of GTI Parseval frames that satisfy the $\infty$-UCP, with $B$-splines serving as generating functions. We return to these GTI systems in Example \ref{ex_pair_of_orthogonal_frames_by_B-spline}, where we demonstrate that these GTI systems are also pairwise orthogonal.
\begin{ex}\label{ex_B_spline}
	Let $G=\Z$ be the LCA group. For some $k \in \Z$, we  let $\J=\{-\infty, \cdots, -1,0,1,\cdots, k\}$. For each $j \in \mathcal{J}$, let $B_j$ denote the fourth-order $B$-spline at level $j$, defined by
	\begin{align}
		B_{j}(x):=\frac{1}{(2^{k-j})^{5/2}} \,\chi_{\{0,1, \cdots,  2^{k-j}-1\}}\ast \chi_{\{0,1, \cdots,  2^{k-j}-1\}}\ast \chi_{\{0,1, \cdots,  2^{k-j}-1\}}\ast \chi_{\{0,1, \cdots,  2^{k-j}-1\}}(x), \,\,\, x \in \Z.\nonumber
	\end{align} 
We define the generating function $\Phi_j$ as the Fourier transform of $B_j$. 
Then
	\begin{align*}
		\Phi_j(\gamma):=	\widehat{B_{j}}(\gamma)&=\frac{1}{(2^{k-j})^{7/2}}\frac{(1-e^{- 2 \pi i (2^{k-j}\gamma)})^{4}}{(1-e^{- 2 \pi i \gamma})^4}\\
		&= \frac{1}{(2^{k-j-1})^{7/2} }\frac{(1-e^{- 2 \pi i (2^{k-j-1}\gamma)})^{4}}{(1-e^{- 2 \pi i \gamma})^4} \frac{(1+e^{- 2 \pi i (2^{k-j-1}\gamma)})^{4}}{2^{7/2}}\\
		&=\widehat{B_{j+1}}(\gamma)H_{j+1}(\gamma)= H_{j+1}(\gamma) \Phi_{j+1}(\gamma),
	\end{align*}
	where  $H_{j+1}(\gamma)= \frac{(1+e^{- 2 \pi i (2^{k-j-1}\gamma)})^{4}}{2^{7/2}} \in L^{\infty}[0, \frac{1}{2^{k-j}})$ is a  $2^{-k+j+1}\mathbb{Z}_{k-j-1}$-periodic function. Next, for $m \in \{1,2,\dots,8\}$ and $i \in \{1,2\}$, 
we define the functions $\Psi_j^{(i)(m)} \in L^2(\mathbb{T})$ by  $\Psi^{(i)(m)}_j(\gamma)=G^{(i)(m)}_{j+1}(\gamma)\Phi_{j+1}(\gamma),$   where   $G_{j+1}^{(1)(m)}$ and $G_{j+1}^{(2)(m)}$ are given by 
	\begin{align*}
		G_{j+1}^{(1)(m)}(\gamma)=&a_{1m} (1+e^{- 2 \pi i (2^{k-j-1}\gamma)})^3 (1-e^{- 2 \pi i (2^{k-j-1}\gamma)})+a_{2m} (1+e^{- 2 \pi i (2^{k-j-1}\gamma)})^2 (1-e^{- 2 \pi i (2^{k-j-1}\gamma)})^2\\
		&+a_{3m} (1+e^{- 2 \pi i (2^{k-j-1}\gamma)}) (1-e^{- 2 \pi i (2^{k-j-1}\gamma)})^3+a_{4m} (1-e^{- 2 \pi i (2^{k-j-1}\gamma)})^4
	\end{align*}
	and
	\begin{align*}
		G_{j+1}^{(2)(m)}(\gamma)=&b_{1m} (1+e^{- 2 \pi i (2^{k-j-1}\gamma)})^3 (1-e^{- 2 \pi i (2^{k-j-1}\gamma)})+b_{2m} (1+e^{- 2 \pi i (2^{k-j-1}\gamma)})^2 (1-e^{- 2 \pi i (2^{k-j-1}\gamma)})^2\\
		&+b_{3m} (1+e^{- 2 \pi i (2^{k-j-1}\gamma)}) (1-e^{- 2 \pi i (2^{k-j-1}\gamma)})^3+b_{4m} (1-e^{- 2 \pi i (2^{k-j-1}\gamma)})^4,
	\end{align*}
	with
	\begin{align*}
		(a_{nm})=\begin{pmatrix}
			\frac{1}{2^4}  &\frac{1}{2^4} &\frac{1}{2^4} &\frac{1}{2^4} &\frac{1}{2^4} &\frac{1}{2^4} &\frac{1}{2^4} &\frac{1}{2^4} \\
			\frac{3^{1/2}}{2^{9/2}} &\frac{3^{1/2}}{2^{5}} +i \frac{3^{1/2}}{2^{5}}  &i\frac{3^{1/2}}{2^{9/2}} &-\frac{3^{1/2}}{2^{5}} +i \frac{3^{1/2}}{2^{5}} &-\frac{3^{1/2}}{2^{9/2}} &-\frac{3^{1/2}}{2^{5}} -i \frac{3^{1/2}}{2^{5}}  &-\frac{3^{1/2}}{2^{9/2}} &\frac{3^{1/2}}{2^{5}} -i \frac{3^{1/2}}{2^{5}}\\
			\frac{1}{2^4} & i \frac{1}{2^4} &-\frac{1}{2^4} &-i\frac{1}{2^4} &\frac{1}{2^4} &i\frac{1}{2^4} &-\frac{1}{2^4} &-i\frac{1}{2^4}\\
			\frac{1}{2^5} &-\frac{1}{2^{11/2}}+i\frac{1}{2^{11/2}} &-i\frac{1}{2^5} &\frac{1}{2^{11/2}}+i\frac{1}{2^{11/2}} &-\frac{1}{2^5} &\frac{1}{2^{11/2}}-i\frac{1}{2^{11/2}} &i\frac{1}{2^5}&-\frac{1}{2^{11/2}}-i\frac{1}{2^{11/2}}
		\end{pmatrix}
	\end{align*}
	and\\
	{\small
		$(b_{nm})
		=\begin{pmatrix}
			\frac{1}{2^4} &-\frac{1}{2^4} &\frac{1}{2^4} &-\frac{1}{2^4} &\frac{1}{2^4} &-\frac{1}{2^4} &\frac{1}{2^4} &-\frac{1}{2^4}\\
			\frac{3^{1/2}}{2^3} &-\frac{3^{1/2}}{2^{7/2}} -i \frac{3^{1/2}}{2^{7/2}} &\frac{3^{1/2}}{2^{7/2}} +i \frac{3^{1/2}}{2^{7/2}} &\frac{3^{1/2}}{2^{7/2}} -i \frac{3^{1/2}}{2^{7/2}} &-\frac{3^{1/2}}{2^3} &\frac{3^{1/2}}{2^{7/2}} +i \frac{3^{1/2}}{2^{7/2}} &-i\frac{3^{1/2}}{2^3} &-\frac{3^{1/2}}{2^{7/2}} +i \frac{3^{1/2}}{2^{7/2}}\\
			\frac{1}{2^4} &-i\frac{1}{2^4}&-\frac{1}{2^4} &i\frac{1}{2^4} &\frac{1}{2^4} &-i\frac{1}{2^4} &-\frac{1}{2^4} &i\frac{1}{2^4}\\
			\frac{1}{2^5} &\frac{1}{2^{11/2}}-i\frac{1}{2^{11/2}} &-i \frac{1}{2^5} &-\frac{1}{2^{11/2}}-i\frac{1}{2^{11/2}}  &-\frac{1}{2^5} &-\frac{1}{2^{11/2}}+i\frac{1}{2^{11/2}} &i\frac{1}{2^5} &\frac{1}{2^{11/2}}+i\frac{1}{2^{11/2}}
		\end{pmatrix}$}\\
	
\noindent	for $	1\leq n \leq 4$ and $	1\leq m\leq 8$.	 The matrix ${\mathfrak{B}^{(i)}_j}(\gamma)$ for $j \in \J$ and $i \in \{1,2\}$, is given by  
	$${\mathfrak{B}^{(i)}_j}(\gamma)=\begin{pmatrix}
		G_{j+1}^{(i)(0)}(\gamma)  &G_{j+1}^{(i)(1)}(\gamma)  	&\cdot &\cdot &\cdot	 &G_{j+1}^{(i)(8)}(\gamma) \\
		G_{j+1}^{(i)(0)}(\gamma+2^{j-k}) &G_{j+1}^{(i)(1)}(\gamma+2^{j-k}) &\cdot &\cdot &\cdot &G_{j+1}^{(i)(8)}(\gamma+2^{j-k})	
	\end{pmatrix}^{T},$$ 
	where   $G_{j+1}^{(i)(0)}= H_{j+1}.$ Next, we  prove that  $({\mathfrak{B}^{(i)}_j}(\gamma))^{\ast} \mathfrak{B}^{(i)}_j(\gamma)=2I_{2}$ for a.e. $\gamma \in [0, 2^{j-k})$, and for each $j \in \mathcal{J}$ and $i \in \{1,2\}$.  For this,  it is sufficient to show that 
	for a.e. $\gamma \in [0, 2^{j-k})$ and $\ell,\ell^{'}\in \{1,2\},$
	\begin{equation}\label{ex eq3}
		\displaystyle\sum_{m=0}^{8}\overline{G^{(i)(m)}_{j+1}(\gamma+\nu_{j,\ell})}G^{(i)(m)}_{j+1}(\gamma+\nu_{j,\ell^{'}})=2\delta_{\ell,\ell^{'}}, \mbox{ where } \nu_{j,1}=0 \mbox{ and } \nu_{j,2}=2^{j-k}.
	\end{equation}	
	We present the proof of \eqref{ex eq3} for $i=1$; the case $i=2$ follows by the same argument. First, suppose that $\ell=\ell'=1$. Then $\nu_{j,\ell}=\nu_{j,\ell'}=0$. Now
	\begin{align}
		\displaystyle\sum_{m=0}^{8}\overline{G^{(1)(m)}_{j+1}(\gamma+\nu_{j,\ell})}G^{(1)(m)}_{j+1}(\gamma+\nu_{j,\ell^{'}})=|H_{j+1}(\gamma)|^2 +\displaystyle\sum_{m=1}^{8}|G^{(1)(m)}_{j+1}(\gamma)|^2. \label{ex eq4}
	\end{align}
	Inserting the expressions for $H_{j+1}$ and $G_{j+1}^{(1)(m)}$ into \eqref{ex eq4} and simplifying, the right-hand side becomes
	\begingroup
	\allowdisplaybreaks
	\begin{align*}
		=&\frac{1}{2^7}	\left|(1+e^{- 2 \pi i (2^{k-j-1}\gamma)})^4\right|^2+ \sum_{m=1}^{8}|a_{1m}|^2\left|  (1+e^{- 2 \pi i (2^{k-j-1}\gamma)})^3 (1-e^{- 2 \pi i (2^{k-j-1}\gamma)})  \right|^2\\
		&+\sum_{m=1}^{8}|a_{2m}|^2\left|  (1+e^{- 2 \pi i (2^{k-j-1}\gamma)})^2 (1-e^{- 2 \pi i (2^{k-j-1}\gamma)})^2  \right|^2\\
		&+\sum_{m=1}^{8}|a_{3m}|^2\left|  (1+e^{- 2 \pi i (2^{k-j-1}\gamma)})^1 (1-e^{- 2 \pi i (2^{k-j-1}\gamma)})^3  \right|^2+\sum_{m=1}^{8}|a_{4m}|^2\left|   (1-e^{- 2 \pi i (2^{k-j-1}\gamma)})^4  \right|^2\\
		=&\frac{1}{2^7}	\left[\left|(1+e^{- 2 \pi i (2^{k-j-1}\gamma)})\right|^2\right]^4+\frac{1}{2^{5}} \left|  (1+e^{- 2 \pi i (2^{k-j-1}\gamma)})\right|^6 \left| (1-e^{- 2 \pi i (2^{k-j-1}\gamma)}) \right|^2
		+\frac{3}{2^{6}}\left|  (1+e^{- 2 \pi i (2^{k-j-1}\gamma)}) \right|^4 
		\\&\times\left|(1-e^{- 2 \pi i (2^{k-j-1}\gamma)})  \right|^4+ \frac{1}{2^{5}}\left|  (1+e^{- 2 \pi i (2^{k-j-1}\gamma)})\right|^2 \left| (1-e^{- 2 \pi i (2^{k-j-1}\gamma)})^6 \right|^2+\left[\frac{1}{2^{7}}\left|   (1-e^{- 2 \pi i (2^{k-j-1}\gamma)}) \right|^2\right]^4
	\end{align*}
	\endgroup
	since $\sum\limits_{m=1}^{8}|a_{1m}|^2=\frac{1}{2^5}$, $\sum\limits_{m=1}^{8}|a_{2m}|^2=\frac{3}{2^6}$, $\sum\limits_{m=1}^{8}|a_{3m}|^2=\frac{1}{2^5}$ and $\sum\limits_{m=1}^{8}|a_{4m}|^2=\frac{1}{2^7}$. This is further equivalent to 	
	\begin{align*}
		=\frac{1}{2^7}\left[\left|   (1+e^{- 2 \pi i (2^{k-j-1}\gamma}) \right|^2+\left|   (1-e^{- 2 \pi i (2^{k-j-1}\gamma}) \right|^2 \right]^4=\frac{1}{2^7}4^4=2.\nonumber \nonumber
	\end{align*}	 
	When $\ell \mbox{ and }\ell'=2,$ we have $\nu_{j,\ell}=\nu_{j,\ell'}=2^{j-k}$. Proceeding as in the case $\ell=\ell'=1$, 
we obtain $\displaystyle\sum_{m=0}^{8}\left|G^{(i)(m)}_{j+1}(\gamma+2^{j-k})\right|^2=2.$
	Next, we suppose $\ell=1$ and $\ell'=2,$ then $\nu_{j,\ell}=0$ and $\nu_{j,\ell'}=2^{j-k}.$ Now,
	\begin{align}
		\displaystyle\sum_{m=0}^{8}\overline{G^{(1)(m)}_{j+1}(\gamma+\nu_{j,\ell})}G^{(1)(m)}_{j+1}(\gamma+\nu_{j,\ell^{'}})=&\overline{H_{j+1}(\gamma)} H_{j+1}(\gamma+2^{j-k}) \nonumber\\
		 &+\displaystyle\sum_{m=1}^{4} \overline{G^{(1)(m)}_{j+1}(\gamma)}G^{(1)(m)}_{j+1}(\gamma+2^{j-k}). \label{ex eq6}
	\end{align}
	Next, using $e^{-2\pi i (2^{k-j-1}(2^{j-k}+\gamma))}=-e^{-2\pi i (2^{k-j-1}\gamma)}$ and  by following a similar steps as in  $\ell,\ell'=1$, the right hand side of \eqref{ex eq6} can be expressed as 
	\begingroup
	\allowdisplaybreaks
	\begin{align*}
		=&\frac{1}{2^7}	\left[(1+\overline{e^{- 2 \pi i (2^{k-j-1}\gamma)}})(1-e^{- 2 \pi i (2^{k-j-1}\gamma)})\right]^4+\frac{1}{2^7}	\left[(1-\overline{e^{- 2 \pi i (2^{k-j-1}\gamma)}})(1+e^{- 2 \pi i (2^{k-j-1}\gamma)})\right]^4\\
		&+\frac{1}{2^{5}} \left[(1+\overline{e^{- 2 \pi i (2^{k-j-1}\gamma)}})(1-e^{- 2 \pi i (2^{k-j-1}\gamma)})\right]^3 \left[(1-\overline{e^{- 2 \pi i (2^{k-j-1}\gamma)}})(1+e^{- 2 \pi i (2^{k-j-1}\gamma)})\right]\\
		&+\frac{3}{2^{6}}\left[(1+\overline{e^{- 2 \pi i (2^{k-j-1}\gamma)}})(1-e^{- 2 \pi i (2^{k-j-1}\gamma)})\right]^2 \left[(1-\overline{e^{- 2 \pi i (2^{k-j-1}\gamma)}})(1+e^{- 2 \pi i (2^{k-j-1}\gamma)})\right]^2
		\\&+ \frac{1}{2^{5}}\left[(1+\overline{e^{- 2 \pi i (2^{k-j-1}\gamma)}})(1-e^{- 2 \pi i (2^{k-j-1}\gamma)})\right] \left[(1-\overline{e^{- 2 \pi i (2^{k-j-1}\gamma)}})(1+e^{- 2 \pi i (2^{k-j-1}\gamma)})\right]^3\\
		= &\frac{1}{2^7}\left[(1-\overline{e^{- 2 \pi i (2^{k-j-1}\gamma)}})(1+e^{- 2 \pi i (2^{k-j-1}\gamma)})
		+(1-\overline{e^{- 2 \pi i (2^{k-j-1}\gamma)}})(1+e^{- 2 \pi i (2^{k-j-1}\gamma)}) \right]^4  
		=  0\nonumber
	\end{align*}
	\endgroup
	using $(1+\overline{z})(1-z)+(1-\overline{z})(1+z)=0$ for $|z|=1$. Similarly, for $\ell=2 \mbox{ and }\ell'=1,$ we have
	$\displaystyle\sum_{m=0}^{8}\overline{G^{(i)(m)}_{j+1}(\gamma+2^{j-k})}G^{(i)(m)}_{j+1}(\gamma)=0.$
	Hence, we obtain
	$$({\mathfrak{B}^{(i)}_j}(\gamma))^{\ast} \mathfrak{B}^{(i)}_j(\gamma)=2I_{2}\,\,  \mbox{  for a.e. } \gamma \in \Omega_j=[0, 2^{j-k}), \ i \in \{1,2\} \mbox{ and } j \in \J.$$
	Let  $\Gamma_{j}:=2^{k-j}\mathbb{Z} \subset \Z.$ Then $\Gamma_{j}^{\perp}=2^{-k+j}\mathbb{Z} $ and its fundamental domain is $\Omega_j=[0, 2^{j-k}).$ Now
	\begin{align*}
		\frac{1}{s(\Gamma_{j})}\left|\Phi_j(\gamma)\right|^2=\left|\frac{1}{\sqrt{s(\Gamma_{j})}} \Phi_j(\gamma)\right|^2&=\left| \frac{1}{(2^{k-j})^{4}}\frac{(1-e^{- 2 \pi i (2^{k-j}\gamma)})^{4}}{(1-e^{- 2 \pi i \gamma})^4}\right|^2.
	\end{align*}
	For $j=k$, we have
	\begin{align*}
		\frac{1}{s(\Gamma_{k})}\left|\Phi_k(\gamma)\right|^2=\left| \frac{1}{(2^{k-k})^{4}}\frac{(1-e^{- 2 \pi i (2^{k-k}\gamma)})^{4}}{(1-e^{- 2 \pi i \gamma})^4}\right|^2= \left| \frac{(1-e^{- 2 \pi i \gamma})^{4}}{(1-e^{- 2 \pi i \gamma})^4}\right|^2=1.
	\end{align*}
	Thus assumption ($\mathcal{N}_1$) holds.  Next,
	$$\lim\limits_{j \to -\infty}\frac{1}{\sqrt{s(\Gamma_{j})}}|\Phi_j(\gamma)|=\left|\lim\limits_{j \to -\infty} \left(\frac{1}{(2^{k-j})}\frac{(1-e^{- 2 \pi i (2^{k-j}\gamma)})}{(1-e^{- 2 \pi i \gamma})}\right)^4\right|=0$$
	using the fact $\lim\limits_{x \to \infty} \left(\frac{1-e^{-ix}}{x}\right)=0$. This implies assumption  ($\mathcal{N}_2$) is true. Thus all the assumptions of Theorem \ref{thm_UCP} are true and hence, for each $i \in \{1,2\}$, the system $\bigcup\limits_{j=-\infty}^{k}\{T_\la \gpi\}_{\la \in 2^{-j+k}\Z,\, p \in  \{(m,j): m=1,2,\ldots, 8\}}$ is a Parseval frames and satisfies $\infty$-UCP, where $
g_{(m,j)}^{(i)}=\mathcal{F}^{-1}\Psi_j^{(i)(m)}.$
\end{ex}

The construction criterion in Theorem~\ref{thm_UCP}, and in particular the matrix condition~\eqref{PFC}, may be viewed as a GTI version of the Unitary Extension Principle (UEP). In the following subsection, our goal is to introduce a more flexible generalization of the UEP, called the Oblique Extension Principle. For a detailed motivation and background on the OEP in $L^2(\R)$, we refer the reader to Chapter 18, Subsection 18.4 in \cite{COB}.

\subsection{Oblique Extension Principle}\label{secOEP}

The next theorem provides an \textit{oblique extension principle} for GTI systems over LCA groups.
\begin{thm}\label{thm_QEP}
	Let $\{\Phi_{j}, H_j, G^{(i)(m)}_{j} \}_{j \in \J, m \in \{0,1, \cdots, s_j\}}$ be defined  as in general setup in Subsection~\ref{subsec_4.1} and assume that conditions $(\mathcal{N}_1)$--$(\mathcal{N}_2)$ hold. 
\begin{enumerate}
\item[(i)]
For every compact set $S\subset \widehat G$, there exists
$J\in\mathcal J$ such that
\[
\mu_{\widehat G}\bigl((\omega+S)\cap(\omega'+S)\bigr)=0
\]
whenever $\omega,\omega'\in\Gamma_J^\perp$ and
$\omega\neq\omega'$.

\item[(ii)]
There exists a sequence
$\{\theta_j\}_{j\in\mathcal J}$ of strictly positive
$\Gamma_j^\perp$-periodic functions in
$L^\infty(\Omega_j)$ such that, for every compact set
$S\subset\widehat G$ and every $\varepsilon>0$, there exists
$J'_1\in\mathcal J$ satisfying
\begin{equation}\label{eq_N_1'}
		\left| \theta_j (\gamma)-1\right| \leq \epsilon \quad \mbox{for all } \, \gamma \in S,
	\end{equation}
for all $j\ge J'_1$.

\item[(iii)]
For each $j\in\mathcal J$, let
\begin{equation} \label{eq:Pk-def}
		\mathfrak{R}_j^{(i)}(\gamma) :=
		\begin{pmatrix}
			H_{j+1}(\gamma + \nu_{j,1})\sqrt{\theta_j (\g)} & \cdots & H_{j+1}(\gamma + \nu_{j,d_j}) \sqrt{\theta_j (\g)}\\
			G^{(1)(i)}_{j+1}(\gamma + \nu_{j,1}) & \cdots & G^{(1)(i)}_{j+1}(\gamma + \nu_{j,d_j}) \\
			\vdots & \ddots & \vdots \\
			G_{j+1}^{(s_j)(i)}(\gamma + \nu_{j,1}) & \cdots & G_{j+1}^{(s_j)(i)}(\gamma + \nu_{j,d_j})
		\end{pmatrix} \quad \mbox{ for }\gamma \in \Omega_j.
	\end{equation}
Assume that
\begin{align}
		({\mathfrak{R}^{(i)}_j}(\gamma))^{\ast} \mathfrak{R}^{(i)}_j(\gamma)=\frac{s(\Gamma_{j})}{s(\Gamma_{j+1})}\theta_{j+1} (\g)I_{d_j},\label{eq_OEP_matrices}
	\end{align}
\end{enumerate}
	Then, for $i \in \{1,2\}$, the GTI system $\displaystyle\bigcup_{j \in \mathcal{J}}\{T_{\lambda}g_p^{(i)}\}_{\lambda \in \Gamma_{j}, p \in P_j}$ is a Parseval frame for $L^2(G)$.
\end{thm}

\bp
	Assume that the hypotheses of Theorem~\ref{thm_QEP} hold. Define the sequence $\{\tpj\}_{j \in \mathcal{J}}$ of functions in $L^2(\widehat{G})$ by
	$$\tpj(\g)= \sqrt{\theta_j (\g)} \Phi_j (\g).$$
	Define the $\Gjp$-periodic functions $\thj(\gamma)$ and $\tgmij(\gamma)$ in $L^\infty(\Omega_j)$,  by 
	\begin{equation}
		\thj(\g)=\sqrt{\frac{\theta_{j-1}(\g)}{\theta_j(\g)}}H_j(\gamma) \mbox{ and } \tgmij(\g)=\sqrt{\frac{1}{\theta_j(\g)}}G^{(i)(m)}_j(\gamma).
	\end{equation}
	Also define the matrices $\widetilde{\mathfrak{Z}^{(i)}_j}(\gamma):=
	{{\begin{pmatrix}
				\widetilde{G^{(i)(m)}_{j+1}}(\gamma+ v_{j,n})
	\end{pmatrix}}}_{\substack{
			0\leq m \leq s_j \\
			1\leq n \leq d_j
	}} \mbox{ for } \ a.e. \ \gamma \in \Omega_j$.
	The idea of the proof is to apply Theorem \ref{thm_UCP} to $\tpj, \thj, \tsij$  and therefore  obtain the GTI  Parseval frame  $\displaystyle\bigcup\limits_{j \in \mathcal{J}}\{T_{\lambda}\widetilde{g_p^{(i)}}\}_{\lambda \in \Gamma_{j}, p \in P_j} $ for $L^2(G)$, where $P_j=\{(m,j): m=1,2,\cdots,s_j\}$ and $\mathcal{F}(\widetilde{g_p^{(i)}})=\tsij(\gamma)=\tgmijplusone \tpjplusone$. Finally it turns out that $\widetilde{g_p^{(i)}}=\gpi$. We now prove that $\tpj, \thj \mbox{ and }\tsij$ satisfy the conditions in the general setup in Subsections~\ref{subsec_4.1} and \ref{subsec_sufficent_condtions_for_UCP}. First,
	\begin{align*}
		\tpj(\g)=\sqrt{\theta_j (\g)} \Phi_j (\g)&=\sqrt{\theta_j (\g)} H_{j+1}(\g) \Phi_{j+1} (\g)\\
		&=\sqrt{\frac{\theta_{j}(\g)}{\theta_{j+1}(\g)}}H_{j+1}(\gamma)\tpjplusone (\gamma)\\
		&=\thjplusone(\g)\tpjplusone(\g).
	\end{align*} 
	Next, Let $S$ be any compact  set in $\hg \setminus \mathcal{B}$, then for $\gamma \in S,$ we have
	\begin{align*}
		\left| 	\frac{1}{s(\Gj)}\left|\tpj(\g)\right|^2-1\right|&=\left| \frac{1}{s(\Gj)}\left|\Phi_j (\g)\right|^2 \theta_j (\g) -1 \right|\\
		&=\left| \frac{1}{s(\Gj)}\left|\Phi_j (\g)\right|^2 \theta_j (\g) -\frac{1}{s(\Gj)}\left|\Phi_j (\g)\right|^2+\frac{1}{s(\Gj)}\left|\Phi_j (\g)\right|^2 -1 \right|\\
		&\leq \frac{1}{s(\Gj)}\left|\Phi_j (\g)\right|^2 \left| \theta_j (\g)-1 \right|+ \left| \frac{1}{s(\Gj)}\left|\Phi_j (\g)\right|^2 -1 \right|.
	\end{align*}
	Now, let $\epsilon>0$ and choose $J:=\max\{J_1, J_1'\}$, then using assumption $(\mathcal{N}_1)$ and \eqref{eq_N_1'}, for all $j \geq J$ and  $\gamma \in S$, we obtain  
	\begin{align*}
		\left| 	\frac{1}{s(\Gj)}\left|\tpj(\g)\right|^2-1\right|&\leq \frac{1}{s(\Gj)}\left|\Phi_j (\g)\right|^2 \epsilon + \epsilon= \left(\frac{1}{s(\Gj)}\left|\Phi_j (\g)\right|^2 +1\right)\epsilon\leq (M+1)\epsilon,
	\end{align*}
	 where $M> (1+\epsilon)$. 
	Also, since  sequence of functions $\theta_j(\gamma)$ is uniformly bounded on every compact subset of $\hg$,  and using ($\mathcal{N}_2$), it follows that for every compact set $S \in \hg \setminus \mathcal{B}$,  there exists $J_2 \in \mathcal{J}$ such that for all $j \leq J_2, \ j \in \mathcal{J},$ 
	$$ \frac{1}{\sqrt{s(\Gamma_{j})}}|\tpj (\gamma)| \leq \epsilon , \,\, \forall \, \gamma \in S.$$ 
	Next,  in order to show that 
	\begin{align*}
		\left(\widetilde{\mathfrak{Z}^{(i)}_j}(\gamma)\right)^\ast\widetilde{\mathfrak{Z}^{(i)}_j}(\gamma)=\frac{s(\Gamma_{j})}{s(\Gamma_{j+1})}I_{d_j} \mbox{ for a.e. } \gamma \in \Omega_j,
	\end{align*}
	equivalently to show that 
	\begin{align*}
		\begin{pmatrix}
			\overline{\thjplusone(\gamma+ v_{j,\ell})}\thjplusone(\gamma+ v_{j,\ell'})+	\sum \limits_{m=1}^{s_{j}} \overline{\tgmijplusone (\gamma+ v_{j,\ell})} \tgmijplusone(\gamma + v_{j,\ell'})
		\end{pmatrix}_{\substack{
				1\leq \ell \leq d_j \\
				1\leq \ell' \leq d_j
		}} =\frac{s(\Gamma_{j})}{s(\Gamma_{j+1})}I_{d_j} 
	\end{align*}
for a.e.	$\gamma \in \Omega_j.$ Equivalently, 
	\begin{align}
		\overline{\thjplusone(\gamma+ v_{j,\ell})}\thjplusone(\gamma+ v_{j,\ell'})+	\sum \limits_{m=1}^{s_{j}} \overline{\tgmijplusone (\gamma+ v_{j,\ell})} \tgmijplusone(\gamma + v_{j,\ell'})=\delta_{\ell,\ell'} \frac{s(\Gamma_{j})}{s(\Gamma_{j+1})} \label{eq_OEP_1}
	\end{align}
for $1\leq \ell,\ell' \leq d_j$  and for  a.e.  $\gamma\in \Omega_{j}$.	Now, by substituting $\thjplusone$ and $\tgmijplusone$, we get
	\begingroup
	\allowdisplaybreaks
	\begin{align*}
			\overline{\thjplusone(\gamma+ v_{j,\ell})}&\thjplusone(\gamma+ v_{j,\ell'})+	\sum \limits_{m=1}^{s_{j}} \overline{\tgmijplusone (\gamma+ v_{j,\ell})} \tgmijplusone(\gamma + v_{j,\ell'})\\
		=&\overline{\sqrt{\frac{\theta_{j}(\gamma+ v_{j,\ell})}{\theta_{j+1}(\g+v_{j,\ell})}}H_{j+1}(\gamma+ v_{j,\ell})}\sqrt{\frac{\theta_{j}(\gamma+ v_{j,\ell'})}{\theta_{j+1}(\g+v_{j,\ell'})}}H_{j+1}(\gamma+ v_{j,\ell'})\\
		&+	\sum \limits_{m=1}^{s_{j}} \overline{\sqrt{\frac{1}{\theta_j(\gamma+ v_{j,\ell})}}G^{(i)(m)}_{j+1} (\gamma+ v_{j,\ell})} \sqrt{\frac{1}{\theta_j(\gamma+ v_{j,\ell'})}}G^{(i)(m)}_{j+1} (\gamma+ v_{j,\ell'})\\
	=&\frac{\theta_{j}(\gamma)}{\sqrt{\theta_{j+1}(\g+v_{j,\ell})\theta_{j+1}(\g+v_{j,\ell'})}}\overline{H_{j+1}(\gamma+ v_{j,\ell})} H_{j+1}(\gamma+ v_{j,\ell'})\\
		&+\sum \limits_{m=1}^{s_{j}}\frac{1}{\sqrt{\theta_{j+1}(\g+v_{j,\ell})\theta_{j+1}(\g+v_{j,\ell'})}} \overline{G^{(i)(m)}_{j+1} (\gamma+ v_{j,\ell})} G^{(i)(m)}_{j+1} (\gamma+ v_{j,\ell'})\\
		=&\frac{1}{\sqrt{\theta_{j+1}(\g+v_{j,\ell})\theta_{j+1}(\g+v_{j,\ell'})}} \bigg[ \theta_{j}(\gamma)\overline{H_{j+1}(\gamma+ v_{j,\ell})} H_{j+1}(\gamma+ v_{j,\ell'}) \\
		&+\sum \limits_{m=1}^{s_{j}} \overline{G^{(i)(m)}_{j+1} (\gamma+ v_{j,\ell})} G^{(i)(m)}_{j+1} (\gamma+ v_{j,\ell'})\bigg]
	\end{align*} 
	\endgroup
	since $\theta_{j}$ is a $\Gamma_{j}^\perp$-periodic, that is,  $\theta_{j}(\gamma+ v_{j,\ell})=\theta_{j}(\g)$ as $v_{j,\ell} \in \Gjp$ for each $\ell \in \{1,2,\cdots,d_j\}$. Therefore, \eqref{eq_OEP_1}, is further equivalent to 
	\begin{align}
		\theta_{j}(\gamma)\overline{H_{j+1}(\gamma)} H_{j+1}(\gamma) +\sum \limits_{m=1}^{s_{j}} \overline{G^{(i)(m)}_{j+1} (\gamma)} G^{(i)(m)}_{j+1} (\gamma)=\theta_{j+1}(\gamma)\frac{s(\Gamma_{j})}{s(\Gamma_{j+1})}\label{eq_OEP_2}
	\end{align}
	and 
	\begin{align}
		\theta_{j}(\gamma)\overline{H_{j+1}(\gamma+ v_{j,\ell})} H_{j+1}(\gamma+ v_{j,\ell'}) +\sum \limits_{m=1}^{s_{j}} \overline{G^{(i)(m)}_{j+1} (\gamma+ v_{j,\ell})} G^{(i)(m)}_{j+1} (\gamma+ v_{j,\ell'})=0.\label{eq_OEP_3}
	\end{align}
	Now,  equations \eqref{eq_OEP_2} and \eqref{eq_OEP_3} are true by the assumption \eqref{eq_OEP_matrices}. Now, we define the functions $\tsij$  for $j \in \J$ and $m=1,2, \dots,s_j$ as follows:
	\begin{equation}
		\tsij(\gamma):=\tgmijplusone (\gamma) \tpjplusone(\gamma),\ \gamma \in \widehat{G} \mbox{ and } m=1,2, \dots,s_j.
	\end{equation}
	For $j \in \mathcal{J}, m \in \{1,2, \dots,s_j\}$ and $i \in \{1,2\},$  we define the functions $g^{(i)}_{(m,j)}$  as  inverse Fourier transform of  $\tsij$  by
	\begin{equation}
		\widetilde{g_{(m,j)}^{(i)}}=\mathcal{F}^{-1}\tsij. 
	\end{equation}
	Thus by Theorem \ref{thm_UCP}, GTI system	$\displaystyle\bigcup\limits_{j \in \mathcal{J}}\{T_{\lambda}\widetilde{g_p^{(i)}}\}_{\lambda \in \Gamma_{j}, p \in P_j} $ is a Parseval frames, where $P_j=\{(m,j): m=1,2,\cdots,s_j\}$.
	Now,  observe that for $m=1,2,\cdots,s_j$,
	\begin{align*}
		\Psi^{(i)(m)}_j(\gamma)=G^{(i)(m)}_{j+1}(\gamma)\Phi_{j+1}(\gamma)&=\sqrt{\frac{1}{\theta_{j+1}(\g)}}G^{(i)(m)}_{j+1}(\gamma)\sqrt{\theta_{j+1} (\g)} \Phi_{j+1} (\g)\\
		&=\tgmijplusone (\gamma) \tpjplusone(\gamma)=\tsij(\gamma).
	\end{align*} 
	This completes the proof.
\ep

\subsection{Construction of Pairwise orthogonal Parseval frames}\label{subsec_pairwise_orthogonal}
The following theorem provides a technique to construct pairwise orthogonal Parseval frames.

\begin{thm}\label{thm_pairwise_orthogonal_frames}
	Assume all the hypotheses of Theorem \ref{thm_UCP} hold. Further, suppose that for each $j \in \mathcal{J}$,  the matrix valued functions $\widetilde{\mathfrak{B}^{(1)}_j}(\gamma)$  and  $\widetilde{\mathfrak{B}^{(2)}_j}(\gamma)$, defined in (\ref{matrices}), satisfy 
	$$(\widetilde{\mathfrak{B}^{(1)}_j}(\gamma))^{\ast} \widetilde{\mathfrak{B}^{(2)}_j}(\gamma)=O_{d_j} \mbox{ for a.e. } \gamma \in \Omega_j,$$
	where  $O_{d_j}$ is the zero matrix of order $d_j$. Then  $\gtione$ and $\gtitwo$ (defined in (\ref{GTI eq6}))   are  pairwise orthogonal Parseval frames in $L^2(G).$ 
\end{thm}

\bp
	For each $i \in \{1,2\}$, the GTI system $\gtii$ satisfies   $\infty$-UCP, using Theorem \ref{thm_UCP}(i) and also Parseval frame by Theorem \ref{thm_UCP}(ii). Now the GTI system $\gtione$ is a Parseval frame, and the GTI system $\gtitwo$ satisfies the $\infty$-UCP. Thus, in view of Remark \ref{rem_dual_UCP}, both the GTI systems satisfy the dual $\infty$-UCP and hence dual $1$-UCP. For proving the systems are pairwise orthogonal Parseval frames, it is sufficient to show the following in view of Theorem \ref{thm_characterization_orthogonal_frame}:    
	\begin{equation}\label{a=0}
		\sum \limits_{j \in \mathcal{J}}\frac{1}{s(\Gamma_{j})} \sum \limits_{p \in P_j} \overline{\widehat{(g_p^{(1)})}(w)} \widehat{(g_p^{(2)})}(w)=0\,\, \mbox{for a.e.}\,\, w \in \widehat{G},
	\end{equation}
	and  for $ \alpha \in \displaystyle\bigcup_{j \in \mathcal{J}} \Gamma^\perp_j\setminus\{0\},$
	\begin{equation}\label{a}
		\sum \limits_{j \in \mathcal{J} :\alpha \in \Gamma_{j}} \frac{1}{s(\Gamma_{j})} \sum \limits_{p \in P_j} \overline{\widehat{(g_p^{(1)})}(w)} \widehat{(g_p^{(2)})}(w+\alpha)=0\,\, \mbox{for a.e.}\,\, w \in \widehat{G}.
	\end{equation}
	To prove (\ref{a=0}) and (\ref{a}), we refer to the proof of \cite[Theorem 3.1]{RGS}. However, \cite[Theorem 3.1]{RGS} relies on the assumption that $\bigcap\limits_{j \in \mathcal{J}} \Gamma_j^\perp= {0}$.
	We  demonstrate that this condition follows from the hypothesis already stated in Theorem~\ref{thm_UCP}, i.e., for a compact set $S$ in $\widehat{G}$, there exists $J \in \mathbb{Z}$ such that
	\begin{equation}\label{eq_UEP_*}
		\mu_{\hg}\big( (\omega + S) \cap (\omega' + S) \big) = 0 \quad \text{for } \omega \neq \omega', \;\; \omega, \omega' \in \Gamma_J^\perp.
	\end{equation}
	To see this, we assume the contrary, i.e.,  $\bigcap\limits_{j \in \mathcal{J}} \Gamma_j^\perp \neq \{0\}.$ Then there exists a  $w \in \bigcap\limits_{j \in \mathcal{J}} \Gamma_j^\perp$ such that $w\neq 0$.  Let $\Omega_j$ be a fundamental domain with $\mu_{\widehat{G}}(\Omega_j) \neq 0$. Choose a compact set $T \subset \Omega_j$ such that $\mu_{\widehat{G}}(T) \neq 0$ and define $S := T \cup (T + w)$. Then $S \cap (S + w) \supset T + w$, so we have
	$$\mu_{\hg}(S \cap S+w) \geq \mu_{\hg}(T+w) \neq 0, w \in \Gjp \mbox{ for  each } j \in \J.$$
	Hence, for this compact set $S$, there is no $j \in \mathcal{J}$ such that $\mu_{\widehat{G}}\big( (w + S) \cap (0 + S) \big) = 0$ for $w \in \Gamma_j^\perp$. This contradicts \eqref{eq_UEP_*}. Therefore $\bigcap\limits_{j \in \mathcal{J}} \Gamma_j^\perp = {0}$.
\ep
In contrast to \cite{RGS}, where the construction requires the subgroups $\{\Gamma_j\}_{j\in J}$ to become stationary as $j \to -\infty$, Theorem \ref{thm_pairwise_orthogonal_frames} provides pairwise orthogonal Parseval GTI systems without this assumption.

\begin{ex}\label{ex_pair_of_orthogonal_frames_by_B-spline}
	Recall the GTI systems $\bigcup\limits_{j=-\infty}^{k}\{T_\la \gpi\}_{\la \in 2^{-j+k}\Z, \, p \in  \{(m,j): m=1,2,\ldots, 8\}}$ for $i \in \{1,2\}$ constructed in Example \ref{ex_B_spline}. Since all the Assumptions of  Theorem \ref{thm_UCP} are satisfied, it follows that both are Parseval frames. 
	Next, we show that for each $j \in \J$,
	$$(\widetilde{\mathfrak{B}^{(1)}_j}(\gamma))^{\ast} \widetilde{\mathfrak{B}^{(2)}_j}(\gamma)=O_{d_j} \mbox{ for a.e. } \gamma \in \Omega_j.$$
	Equivalently, to show that
	for a.e. $\gamma \in [0, 2^{j-k})$ and $\ell,\ell^{'}\in \{1,2\},$
	\begin{equation}\label{ex o eq1}
		\displaystyle\sum_{m=1}^{8}\overline{G^{(1)(m)}_{j+1}(\gamma+\nu_{j,\ell})}G^{(2)(m)}_{j+1}(\gamma+\nu_{j,\ell^{'}})=0, \mbox{ where } \nu_{j,1}=0 \mbox{ and } \nu_{j,2}=2^{j-k}.
	\end{equation}	
	First suppose  $\ell=\ell'=1,$ then $\nu_{j,\ell}=\nu_{j,\ell'}=0.$  Now, we calculate 
	\begin{align*}
		\displaystyle\sum_{m=1}^{8}&\overline{G^{(1)(m)}_{j+1}(\gamma+\nu_{j,\ell})}G^{(2)(m)}_{j+1}(\gamma+\nu_{j,\ell^{'}})=\displaystyle\sum_{m=1}^{8}\overline{G^{(1)(m)}_{j+1}(\gamma)}G^{(2)(m)}_{j+1}(\gamma) \nonumber\\
		=& \sum_{m=1}^{8}\overline{a_{1m}}b_{1m}\left|  (1+e^{- 2 \pi i (2^{k-j-1}\gamma})^3 (1-e^{- 2 \pi i (2^{k-j-1}\gamma})  \right|^2\\
		&+\sum_{m=1}^{8}\overline{a_{2m}}b_{2m}\left|  (1+e^{- 2 \pi i (2^{k-j-1}\gamma})^2 (1-e^{- 2 \pi i (2^{k-j-1}\gamma})^2  \right|^2\\
		&+\sum_{m=1}^{8}\overline{a_{3m}}b_{3m}\left|  (1+e^{- 2 \pi i (2^{k-j-1}\gamma})^1 (1-e^{- 2 \pi i (2^{k-j-1}\gamma})^3  \right|^2+\sum_{m=1}^{8}\overline{a_{4m}}b_{4m}\left|   (1-e^{- 2 \pi i (2^{k-j-1}\gamma})^4  \right|^2\\
		&=0,
	\end{align*}
	since $\sum\limits_{m=1}^{8}\overline{a_{nm}}b_{nm}=0$ for each $n \in \{1,2,3,4\}$.  Similarly, we can show that (\ref{ex o eq1})  is true for the other values of $\ell$ and $\ell'.$ 
	Hence, using Theorem \ref{thm_pairwise_orthogonal_frames}, the GTI systems  $\bigcup\limits_{j=-\infty}^{j_0}\{T_\gamma \gpone\}_{\gamma \in 2^{-j+k}\Z, \, p \in  \{(m,j): m=1,2,\ldots, 8\}}$ and $\bigcup\limits_{j=-\infty}^{j_0}\{T_\la \gptwo\}_{\la \in 2^{-j+k}\Z, \, p \in  \{(m,j): m=1,2,\ldots, 8\}}$ are pairwise orthogonal Parseval frames. 
\end{ex}
The following application of Theorem \ref{thm_pairwise_orthogonal_frames} provides a recipe to construct pairwise orthogonal Parseval frames of the form
$$\displaystyle\bigcup_{j \in \mathcal{J}}\{T_{\la}h_p^{(1)}\}_{\la \in \e\Gamma_{j}, p \in P_j} \mbox{ and } \displaystyle\bigcup_{j \in \mathcal{J}}\{T_{\la}h_p^{(2)}\}_{\la \in \z\Gamma_{j}, p \in P_j}.$$

\begin{thm}\label{thm_new_pairwise_orthogonal_frames}
Let the GTI systems $\gtione$ and $\gtitwo$ be constructed as in Theorem~\ref{thm_pairwise_orthogonal_frames}. Then 
	$$\displaystyle\bigcup_{j \in \mathcal{J}}\{T_{\la}h_p^{(1)}\}_{\la \in \e\Gamma_{j}, p \in P_j} \mbox{ and } \displaystyle\bigcup_{j \in \mathcal{J}}\{T_{\la}h_p^{(2)}\}_{\la \in \z\Gamma_{j}, p \in P_j}$$
	are pairwise orthogonal Parseval frames, where $h_p{(1)}:=D_\e^{-1} g_p^{(1)} \mbox{ and } h_p{(2)}:=D_\z^{-1} g_p^{(2)}.$
\end{thm}
\bp
Let $\Theta_{\e, \z}$ be the mixed dual Gramian operator for the systems $\displaystyle\bigcup_{j \in \mathcal{J}}\{T_{\la}h_p^{(1)}\}_{\la \in \e\Gamma_{j}, p \in P_j}  $ and $\displaystyle\bigcup_{j \in \mathcal{J}}\{T_{\la}h_p^{(2)}\}_{\la \in \z\Gamma_{j}, p \in P_j}$. By polarisation identity, these systems  are pairwise orthogonal if and only if for all $f \in \ltg$, we have
\begin{align*}
	\langle D_\z \Theta_{\e, \z} D_\e^{-1}f, f \rangle =0.
\end{align*}
By following the similar steps as in the proof of Theorem \ref{thm_new_characterization_orthogonal_frame}, the above condition is equivalent to the systems $\displaystyle\bigcup\limits_{j \in \mathcal{J}} \{T_{\la} D_\e h^{(1)}_p \}_{\la \in \Gamma_j, p \in P_j}$ and $\displaystyle\bigcup\limits_{j \in \mathcal{J}} \{T_{\la} D_\z h^{(2)}_p \}_{\la \in \Gamma_j, p \in P_j}$  are pairwise orthogonal. By substituting values of $h^{(1)}_p$ and $h^{(2)}_p$,
it is equivalent to $\gtione$ and $\gtitwo$ are pairwise orthogonal. This is true by Theorem \ref{thm_pairwise_orthogonal_frames}.

Note that the system 
$$\displaystyle\bigcup_{j \in \mathcal{J}}\{T_{\la}h_p^{(1)}\}_{\la \in \e\Gamma_{j}, p \in P_j}=\displaystyle\bigcup_{j \in \mathcal{J}}\{T_{\e\gamma} D_\e^{-1} g_p^{(1)}\}_{\la \in \Gamma_{j}, p \in P_j}= \displaystyle\bigcup_{j \in \mathcal{J}}\{ D_\e^{-1}T_{\la}  g_p^{(1)}\}_{\la \in \Gamma_{j}, p \in P_j}$$
which is a Parseval frame, since $\gtione$ is a Parseval frame by Theorem \ref{thm_pairwise_orthogonal_frames} and $D_\e^{-1}$ is a unitary operator. A similar arguments shows that $\displaystyle\bigcup_{j \in \mathcal{J}}\{T_{\la}h_p^{(2)}\}_{\la \in \z\Gamma_{j}, p \in P_j}$ is also a Parseval frame. This finishes the proof.
\ep

	In \cite{RGS}, the authors present a method for constructing two orthogonal frames and $N$ Parseval frames from a given Parseval frame. However, the construction of $N$ Parseval frames is possible only when the original frame has a single filter $G^{(1)(1)}$, that is, when $s_j=1$. If the given frame contains more than one filter, only two new Parseval frames can be generated using the method from \cite{RGS}.

In contrast, the following theorem  addresses this limitation: it enables the construction of $N$ Parseval frames from a single given Parseval frame, regardless of the number of filters present. This advancement generalizes \cite[Theorems 3.2 and 3.3]{RGS}, offering a broader, more powerful construction framework.

\begin{thm}\label{thm_construction_of_N_pair_OF}
	Consider the GTI system  $\,\bigcup_{j \in \mathcal{J}}\{T_{\lambda}g_p^{(1)}\}_{\lambda \in \Gamma_{j}, p \in P_j}$  (defined in (\ref{GTI eq6})) satisfying the assumptions of Theorem \ref{thm_UCP}. Let  $V^{j}(\gamma)=\begin{pmatrix}
		V^{j}_{1}(\gamma) \ 	V^{j}_{2}(\gamma) \cdots \	V^{j}_{Ns_j}(\gamma)
	\end{pmatrix}$ be an $N s_j \times N s_j$ unitary matrix whose entries are $\Gamma_{j}^\perp$-periodic, where $V_r^j(\gamma)$ denotes its $r$-th column. For $j \in \mathcal{J}$ and $i \in \{2,3,\dots,N+1\}$, assume that  $G^{(i)(0)}_{j+1}(\gamma)=H_{j+1}(\gamma)$ and for $m \in \{1,2,\dots,N s_j\}$, suppose that the functions $G^{(i)(m)}_{j+1}$ satisfy
	\begin{equation}\label{A1_1}
		\small{
			\begin{pmatrix}
				G_{j+1}^{(i)(1)}(\gamma)  \\
				G_{j+1}^{(i)(2)}(\gamma) \\
				\vdots\\
				G_{j+1}^{(i)(Ns_j)}(\gamma) 
			\end{pmatrix}= \begin{pmatrix}
				V_{1+(i-2)s_j}^{j}(\gamma) \, V_{2+(i-2)s_j}^{j}(\gamma) \, \cdots \, V_{s_j+(i-2)s_j}^{j}(\gamma)
			\end{pmatrix}
			\begin{pmatrix}
				G_{j+1}^{(1)(1)}(\gamma)  \\
				G_{j+1}^{(1)(2)}(\gamma) \\
				\vdots\\
				G_{j+1}^{(1)(s_j)}(\gamma) 
			\end{pmatrix}, \ \gamma \in \widehat{G}.}
	\end{equation}
	  
	Then   $\,\bigcup_{j \in \mathcal{J}}\{T_{\lambda}g_p^{(i)}\}_{\lambda \in \Gamma_{j}, p \in P_j^{'}}$ and $\,\bigcup_{j \in \mathcal{J}}\{T_{\lambda}g_p^{(i')}\}_{\lambda \in \Gamma_{j}, p \in P_j^{'}}$ (defined in (\ref{GTI eq6})) are pairwise orthogonal Parseval frames, where $P_j^{'}=\{(m,j):m=1,2,\cdots, Ns_j \},$ and $ i,i' \in \{2,3, \cdots, N+1\}$. 
\end{thm}

\noindent{\bf Proof of Theorem~\ref{thm_construction_of_N_pair_OF}}
First, note that the matrices $\mathfrak{B}^{(1)}_{j}(\gamma)$ and $\widetilde{\mathfrak{B}^{(1)}_{j}}(\gamma)$ are of order $(s_j+1)\times d_j$ and $s_j \times d_j$, respectively. For each $i \in \{2,\dots,N+1\}$, we define the matrices 
(similar to~\eqref{matrices}) $\mathfrak{B}^{(i)}_j(\gamma)$ and $\widetilde{\mathfrak{B}^{(i)}_j}(\gamma)$, as follows:
$$\displaystyle \mathfrak{B}^{(i)}_j(\gamma):={{\begin{pmatrix}
				G^{(i)(m)}_{j+1}(\gamma+ v_{j,n})
	\end{pmatrix}}}_{\substack{
			0\leq m \leq Ns_j \\
			1\leq n \leq d_j
	}} \qquad \mbox{and} \qquad  \widetilde{\mathfrak{B}^{(i)}_j}(\gamma):={{\begin{pmatrix}
				G^{(i)(m)}_{j+1}(\gamma+ v_{j,n})
	\end{pmatrix}}}_{\substack{
			1\leq m \leq Ns_j \\
			1\leq n \leq d_j
	}}. $$ 
    Clearly, the matrices $\mathfrak{B}^{(i)}_j(\gamma)$ and $\widetilde{\mathfrak{B}^{(i)}_j}(\gamma)$ are of order 
 $(Ns_j+1)\times d_j$ and $Ns_j \times d_j$, respectively.
	To prove that, for each $i \in \{2,\dots,N+1\}$, the system $\gtii$ is a Parseval frame for $L^2(G)$, we use Theorem~\ref{thm_UCP}(ii). In order to use this, it is sufficient to show that for each  $j \in \mathcal{J}$ and $i \in \{2,3, \cdots, N+1\},$ we have 
	$$	({\mathfrak{B}^{(i)}_j}(\gamma))^{\ast} \mathfrak{B}^{(i)}_j(\gamma)=\frac{s(\Gamma_{j})}{s(\Gamma_{j+1})}I_{d_j}\,  \mbox{  for a.e. } \gamma \in \Omega_j.$$
	 For $i \in \{2,3,\cdots,N+1\}$, the matrix $\mathfrak{B}^{(i)}_j(\gamma)$  can be expressed as $$\mathfrak{B}^{(i)}_j(\gamma)=U^{(i)}_j(\gamma)\mathfrak{B}^{(1)}_j(\gamma), \mbox{ where } U^{(i)}_j(\gamma)=\begin{pmatrix}
		1 &0 &0  &\cdots &0\\
		0 &V_{1+(i-2)s_j}^{j}(\gamma) &V_{2+(i-2)s_j}^{j}(\gamma) &\cdots &V_{s_{j}+(i-2)s_j}^{j}(\gamma)
	\end{pmatrix}, $$ 	as  the entries of $V^{j}_{i}(\gamma)$ are $\Gamma_{j}^\perp$-periodic. Here, the matrix $U^{(i)}_j(\gamma)$ is of order $(Ns_j+1)\times (s_j+1)$.
	The columns  $V_{1}^{j}(\gamma),V_{2}^{j}(\gamma) \cdots, V_{Ns_j}^{j}(\gamma)$ are orthonormal implies $(U^{(i)}_j(\gamma))^{\ast} U^{(i)}_j(\gamma)= I_{s_j+1}.$ Thus,  we have
	\begin{align*}
		({\mathfrak{B}^{(i)}_j}(\gamma))^{\ast} \mathfrak{B}^{(i)}_j(\gamma)&=(\mathfrak{B}^{(1)}_j(\gamma))^{\ast}(U^{(i)}_j(\gamma))^{\ast} U^{(i)}_j(\gamma) \mathfrak{B}^{(1)}_j(\gamma)\\
		&=(\mathfrak{B}^{(1)}_j(\gamma))^{\ast}I_{s_j+1} \mathfrak{B}^{(1)}_j(\gamma)\\
		&=({\mathfrak{B}^{(1)}_j}(\gamma))^{\ast} \mathfrak{B}^{(1)}_j(\gamma)=\frac{s(\Gamma_{j})}{s(\Gamma_{j+1})}I_{d_j}
	\end{align*}
    using (\ref{PFC}). Therefore, by Theorem~\ref{thm_UCP}, for each $i \in \{2,3,\dots,N+1\}$, the GTI system  $\,\bigcup_{j \in \mathcal{J}}\{T_{\lambda}g_p^{(i)}\}_{\lambda \in \Gamma_{j}, p \in P_j^{'}}$ is a Parseval frame. 
	Next, for proving the systems   $\displaystyle\bigcup_{j \in \mathcal{J}}\{T_{\lambda}g_p^{(i)}\}_{\lambda \in \Gamma_{j}, p \in P_j^{'}}$ and  $\displaystyle\bigcup_{j \in \mathcal{J}}\{T_{\lambda}g_p^{(i')}\}_{\lambda \in \Gamma_{j}, p \in P_j^{'}}$ are  pairwise orthogonal, we use  Theorem \ref{thm_pairwise_orthogonal_frames}. For this, it suffices to show that, for each $j \in \mathcal{J}$ and $i,i' \in \{2,3,\dots,N+1\}$ with $i \neq i'$,  $(\widetilde{\mathfrak{B}^{(i)}_j})^{\ast}(\gamma) \widetilde{\mathfrak{B}^{(i')}_j}(\gamma)=O_{d_j}, \mbox{  a.e. } \gamma \in \Omega_j$. For $i,i' \in \{2,3,\dots,N+1\}$, we can express the matrices $\widetilde{\mathfrak{B}^{(i)}_j}(\gamma)$ and $\widetilde{\mathfrak{B}^{(i')}_j}(\gamma)$ as follows:
	$$\widetilde{\mathfrak{B}^{(i)}_j}(\gamma)=\begin{pmatrix}
		V_{1+(i-2)s_j}^{j}(\gamma) &V_{2+(i-2)s_j}^{j}(\gamma) &\cdots &V_{s_{j}+(i-2)s_j}^{j}(\gamma)
	\end{pmatrix} \widetilde{\mathfrak{B}^{(1)}_j}(\gamma) \mbox{ and }$$  $$\widetilde{\mathfrak{B}^{(i')}_j}(\gamma)=\begin{pmatrix}
		V_{1+(i'-2)s_j}^{j}(\gamma) &V_{2+(i'-2)s_j}^{j}(\gamma) &\cdots &V_{s_{j}+(i'-2)s_j}^{j}(\gamma)
	\end{pmatrix} \widetilde{\mathfrak{B}^{(1)}_j}(\gamma)$$
     using the periodicity of entries of $V^{j}(\gamma)$.
	Since $$\begin{pmatrix}
		V_{1+(i-2)s_j}^{j}(\gamma) &\cdots &V_{s_{j}+(i-2)s_j}^{j}(\gamma)
	\end{pmatrix}^{\ast}\begin{pmatrix}
		V_{1+(i'-2)s_j}^{j}(\gamma)  &\cdots &V_{s_{j}+(i'-2)s_j}^{j}(\gamma)
	\end{pmatrix}=O_{s_j},$$ it follows that 
	\begin{align*}
		(\widetilde{\mathfrak{B}^{(i)}_j})^{\ast}(\gamma) \widetilde{\mathfrak{B}^{(i')}_j}(\gamma)=& \left(\begin{pmatrix}
			V_{1+(i-2)s_j}^{j}(\gamma) &V_{2+(i-2)s_j}^{j}(\gamma) &\cdots &V_{s_{j}+(i-2)s_j}^{j}(\gamma)
		\end{pmatrix} \widetilde{\mathfrak{B}^{(1)}_j}(\gamma)\right)^\ast\\
		&\times \begin{pmatrix}
			V_{1+(i'-2)s_j}^{j}(\gamma) &V_{2+(i'-2)s_j}^{j}(\gamma) &\cdots &V_{s_{j}+(i'-2)s_j}^{j}(\gamma)
		\end{pmatrix} \widetilde{\mathfrak{B}^{(1)}_j}(\gamma)\\		
		=&(\widetilde{\mathfrak{B}^{(1)}_j})^{\ast}(\gamma) \begin{pmatrix}
			V_{1+(i-2)s_j}^{j}(\gamma) &V_{2+(i-2)s_j}^{j}(\gamma) &\cdots &V_{s_{j}+(i-2)s_j}^{j}(\gamma)
		\end{pmatrix}^{\ast}\\
		&\times \begin{pmatrix}
			V_{1+(i'-2)s_j}^{j}(\gamma) &V_{2+(i'-2)s_j}^{j}(\gamma) &\cdots &V_{s_{j}+(i'-2)s_j}^{j}(\gamma)
		\end{pmatrix}\widetilde{\mathfrak{B}^{(1)}_j} (\gamma)\\
		=&O_{d_j},
	\end{align*}
	Hence, the result follows.
\ep

	\section{Applications to Sampling theory}\label{sec 5}

    Building on Theorem~\ref{thm_new_characterization_orthogonal_frame}, we apply the orthogonality theory of generalized translation-invariant systems to sampling on LCA groups. More precisely, we extend Weber's sampling results for unions of lattices in $\mathbb{R}^d$ \cite{W} to sampling sets arising from co-compact subgroups of an arbitrary LCA group. By realizing sampling transforms as analysis operators associated with GTI systems, sampling properties are translated into frame-theoretic properties of these systems. This correspondence enables us to characterize orthogonality of sampling transforms through the orthogonality criteria for GTI frames obtained earlier. The resulting characterization recovers Weber's theorem in the Euclidean setting and provides a unified framework for orthogonal sampling in a significantly broader class of translation-invariant systems. This framework allows us to identify boundedness, sampling, and tight sampling properties with the corresponding Bessel, frame, and tight frame properties of the associated GTI systems.  This illustrates how orthogonal frame theory provides a unifying and powerful perspective for sampling in harmonic analysis and signal processing.	
	
	We begin with the \textit{translation-invariant subspace} $V_{E} := \{ f \in L^2(G) : \operatorname{supp}(\widehat{f}) \subset E \},$
	where $E \subset \widehat{G}$ is a band, i.e., a measurable subset of finite Haar measure. 
	It is immediate that $T_\lambda V_E = V_E$ for all $\lambda \in G$.
	Since $V_E$ is a reproducing kernel Hilbert space, every element $f \in V_E$ satisfies
	\begin{equation}\label{eq_sample_intro}
		f(\lambda) = \langle f, T_\lambda \psi \rangle, 
	\end{equation}
	where $\widehat{\psi}_E = \chi_E$, the indicator function of the set $E$.
	For $\eta_j \in \operatorname{Aut}(G)$, $j \in \J$, 
	we define the \emph{sampling transform} associated with the union of  samples $\x:=\dx$  by
	\begin{align*}
		\tx: V_E \to \bigoplus\limits_{j \in \J} L^2(\Gamma_j): f \mapsto \left(f(\eta_j \la)\right)_{j \in \mathcal{J}},
	\end{align*}
	provided $\tx$ is bounded.
	The operator $\tx$ is \textit{bounded} if there exists $B>0$ such that
	\begin{equation}\label{eq_theta_bounded}
		\sum_{j \in \J} \int_{\Gamma_j} 
		\big| \langle f, T_{\eta_j\la} \psi_{E} \rangle \big|^2 \, 
		d\mu_{\Gamma}(\la) 
		\;\leq\; B \|f\|^2, 
		\qquad f \in V_E.
	\end{equation}
	In particular, this holds if and only if the system 
	$\bigcup_{j \in \mathcal{J}} \{T_\la \psi_E \}_{ \la \in \eta_j\Gamma}$
	is a Bessel sequence for $V_E$. We say that the samples  $\dx$ form a \textit{set of sampling} for the band $E$ if $\tx$ 
	is bounded above and below, equivalently, if 
	$\bigcup_{j \in \mathcal{J}} \{T_\la \psi_E \}_{ \la \in \eta_j\Gamma}$ is a frame for $V_E$. 	The set of sampling  is called \emph{tight} with constant $K$ if $\|\tx f\|^2 = K \|f\|^2,\ f \in V_E.$ \textit{ We say that the sampling set satisfies the $1$-UCP if the system $\bigcup_{j \in \mathcal{J}} \{T_\la \psi_E \}_{\la \in \eta_j \Gamma}$ satisfies the $1$-UCP.}

	Throughout the remainder of this section we fix $\eta_{j}, \zeta_{j} \in \operatorname{Aut}(G)$ for each $j \in \J$. Suppose that the sampling sets \( \x=\dx\) and \( \y=\dy\) are given, with associated sampling transforms \(\tx\) and \(\ty\) that are bounded operators on bands \(E\) and \(F\), respectively. We say that the samples \(\dx\) and \(\dy\) are \textit{orthogonal} on bands \(E\) and \(F\) if and only if the image of \(\ve\) under \(\tx\) is orthogonal to the image of \(V_F\) under \(\ty\) in \(\ops\),
	which is equivalent to $\ty^* \tx = 0,$ that is,
	\begin{align*}
		\sum_{j \in \J} \int_{\Gamma_j} 
		\langle f, T_{\eta_j\la} \psi_{E} \rangle  T_{\zeta_j\la} \psi_{F} \, 
		d\mu_{\Gamma}(\la)=0 
		\qquad f \in V_E.
	\end{align*}
	This section studies when the sampling transforms 
$\tx$ and $\ty$
 are pairwise orthogonal, and when they form tight frames.

    The following result gives necessary and sufficient conditions under which two sampling sets are orthogonal; it follows as an application of Theorem~\ref{thm_new_characterization_orthogonal_frame}. 

    \bpr
\label{prop_ness_suff_conditions_for_pos_UCP}
		Let  \( \{\eta\la:\la \in \Gamma_j\}_{j \in \J}\) and  \(\{\zeta \la:\la \in \Gamma_j\}_{j \in \J}\) be sampling sets such that at least one of the sampling set satisfy the $1$-UCP. The following assertions are equivalent:
		\begin{itemize}
			\item[(i)]  The samples  \( \{\eta\la:\la \in \Gamma_j\}_{j \in \J}\) and  \(\{\zeta \la:\la \in \Gamma_j\}_{j \in \J}\)  are orthogonal.
			\item[(ii)]  For every  $\alpha \in \union$, the set 
$\mu_{\widehat{G}}\left((\eta^*)^{-1}E \cap \big((\zeta^*)^{-1}F+\alpha\big)\right)=0$.
		\end{itemize}
    \epr
 \bp
First, note that the samples \( \{\eta\la:\la \in \Gamma_j\}_{j \in \J}\) and  \(\{\zeta \la:\la \in \Gamma_j\}_{j \in \J}\) are orthogonal if and only if the GTI systems $\bigcup\limits_{j \in \J}\{T_{\la} \psi_{E}\}_{\la \in \eta\Gj}$ and $\bigcup\limits_{j \in \J}\{T_{\la} \psi_{F}\}_{\la \in \z\Gj}$ are pairwise orthogonal, where $\widehat{\psi}_E = \chi_E$ and $\widehat{\psi}_F = \chi_F$. By Theorem~\ref{thm_new_characterization_orthogonal_frame}, this is further equivalent to requiring that, for each $\al \in \umzero$,
\begin{align*}
\displaystyle\sum_{j \in \mathcal{J} : \alpha \in \Gamma^\perp_j} \frac{1}{s(\Gamma_{j})}
\overline{\chi_E(\eta^\ast\gamma)} \chi_{F}(\zeta^\ast(\gamma+\al)) = 0
\mbox{ for a.e. } \gamma \in \widehat{G},
\end{align*}
and
\begin{equation*}
\displaystyle\sum_{j \in \mathcal{J} }\frac{1}{s(\Gamma_{j})}
\overline{\chi_E(\eta^\ast\gamma)} \chi_{F}(\zeta^\ast\gamma) = 0
\mbox{ for a.e. } \gamma \in \widehat{G}.
\end{equation*}
A direct calculation shows that this is further equivalent to requiring that for each $\alpha \in \union$, the set 
$\mu_{\widehat{G}}\left((\eta^*)^{-1}E \cap \big((\zeta^*)^{-1}F+\alpha\big)\right)=0$. This completes the proof.
\ep
The following result shows that two unions of samples are orthogonal if and only if each corresponding pair of individual samples is orthogonal. This result generalizes \cite[Theorem~1]{W2004} in two respects: first, by extending the underlying space from $\mathbb{R}^d$ to a locally compact abelian (LCA) group, and second, by replacing the lattice $\mathbb{Z}^d$ with co-compact subgroups.

    \bt\label{thm_pairwise_orthogonal_samples} 
		Suppose that the sampling sets \( \dx\) and \( \dy\) satisfy the dual $1$-UCP.  These samples are orthogonal on the bands $E$ and $F$  if and only if for every $j \in \mathcal{J}$, for all $\alpha \in \Gamma_j$ we have $\mu_{\widehat{G}}\left((\eta^*_j)^{-1} E \cap \left( (\zeta^*_j)^{-1} F + \alpha \right) \right)=0.$
		Equivalently, the samples above are orthogonal if and only if the samples $\{\eta_j \la \}_{\la \in \Gamma_j} \mbox{ and } \{\zeta_j \la \}_{ \la \in \Gamma_j} $ are orthogonal on the bands $E$ and $F$ for each $j \in \J$.
    \et
  Although Theorem~\ref{thm_pairwise_orthogonal_samples} generalizes Proposition~5.1, its proof does not follow directly from Section~3 due to the variability of the maps $\eta_j$ and $\z_j$. The following lemma, based on bracket-map techniques, will be instrumental in the proof.
 To this end, we define the bracket function  
	$$[f,g](x,\Gamma)=\int\limits_\Gamma f(x+\la) \overline{g(x+\la)} d\mu_{\Gamma}(\la).$$
    \bl\label{lem_for_orthogonal_sample_proof}
	Let the  systems $\displaystyle\bigcup_{j \in \mathcal{J}}\{T_{\lambda}g_j^{(1)}\}_{\lambda \in \eta_j\Gamma}$ and $\displaystyle\bigcup_{j \in \mathcal{J}}\{T_{\lambda}g_j^{(2)}\}_{\lambda \in \zeta_j\Gamma}$ be Bessel systems, satisfying the dual $1$-UCP. Define $\Theta_{g^{(1)}, g^{(2)}}$ analogous to mixed dual gramian operator:
	$$\Theta_{g^{(1)}, g^{(2)}}:L^2(G) \to L^2(G): \ f \mapsto \sum_{j \in \mathcal{J}} \int\limits_{\Gamma} \langle f, T_{\eta_j \la} g_j^{(1)} \rangle T_{\zeta_j \la} g_j^{(2)} d \mu_{\Gamma}(\la).$$
	Then for all $f,g \in \mathcal{D}_B$, we have 
	$$\langle \Theta_{g^{(1)}, g^{(2)}}f, g \rangle=\sum_{j \in \mathcal{J}} \frac{1}{\Delta(\eta_j)\Delta(\zeta_j)}  \int\limits_{\widehat{G}/\G^\perp} \left[\widehat{f}, \widehat{g_j^{(1)}}\right] \left(\eta_j^\ast w, \eta_j\G^\perp\right) \left[ \widehat{g_j^{(2)}}, \widehat{g}\right] \left(\zeta_j^\ast w, \zeta_{j}\G^\perp\right) \ d \mu_{\widehat G /  \Gamma^\perp}(\dot{w}).$$
	\el
    \bp 
	We compute $\langle \Theta_{g^{(1)}, g^{(2)}} f, g \rangle$ explicitly in the Fourier domain. 
	Since both systems are Bessel and satisfy the dual-1 UCP, the operator $\Theta_{g^{(1)}, g^{(2)}}$ is well defined, and all sums and integrals below converge absolutely.
	 By definition,
	\bee
	\langle \Theta_{g^{(1)}, g^{(2)}}f,g \rangle
	&=&\sum_{j \in \mathcal{J}} \int\limits_{\Gamma} 
	\langle f, T_{\eta_j \la} g_j^{(1)} \rangle 
	\langle T_{\zeta_j \la} g_j^{(2)}, g \rangle 
	\, d\mu_{\Gamma}(\la) \nonumber\\
	&=&\sum_{j \in \mathcal{J}} \int\limits_{\Gamma} 
	\langle \widehat{f}, M_{-\eta_j \la} \widehat{g_j^{(1)}} \rangle \,
	\overline{\langle \widehat{g}, M_{-\zeta_j \la} \widehat{g_j^{(2)}} \rangle}\,
	d \mu_{\Gamma}(\la). \label{4Ceq:Omega-inner}
	\ene
	\medskip\noindent
	Consider the first factor in \eqref{4Ceq:Omega-inner}. By the substitution $\xi=(\eta_j')^{-1}w:=\eta_j^\ast w$ we obtain
	\[
	I_1 := \int\limits_{\widehat G} \widehat f(\xi)\,\overline{(\xi,\eta_j\la)}\,\overline{\widehat{g_j^{(1)}}(\xi)} \, d\mu_{\widehat G}(\xi)
	= \frac{1}{\Delta(\eta_j)} \int\limits_{\widehat G} \widehat f(\eta_j^* w)\,
	\overline{\widehat {g_j^{(1)}}(\eta_j^*w)}\,\overline{(w,\la)} \, d\mu_{\widehat G}(w).
	\]
	Applying Weil’s formula \eqref{eq_weil_integral_formula_2} gives
	\[
	I_1 = \frac{1}{\Delta(\eta_j)} 
	\int\limits_{\widehat G / \Gamma^\perp}
	\sum_{\al \in \Gamma^\perp} 
	\widehat f(\eta_j^*(w+\al))\,
	\overline{\widehat {g_j^{(1)}}(\eta_j^*(w+\al))}\,
	\overline{(w,\la)} \,
	d\mu_{\widehat G / \Gamma^\perp}(\dot{w}).
	\]
	Since $(\al,\la)=1$ for all $\al \in \Gamma^\perp$, the factor $(w+\al,\la)$ reduces to $(w,\la)$, where $\dot{w}$ denotes the coset in $\widehat G/\Gamma^\perp$. Hence
	\[
	I_1 = \frac{1}{\Delta(\eta_j)} 
	\int\limits_{\widehat G / \Gamma^\perp}
	[\widehat f,\widehat {g_j^{(1)}}](\eta_j^* w, \eta_j^* \Gamma^\perp)\,
	\overline{(w,\la)} \,
	d\mu_{\widehat G / \Gamma^\perp}(\dot{w}).
	\]
	A similar computation with $\xi=(\zeta_j^*)^{-1}w$ yields
	\[
	I_2 := \int\limits_{\widehat G} \widehat g(\xi)\,\overline{(\xi,\zeta_i \gamma)}\,\overline{\widehat {g_j^{(2)}}(\xi)} \, d\mu_{\widehat G}(\xi)
	= \frac{1}{\Delta(\zeta_j)} 
	\int\limits_{\widehat G / \Gamma^\perp}
	[\widehat g,\widehat {g_j^{(2)}}](\zeta_j^* w, \zeta_j^*\Gamma^\perp)\,
	\overline{(w,\gamma)} \,
	d\mu_{\widehat G / \Gamma^\perp}(\dot{w}).
	\]
	\medskip\noindent
	Substituting $I_1$ and $I_2$ into \eqref{4Ceq:Omega-inner} gives
	\bes
	\langle \Theta_{g^{(1)}, g^{(2)}}f,g \rangle
	= \sum_{j \in \mathcal{J}} \frac{1}{\Delta(\eta_j)\Delta(\zeta_j)}
	\int\limits_{\Gamma}
	\Bigg(
	\int\limits_{\widehat G / \Gamma^\perp}
	[\widehat f,\widehat {g_j^{(1)}}](\eta_j^*w,\eta_j^* \Gamma^\perp)\,
	\overline{(w,\gamma)} \,
	d\mu_{\widehat G / \Gamma^\perp}(\dot{w})
	\Bigg) \nonumber\\
	\qquad\qquad\qquad\times
	\overline{
		\Bigg(
		\int\limits_{\widehat G / \Gamma^\perp}
		[\widehat g,\widehat {g_j^{(2)}}](\zeta_j^* w,\zeta_j^*\Gamma^\perp)\,
		\overline{(w,\gamma)} \,
		d\mu_{\widehat G / \Gamma^\perp}(\dot{w})
		\Bigg)} \,
	d\mu_\Gamma(\la).
	\ens
	Define, for each $j\in\J$, the functions on the quotient $\widehat G/\Gamma^\perp$
	\[
	A_j(\dot{w}):=\big[\widehat f,\widehat g_j^{(1)}\big](\eta_j^* w,\eta_j^*\Gamma^\perp) \quad \mbox{ and } \quad
	B_j(\dot\xi):=\big[\widehat g,\widehat g_j^{(2)}\big](\zeta_j^* w, \zeta_j^*\Gamma^\perp).
	\]
	For each fixed $j$, the inner integrals in the displayed formula above can be written as
	\[
	F_j(\gamma)=\int\limits_{\widehat G/\Gamma^\perp} A_j(\dot\xi)\,\overline{(\xi,\gamma)}\,d\mu_{\widehat G/\Gamma^\perp}(\dot\xi)
	\quad \mbox{ and } \quad
	G_j(\gamma)=\int\limits_{\widehat G/\Gamma^\perp} B_j(\dot{w})\,\overline{(\xi,\gamma)}\,d\mu_{\widehat G/\Gamma^\perp}(\dot\xi),
	\]
	which are precisely the inverse Fourier transforms of $A_j$ and $B_j$ when we identify
	$\widehat\Gamma\cong\widehat G/\Gamma^\perp$. Hence $F_j,G_j\in L^2(\Gamma)$ and by Plancherel on~$\Gamma$ (or equivalently on~$\widehat\Gamma$), we have
	\[
	\int\limits_\Gamma F_j(\la)\,\overline{G_j(\la)}\,d\mu_\Gamma(\la)
	= \int\limits_{\widehat G/\Gamma^\perp} A_j(\dot{w})\,\overline{B_j(\dot{w})}\,d\mu_{\widehat G/\Gamma^\perp}(\dot{w}).
	\]
	Applying this identity to the earlier double integral collapses the integration over $\Gamma$ and yields
	\[
	\langle \Theta_{g^{(1)}, g^{(2)}}f,g \rangle
	= \sum_{j\in\J}\frac{1}{\Delta(\eta_j)\Delta(\zeta_j)}
	\int\limits_{\widehat G/\Gamma^\perp}
	\big[\widehat f,\widehat g_j^{(1)}\big](\eta_j^*w,\eta_j^*\Gamma^\perp)\,
	\overline{\big[\widehat g,\widehat g_j^{(2)}\big](\zeta_j^*w, \zeta_j^*\Gamma^\perp)}\,
	d\mu_{\widehat G/\Gamma^\perp}(\dot{w}),
	\]
	which is the desired formula.
	\ep

    \noindent{\bf Proof of Theorem~\ref{thm_pairwise_orthogonal_samples}.}
	The samples $\dx$ and $ \dy$  are bounded on the band $E$ and $F$.  Thus, the corresponding GTI systems $\uniontsi$ and $\uniontsbi$ are Bessel on subspaces $V_E$ and $V_F$, respectively.   Suppose the samples $\dx$ and $ \dy$ are orthogonal, equivalently the $\uniontsi$ and $\uniontsbi$ are pairwise orthogonal. Since the systems $\uniontsi$ and $\uniontsbi$ satisfy dual $1$-UCP, this is further equivalent to 
	\begin{align*}
		\langle \Theta f,g \rangle=0 \mbox{ for all } f \in \mathcal{D}_\mathcal{B}\cap V_E \mbox{ and } g \in \mathcal{D}_\mathcal{B}\cap V_F.
	\end{align*}
	By Lemma \ref{lem_for_orthogonal_sample_proof}, which is further equivalent to 
	\begin{equation}\label{eq_OS1}
		\sum_{j \in \J}\frac{1}{\Delta(\eta_j)\Delta(\zeta_j)}
		\int_{\widehat G/\Gamma_j^\perp}
		\big[\widehat f,\widehat g_j^{(1)}\big](\eta_j^*w,\eta_j^*\Gamma_j^\perp)\,
		\overline{\big[\widehat g,\widehat g_j^{(2)}\big](\zeta_j^*w,\zeta_j^*\Gamma_j^\perp)}\,
		d\mu_{\widehat G/\Gamma_j^\perp}(\dot w)=0.
	\end{equation}
	Define  $g_0 \in V_E$ and $f_0 \in V_F$ by $\widehat{g}_0=\chi_{E_{0}}$  and $\widehat{f}_0=\chi_{F_{0}}$.  
	Replacing $f$ and $g$ by $f_0$ and $g_0$ in (\ref{eq_OS1}) and $\widehat{g_j^{(2)}}=\chi_{F}$, $\widehat{g_j^{(1)}}=\chi_E$, we get
	\begin{equation}
		\sum_{j \in \J}\frac{1}{\Delta(\eta_j)\Delta(\zeta_j)}
		\int_{\widehat G/\Gamma_j^\perp}
		\big[\chi_{E_0},\chi_{E}\big](\eta_j^*w,\eta_{j}^\ast\Gamma_j^\perp)\,
		\overline{\big[\chi_{F},\chi_{F_{0}}\big](\zeta_j^*w, \zeta_j^*\Gamma_j^\perp)}\,
		d\mu_{\widehat G/\Gamma_j^\perp}(\dot w)=0.
	\end{equation}
	Therefore for each $j \in \J$, we have 
	\begin{align*}
		&	\int_{\widehat G/\Gamma_j^\perp}
		\big[\chi_{E_0},\chi_{E}\big](\eta_j^*w,\eta_{j}^\ast\Gamma_j^\perp)\,
		\overline{\big[\chi_{F},\chi_{F_{0}}\big](\zeta_j^*w,\zeta_{j}^\ast\Gamma_j^\perp)}\,
		d\mu_{\widehat G/\Gamma_j^\perp}(\dot w)=0\\
		\implies&\displaystyle\sum\limits_{ \al \in \Gamma_j^\perp} \chi_{E_0}(\eta_j^\ast (w+\al) ) \displaystyle\sum\limits_{ \beta \in \Gamma_j^\perp} \chi_{F_0}(\zeta_j^\ast (w+\beta))  =0.
	\end{align*}
	Since $E_0 \subset E$ and $F_{0} \subset F$ are arbitrary measurable sets, the above formula holds for $E$ and $F$, i.e. 
	\bee
	\displaystyle\sum\limits_{ \al \in \Gamma_j^\perp} \chi_{E}(\eta_j^\ast (w+\al) ) \displaystyle\sum\limits_{ \beta \in \Gamma_j^\perp} \chi_{F}(\zeta_j^\ast(w+\beta) ) =0 \mbox{ for a.e. } w \in \widehat{G}. \label{eq_jth_orthogonal_sample}
	\ene
	 This is further equivalent to saying that for every $j \in \mathcal{J}$, for all $\alpha \in \Gamma_j$ we have $$\mu_{\widehat{G}}\left((\eta^*_j)^{-1} E \cap \left( (\zeta^*_j)^{-1} F + \alpha \right) \right)=0.$$ Converse follows by reversing the steps.
	\ep
	The following result shows that a union of samples is tight if and only if each individual sample is tight.
	\bt\label{thm_tight_sample}
	The samples $\x=\{\eta_j \la: \la \in \G\}_{1 \leq j \leq n}$ are tight on band $E$ with constant $K$ if and only if, for each $j \in \{1,2,\cdots,n\}$, the samples $=\{\eta_j \la : \la \in \Gamma\}$ are tight on $E$ with constant $K_j$. In this case, we have $K_j=\frac{1}{s(\eta_{j} \Gamma)}$ and $K=\displaystyle\sum\limits_{j=1}^n K_j$. 
	\et
Before proving this theorem, we first present a corollary of Theorem~\ref{thm_char_Parseval_frames}.
	\bc\label{cor_tight_frame}
	Suppose $\{g_j\}_{j \in \mathcal{J}} \subset V_{E}$ and $\Theta_g$ be the analysis operator of the system $\displaystyle\cup_{j \in \mathcal{J}}\{T_{\lambda}g_j\}_{\lambda \in \eta_j\Gamma_j}$ satisfying the $1$-UCP. Then $\|\theta_g f\|^2=K\|f\|^2$ if and only if for all $\alpha \in \displaystyle\cup_{j \in \mathcal{J}} \eta_j^\ast \Gamma^\perp_j$, we have $$\displaystyle\sum_{j \in {\mathcal{J}} : \alpha \in \eta_j^\ast \Gamma^\perp_j} \frac{1}{s(\eta_j\Gamma_j)}   \overline{\widehat{g_{j}^{(1)}}(\gamma)} \widehat{g_{j}^{(1)}}(\gamma+\alpha) =K\delta_{\alpha,0}\chi_{E}(\gamma)  \mbox{ for a.e. } \gamma \in \widehat{G}.$$
	\ec
	\bp
	In Theorem \ref{thm_char_Parseval_frames}, we assume $P_j=\{j\}$ and $\Gamma_j=\alpha_{j}\Gamma_j$ for each $j \in \mathcal{J}$.  Also 
	supp($\widehat{g_j})\subset  E$, thus now the result directly follows from Theorem \ref{thm_char_Parseval_frames}.
	\ep
	Now we are ready to prove our result.
	
	\noindent{\bf Proof of Theorem \ref{thm_tight_sample}\ :\ } 
     Let $\mathcal{J}=\{1,2,\cdots,n\}$. Observe that the $1$-UCP automatically holds for the system $\uniontsi$. Suppose first that the samples are tight, that is,
	\[
	\|\tx f\|^2 = K\|f\|^2
	\]  
  By Corollary~\ref{cor_tight_frame}, this holds if and only if,  for each $\alpha \in \bigcup_{j \in \mathcal{J}} \eta_j^\ast \Gamma^\perp$,  
	\bee \label{eq_union_Tight_frame}
	K\delta_{\alpha,0}\chi_{E}(\g) 
	&=& \sum_{j \in \mathcal{J}: \,\alpha \in \eta_j^\ast \Gamma^\perp} \tfrac{1}{s(\eta_j\Gamma)} \overline{\widehat{g_{j}^{(1)}}(\gamma)} \widehat{g_{j}^{(1)}}(\gamma+\alpha) \quad \text{a.e. } \gamma \in \widehat{G} \nonumber\\
	&=& \sum_{j \in \mathcal{J}: \,\alpha \in \eta_j^\ast \Gamma^\perp} \tfrac{1}{s(\eta_j\Gamma)} \overline{\widehat{\psi_{E}}(\gamma)} \widehat{\psi_{E}}(\gamma+\alpha) \quad \text{a.e. } \gamma \in \widehat{G} \nonumber\\
	&=& \sum_{j \in \mathcal{J}: \,\alpha \in \eta_j^\ast \Gamma^\perp} \tfrac{1}{s(\eta_j\Gamma)} \chi_{E}(\gamma)\chi_{E}(\gamma+\alpha) \quad \text{a.e. } \gamma \in \widehat{G}.
	\ene  
	Thus, \eqref{eq_union_Tight_frame} holds precisely when, for each $j \in \{1,2,\dots,n\}$ and all $\alpha \in \eta_j^\ast \Gamma^\perp$,  
	\begin{equation}
		K_j \delta_{\alpha,0}\chi_{E}(\g) = \tfrac{1}{s(\alpha_j\Gamma)} \chi_{E}(\gamma)\chi_{E}(\gamma+\alpha) \quad \text{a.e. } \gamma \in \widehat{G},
	\end{equation}  
	where $K_j=\tfrac{1}{s(\eta_j\Gamma)}$. Equivalently, for each $j \in \{1,2,\dots,n\}$, the samples $\{\eta_j \gamma\}_{\gamma \in \Gamma}$ form a tight frame on $E$ with constant $K_j$.  This completes the proof.
	\ep

        	\section*{Funding}
		The research of N. Redhu was supported by a research grant from CSIR, New Delhi [09/1022(0099)/2020-EMR-I]. N. K. Shukla's research was supported by the DST-SERB Project [MTR/2022/000176].	A. Gumber was supported by the Austrian Science Fund (FWF)[10.55776/ESP649]

					{\bf \vspace{.1in} \noindent Navneet Redhu\\
							Department of Mathematics\\
							Indian Institute of Technology Indore\\
							Simrol, Khandwa Road,
							Indore-453 552, Madhya Pradesh, India \\
							Email: navneetredhu.iiti@gmail.com}

						{\bf \vspace{.1in} \noindent Anupam Gumber\\
							Faculty of Mathematics\\
							 University of Vienna\\
						 Oskar-Morgenstern-Platz 1, 1090, Vienna, Austria\\
							 Email: anupam.gumber@univie.ac.at
                             
						 {\bf \vspace{.1in} \noindent Hartmut F\"uhr\\
							Lehrstuhl f\"ur Geometrie und Analysis\\
							 RWTH Aachen University\\
						 Room 50, Kreuzherrenstrasse 2, Aachen, Germany\\
							 Email: fuehr@mathga.rwth-aachen.de
                             
						{\bf \vspace{.1in} \noindent Niraj K. Shukla\\
								Department of Mathematics\\
								Indian Institute of Technology Indore\\
								Simrol, Khandwa Road,
								Indore-453 552, Madhya Pradesh, India \\
								Email: nirajshukla@iiti.ac.in} 
				
				\end{document}